\documentclass[11pt, a4paper, reqno]{amsart}

\usepackage[dvipsnames]{xcolor}
\usepackage[utf8]{inputenc}
\usepackage[english]{babel}
\usepackage{amsmath}        
\usepackage{amsfonts}       
\usepackage{amsthm}         
\usepackage{soul}
\usepackage{comment}

\usepackage{outlines}
\usepackage{enumitem}
\usepackage[colorlinks]{hyperref}
\usepackage{tikz-cd}
\usepackage{wrapfig}
\usepackage{subcaption}
\usepackage{mathtools}
\usepackage{amssymb}
\usepackage{todonotes}
\usepackage[nameinlink]{cleveref}
\crefname{subsection}{subsection}{subsections}

\usepackage{geometry}
\hypersetup{linkcolor=Cerulean, citecolor=Orange, urlcolor=DarkOrchid}

\theoremstyle{plain}
\newtheorem{thm}{Theorem}[subsection]
\theoremstyle{definition}
\newtheorem{defn}[thm]{Definition}
\newtheorem{proposition}[thm]{Proposition}
\newtheorem{corollary}[thm]{Corollary}
\newtheorem{lemma}[thm]{Lemma}
\newtheorem{rem}[thm]{Remark}
\newtheorem{example}[thm]{Example}

\crefname{defn}{definition}{definitions}
\Crefname{defn}{Definition}{Definitions}
\crefname{proposition}{proposition}{propositions}
\Crefname{proposition}{Proposition}{Propositions}
\crefname{corollary}{corollary}{corollaries}
\Crefname{corollary}{Corollary}{Corollaries}
\crefname{lemma}{lemma}{lemmas}
\Crefname{lemma}{Lemma}{Lemmas}
\crefname{rem}{remark}{remarks}
\Crefname{rem}{Remark}{Remarks}
\crefname{example}{example}{examples}
\Crefname{example}{Example}{Examples}
\crefname{notation}{notation}{notations}
\Crefname{notation}{Notation}{Notations}
\crefname{question}{question}{questions}
\Crefname{question}{Question}{Questions}

\usepackage{cite}  

\newcommand{\im}{\operatorname{Im}}

\numberwithin{equation}{section}

\usepackage{xargs}
\usepackage{xifthen}
\newcommand{\ifem}[3]{\ifthenelse{\isempty{#1}}{#2}{#3}}  
\newcommandx{\roland}[3][1=, 2=]{
	\ifem{#1}{
		\todo[color=green!50]{#3}
	}{
	\ifem{#2}{
		\todo[color=green!50, #1, caption={}]{#3}
	}{
		\todo[color=green!50, #1, caption={#2}]{ #3}
}}}

\usepackage{xpatch}
\makeatletter   
\xpatchcmd{\@tocline}
{\hfil\hbox to\@pnumwidth{\@tocpagenum{#7}}\par}
{\ifnum#1<0\hfill\else\dotfill\fi\hbox to\@pnumwidth{\@tocpagenum{#7}}\par}
{}{}
\makeatother    

\makeatletter
 \def\l@subsection{\@tocline{2}{0pt}{1.15cm}{1cm}{}}
\def\l@subsubsection{\@tocline{3}{0pt}{8pc}{8pc}{}}
 \makeatother

\newcommand{\nocontentsline}[3]{}
\newcommand{\tocless}[2]{\bgroup\let\addcontentsline=\nocontentsline#1{#2}\egroup}

\makeatletter
\providecommand\@dotsep{5}
\def\listtodoname{List of Todos}
\def\listoftodos{\@starttoc{tdo}\listtodoname}
\makeatother

\date{}
\title{Toric separable geometries and extremal Kähler metrics}
\author{Roland Púček}
\address{Roland Púček, Fakultät für Mathematik, Ruhr-Uni­ver­si­tät Bo­chum, Uni­ver­si­täts­stra­ße 150, D-44780 Bo­chum, Germany}
\email{\href{mailto:Roland.Pucek@ruhr-uni-bochum.de}{Roland.Pucek@ruhr-uni-bochum.de}}

\apptocmd{\thebibliography}{\fontsize{11}{15}\selectfont}{}{}

\begin{document}

\maketitle

\begin{abstract}
This paper introduces the framework of (local) toric separable geometries, where toric separable Kähler geometries come in families, each uniquely determined by an underlying factorization structure.
This unifying framework captures all known explicit Calabi-extremal toric Kähler metrics, previously constructed through diverse methods, as two distinct families corresponding to the simplest factorization structures: the product Segre and the Veronese factorization structure.
Crucially, the moduli of typical factorization structures has a positive dimension, revealing an immensely rich and previously untapped landscape of toric separable geometries.
The scalar curvature of toric separable geometries is computed explicitly, necessary conditions for the PDE governing extremality are derived, and new extremal metrics are obtained systematically.
In particular, for a $2m$-dimensional toric separable geometry, solutions of the PDE are necessarily $m$-tuples of rational functions of one variable belonging to an at most $(m+2)$-dimensional real vector space and whose denominators are determined by the factorization structure.
Toric separable geometries serve as a separation of variables technique and are well-suited for the analytic study of geometric PDEs.
\end{abstract}
\vspace{1cm}

\thispagestyle{empty}
The study of extremal Kähler metrics occupies a central position in complex differential geometry due to their deep connections with stability conditions in algebraic geometry and their foundational role in geometric analysis. 
Among the many classes of Kähler metrics, toric and related geometries have drawn significant attention because their rich symmetries allow explicit descriptions and classifications.

The main contribution of this paper is the introduction and development of the framework of toric separable geometries.
This framework unifies all known explicit extremal toric Kähler metrics originating from several previously distinct constructions and yields new explicit examples, some presented herein.
Together with \cite{pucekfs} and \cite{brandenburg2024veronese}, it lays a foundation for future studies of extremality, compactifications, fibrations, weighted curvatures and solitons, and stability tests, among others. \\

One of the original and most accessible definitions of an extremal Kähler metric characterises it as a critical point of the Calabi functional, which assigns to each metric in a fixed Kähler class the $L^2$-norm of its scalar curvature \cite{calabi1}: the Euler-Lagrange equation is a fourth order non-linear PDE.
Despite its simple formulation, the concept of extremal metrics has proven difficult to fully understand, leading to extensive research and significant insights.

Famously, the existence problem is conjecturally equivalent to a stability of Kodaira embeddings \cite{Mabuchi2021}.
The equivalence is known to hold between the Kähler-Einstein metrics on Fano manifolds and K-polystability for polarisation by the anti-canonical line bundle \cite{Chen2014,Chen2014a,Chen2014b}, and between constant scalar curvature Kähler (cscK) metrics on toric surfaces and K-polystability for arbitrary polarisations \cite{Donaldson2005,Donaldson2008b,Donaldson2009}.

Assuming such an equivalence holds true for extremal metrics and some form of K-stability, verifying either side of the correspondence for a given manifold remains a significant challenge.
Therefore, spaces where extremality and stability prove accessible attract particular interest.
These come as ruled \cite{Abreu2012,Abreu1997,apostolov2018weighted,AST_2008__322__93_0,Boyer2022,calabi1,calabi_extremal_1985,Lu2014,Simanca1991,Simanca1992,apostolov_cr_2020,apostolov2011extremal,apostolov_hamiltonian_2004,ToennesenFriedman2002,ToeNNESENFRIEDMAN1998}, toric \cite{Abreu1997,abreu2001kahler,Abreu2009a,Abreu2006,apostolov_cr_2020,Apostolov2023,Apostolov2015a,apostolov2003geometry,apostolov2017levi,apostolov2015ambitoric,apostolov_hamiltonian_2006,Apostolov2017b,Boyer2024,Boyer2018,bryant2001bochner,Feng2024,LeBrun2010,legendre2016toric}, blow-ups \cite{Arezzo2004,Arezzo2005,Arezzo2007,Seyyedali2016,Simanca1991,Tipler2012,Tipler2011}, and as spaces that do not fit into any of these categories \cite{Guan2000,Rollin2009,Stenzel1993}.

These references show that literature on 4-dimensional geometry is extensive, driven by connections to physics, additional symmetries unique to the low dimension, and its role as the first non-trivial dimension for studying extremal metrics.
While Kähler-Einstein and cscK metrics are relatively well-understood examples of extremal metrics in any dimension, explicit extremal metrics that are not cscK remain scarce, mostly arising from ruled manifolds,
generalising the original extremal non-cscK examples, being $\mathbb{P}(\mathcal{O}(-m) \oplus \mathbb{C})$ over $\mathbb{CP}^{n-1}$ for $m,n \in \mathbb{N}$ \cite{calabi1, Abreu2012} (see also \cite{apostolov2011extremal}).
Beyond ruled spaces, there are only a handful of explicit examples arising from blow-ups, and even fewer examples constructed by methods other than ruled, toric or blow-up.

To the best of current knowledge and after an extensive literature survey, all explicit toric extremal Kähler metrics fall into two classes: those generalised through constructions in \cite{apostolov2016ambitoric,apostolov2015ambitoric,apostolov_hamiltonian_2006,apostolov_hamiltonian_2004,apostolov_cr_2020,apostolov2017levi}, and Bochner-flat metrics on weighted projective spaces (e.g., \cite{bryant2001bochner,abreu2001kahler}).
Coordinate expressions of Kähler structures in these works reveal that all local examples are toric separable geometries as defined in this paper, with the compact ones appearing as compactifications thereof, except for the Bochner-flat metric on $\mathbb{CP}^m_{a_0,\ldots,a_m}$, where some weights coincide.
Compactifications of toric separable geometries will be addressed in future work.\\

We outline the evolution of the idea of separable toric geometries, which were alluded to in \cite{apostolov_cr_2020} as geometries admitting separable coordinates $x_1,\ldots,x_m$ in which the metric is determined by $m$ functions of one variable and some explicit data.

The origins trace back to the study of weakly self-dual Kähler surfaces \cite{apostolov2003geometry}, where hamiltonian 2-forms and 4-dimensional orthotoric geometries were discovered, the latter being one of the first prototypes of separable toric geometries in the above sense.
A comprehensive study of hamiltonian 2-forms in higher dimensions followed in \cite{apostolov_hamiltonian_2006,apostolov_hamiltonian_2004,apostolov_hamiltonian_2008,apostolov_hamiltonian_2008-1}, where the Bochner-flat metric on $\mathbb{CP}^m_{a_0,\ldots,a_m}$, all $a_j$ distinct, was shown to be a compactification of an $m$-dimensional orthotoric metric.
Then, \cite{apostolov2016ambitoric} locally classified ambitoric geometries, i.e., 4-dimensional orbifolds carrying two oppositely oriented and conformal toric Kähler structures, as ambitoric product, Calabi type and regular, where the latter is a generic case consisting of 4-dimensional orthotoric geometries and their twists.
In addition to the $m$-dimensional orthotoric metric, this classification provided further examples of separable toric geometries, while the geometrical characterisation beyond that their compactifications correspond to either equipoised trapezium (see also \cite{legendre2011toric}) or temperate quadrilateral remained to be explored.
The local ambitoric classification lead to the complete resolution of the extremality problem on compact 4-orbifolds with the second Betti number two: \cite{apostolov2015ambitoric} established equivalences between toric extremal metrics, extremal ambitoric geometries, and analytical relative K-polystability with respect to toric degenerations.

The proof of these equivalences essentially uses (both) 2-dimensional factorization structures, originally defined and classified in \cite{apostolov2015ambitoric} as Segre and Veronese factorization structures.
Factorization structures of dimension $m$ were comprehensively studied recently in \cite{pucekfs} as a basis for work on separable geometries, and applied to the geometry of polytopes \cite{brandenburg2024veronese}, where they revealed surprising connections between classes of polytopes that were previously considered fundamentally different.
To keep this introduction accessible and hence free of terminology specific to factorization structures, it is informative to conceptualise a factorization structure as a finite collection of projective algebraic curves assembled in a common projective space.
Specifically, Segre and Veronese factorization structures of dimension 2 correspond to two lines and a quadric in a projective plane, respectively.

Later, \cite{apostolov2017levi} showed that ambitoric product, Calabi type and negative orthotoric geometries are Levi-Kähler quotients of the CR manifold $\mathbb{S}^3 \times \mathbb{S}^3 \subset \mathbb{C}^2 \times \mathbb{C}^2$, thereby providing significant geometrical context for these separable toric geometries.
At last, using the theory of CR twists, \cite{apostolov_cr_2020} interpreted these quotients and positive regular ambitoric geometries, and thus all ambitoric geometries, as two distinct families of quotients/CR twists, generalised them to arbitrary dimension $2m$ as the CR twisted toric product ansatz and the CR twisted orthotoric ansatz, and established extremality of these.
Apart from the Bochner-flat metric on $\mathbb{CP}^m_{a_0,\ldots,a_m}$ with repeated weights, these two families encompass and generalise all known explicit extremal toric geometries.

Observing in these two families that the CR structures separate variables (see \eqref{J initial}) and are determined by the $m$-dimensional Segre and Veronese factorization structures, former being $m$ lines collectively intersecting at a unique point, the latter given by the rational normal curve of degree $m$, one arrives naturally at the definition of toric separable geometries as formulated in this paper: a toric separable CR geometry is such that the CR structure is determined by a factorization structure, and a toric separable Kähler geometry is any of its Sasaki-Reeb quotients.
This generalises the two Ansätze into a huge class of geometries corresponding to countless factorization structures.\\

Segre and Veronese factorization structures classify factorization structures in dimension 2; however, in higher dimensions, they represent only the simplest examples in the vast universe of factorization structures.
A more general class of factorization structures, called \textit{Segre-Veronese}, can be described as $k$ rational normal curves of degrees $d_j$, $d_j\geq 1$, $j=1,\ldots,k$, embedded in a common $(d_1+\cdots+d_k)$-dimensional projective space, with each degree $d_j$ curve formally corresponding to $d_j$ identical curves.
Even when considering just lines, one observes that there are several ways how they can sit in $k$-dimensional projective space, e.g., by analysing cardinalities of intersections.
An important example relevant for this work is the \textit{product Segre-Veronese} factorization structure whose rational normal curves intersect mutually at a single unique point:
when all curves are lines, it recovers the \textit{Segre} factorization structure, while when there is only one curve, it specialises to the \textit{Veronese} factorization structure.

Factorization structures, however, encode far more than intersection cardinalities.
For instance, a line and a cubic in 4-dimensional projective space give rise to infinitely many non-isomorphic factorization structures, corresponding to the orbits in the third symmetric power $S^3W$ of a 2-dimensional vector space $W$ under the diagonal action of $GL(W)$.
This situation is typical in the general setting: the number of non-equivalent factorization structures is vast.
Apart from a few special cases, the moduli spaces have positive dimension, offering an overwhelming amount of toric separable geometries.\\

An essential point of view on a factorization structure is through separable coordinates it induces.
A factorization structure, viewed as $m$ curves in an $m$-dimensional projective space $\mathbb{P}(\mathfrak{h})$, provides local coordinates on $\mathbb{P}(\mathfrak{h}^*)$ by $m$ 1-parametric families of hyperplanes whose normal directions lie on the curves: a curve of degree $d_j$ carries $d_j$ families, or, when counting the curves formally as above, one has one family per curve.

For example, \Cref{2b} displays the 2-dimensional Segre factorization structure in $\mathbb{P}^2$ given by lines $\psi_1$ and $\psi_2$, each containing a distinguished interval $[a_j,b_j]$, and the intersection point $Q$.
Dually, \Cref{2a} shows a region of the dual $\mathbb{P}^2$, bounded by thick red and blue segments, which can be \textit{factored} into the product of intervals via separable coordinates, i.e., the region is 1:1 with $[a_1,b_1] \times [a_2, b_2]$.\\

\begin{figure}[h!]
\centering
\begin{subfigure}{0.49\textwidth}
	\centering
	\begin{tikzpicture}
	\node at (0,0) {\includegraphics[width=9cm]{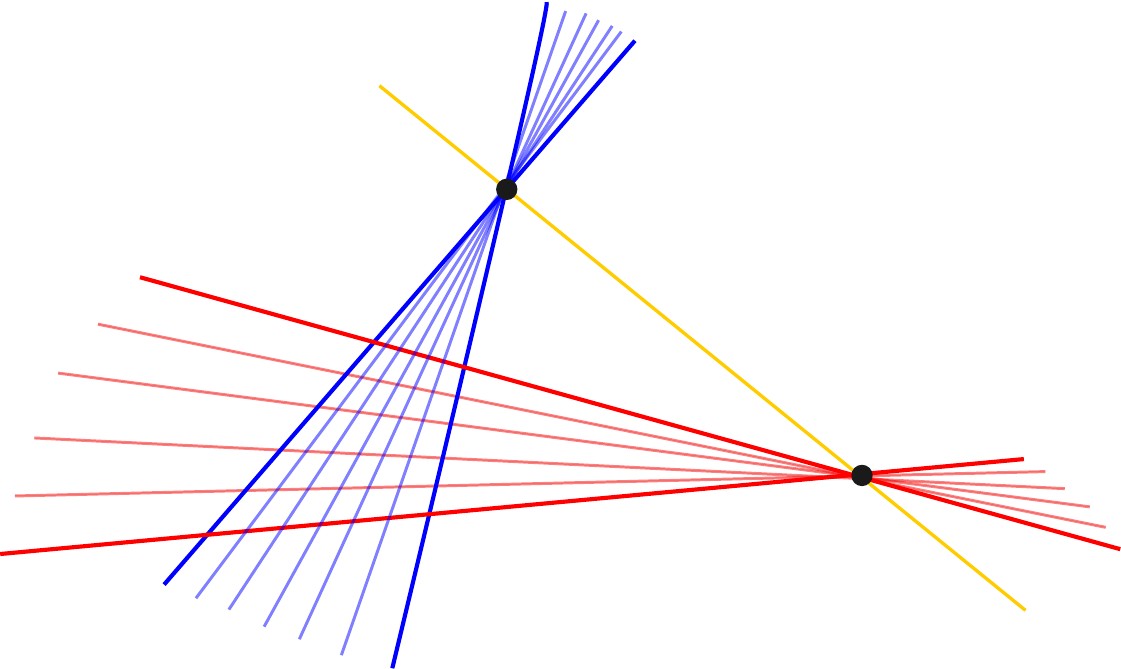}};
	\node at (1.5,.5) {$Q^0$};
	\node at (0,3) {$(a_2)^0$};
	\node at (1,2.6) {$(b_2)^0$};
	\node at (-1.4,1) {$(\psi_2)^0$};
	\node [rotate=69] at (2,-2) {$(\psi_1)^0$};
	\node at (-3.7,.7) {$(b_1)^0$};
	\node at (-4.1,-2.1) {$(a_1)^0$};
	\end{tikzpicture}
	\caption{separable coordinates}
	\label{2a}
\end{subfigure}
\begin{subfigure}{0.49\textwidth}
\hspace{1cm}
	\begin{tikzpicture}
	\node at (0,0) {\includegraphics[width=5.5cm]{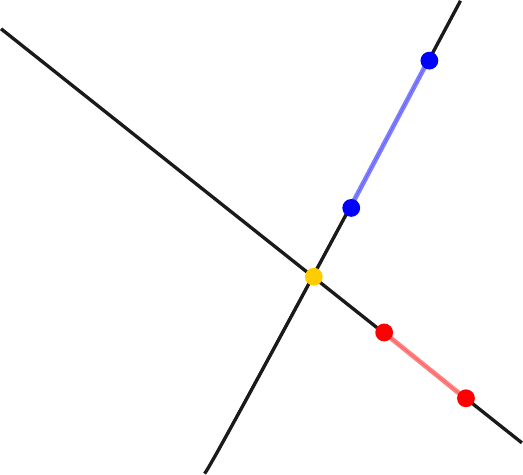}};
	\node at (2.8,2.5) {$\psi_2$};
	\node at (1.4,2) {$b_2$};
	\node at (.6,.5) {$a_2$};
	\node at (-2.1,2.5) {$\psi_1$};
	\node at (1.2,-1.4) {$a_1$};
	\node at (2,-2.1) {$b_1$};
	\node at (0,-.5) {$Q$};
	\end{tikzpicture}
\hspace{2cm}
\caption{Segre factorization structure}
\label{2b}
\end{subfigure}
\caption{} \label{fig2}
\end{figure}

Analogously, \Cref{3b} shows the 2-dimensional Veronese factorization structure given by a quadric in $\mathbb{P}^2$, which carries two intervals $[a_1,a_2]$ and $[a_3,a_4]$ in the obvious sense.
Dually, \Cref{3a} displays separable coordinates, which due to the low dimension arise via tangents to the dual conic, and a region in bijection with the product of intervals $[a_1,a_2] \times [a_3,a_4]$.

\begin{figure}[h!]
\centering
\begin{subfigure}{0.49\textwidth}
	\centering
	\begin{tikzpicture}
	\node at (0,0) {\includegraphics[width=7cm]{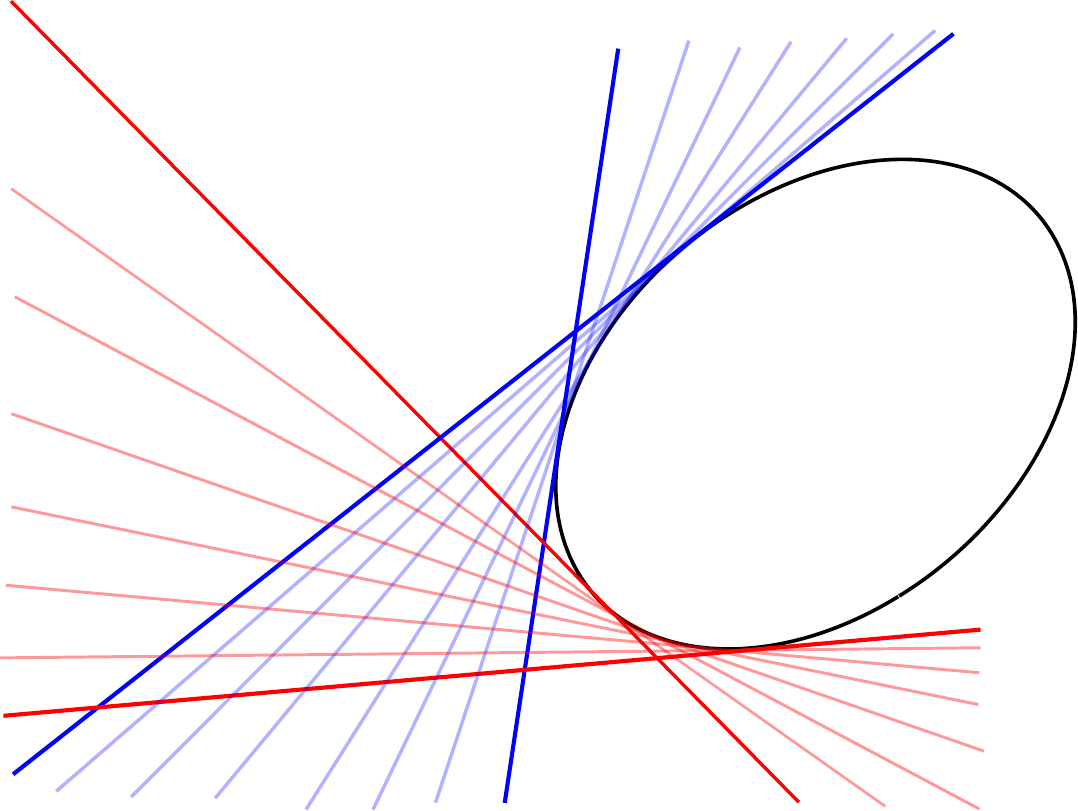}};
	\node at (3.8,1.8) {$\psi^*$};
	\node at (3.2,2.8) {$(a_1)^0$};
	\node at (.6,2.8) {$(a_2)^0$};
	\node at (-4.1,2.8) {$(a_3)^0$};
	\node at (-4.1,-2) {$(a_4)^0$};
	\end{tikzpicture}
	\caption{separable coordinates}
	\label{3a}
\end{subfigure}
\begin{subfigure}{0.49\textwidth}
\hspace{2cm}
	\begin{tikzpicture}
	\node at (0,0) {\includegraphics[width=3cm]{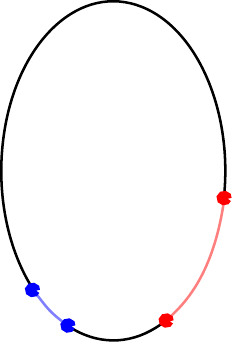}};
	\node at (1.5,2) {$\psi$};
	\node at (1.8,-.5) {$a_4$};
	\node at (1,-2.2) {$a_3$};
	\node at (-.6,-2.4) {$a_2$};
	\node at (-1.5,-1.6) {$a_1$};
	\end{tikzpicture}
\caption{Veronese factorization structure}
\label{3b}
\end{subfigure}
\caption{}
\end{figure}

The underlying idea of a toric separable geometry is that the image of its momentum map lies in a region admitting separable coordinates.
For more details see \cite{pucekfs}.\\

This paper defines and studies toric separable contact, CR and Kähler geometries associated with a factorization structure, focusing primarily on the latter.
For any $2m$-dimensional separable Kähler geometry \eqref{geometry tensors}, we compute the Laplace operator acting on invariant functions \eqref{Laplace separable}, the Ricci potential, symplectic and holomorphic volume forms, and the scalar curvature (see \Cref{scalar curvature}).
This enables us to express the Euler-Lagrange equation of the Calabi functional, known as the extremality equation, in terms internal to separable geometries and factorization structures (see \Cref{extremality equation}).

We prove in \Cref{solutions 1} that its solutions are $m$-tuples of rational functions of one variable, where the degree of the numerator is bounded by $m+2$ and the denominator is given by an explicit polynomial determined by the factorization structure.
Refinements of this bound appear in \Cref{deccor} and \Cref{final piece}, and additional necessary conditions are derived in the form of ODEs in \eqref{diagonal extr eq} and \Cref{LI - solutions}.
For a more specific, yet broad, class of factorization structures, trivially including Segre, Veronese and the product Segre-Veronese factorization structures, we solve these necessary conditions and summarise the results in \Cref{main 1}.

This reduces the problem to verifying which rational functions, subject to the these constrains and depending on at most $m+2$ real parameters, satisfy the extremality equation.
While the overall strategy is straightforward, the verification requires some care; for this reason \Cref{appendix1} includes various summation formulae essential to this work.
In \Cref{cs}, we recover known extremal CR twists of Riemann surfaces and orthotoric geometry, casting these results within the unified framework of separable geometries.

Moreover, in \Cref{main pSVfs}, we identify all extremal toric separable geometries corresponding to the product Segre-Veronese factorization structure, thus providing new examples of extremal Kähler metrics.
The paper concludes with an extremality analysis of a separable geometry associated with a factorization structure that is the first non-trivial example featuring decomposable defining tensors; by contrast, the Segre, Veronese, and product Segre-Veronese structures are trivial in this regard. 
This example demonstrates that the extremality equation remains tractable beyond the product Segre-Veronese case and suggests the possibility of a uniform treatment for solutions within the class of decomposable factorization structures. \\

This paper paves the way for future research in toric separable geometries, which naturally lend themselves to the study of geometric partial differential equations and can be viewed as employing a separation of variables technique.

The motivation for their study is strengthened by the fact that constructions previously limited to orthotoric or ambitoric geometries can now be extended to a much broader class.
A deeper understanding of extremality, particularly for decomposable and general Segre-Veronese factorization structures, remains an important direction for further investigation.
Further potential avenues well-suited for future work in separable geometries include weighted extremality and Sasaki geometry, solitons and their generalisations, geometric flows, compactifications, fibrations, and stability analyses.

\addtocontents{toc}{\protect{\pdfbookmark[1]{\contentsname}{toc}}}
\renewcommand\contentsname{\vspace{.5cm} Contents}

{\hypersetup{hidelinks}
\tableofcontents}

\section{Background}

\subsection{Factorization structures} \label{fs subsection}
The full account of the theory of factorization structures can be found in \cite{pucekfs}.
Results of \cite{pucekfs} will be used freely throughout without further explicit citation.

Let $V_1,\ldots,V_m$, $m\geq2$, be 2-dimensional vector spaces.
For $j \in \{1,\ldots,m\}$ and $\ell \subset V_j$ a 1-dimensional subspace, we define
\begin{align}
V=\bigotimes_{r=1}^m V_r
\hspace{1cm}
\text{and}
\hspace{1cm}
\Sigma_{j,\ell}=
V_1\otimes\cdots\otimes V_{j-1}\otimes\ell\otimes V_{j+1}\otimes\cdots\otimes V_m,
\end{align}
and denote the dual of $V$ by $V^*$ and the annihilator of $\Sigma_{j,\ell}$ in $V^*$ by $\Sigma_{j,\ell}^0$. .\par
The projective space $\mathbb{P}(W)$ is viewed as the set of 1-dimensional subspaces in the vector space $W$ equipped with the Zariski topology. Often, we identify $\ell\in\mathbb{P}(W)$ with the corresponding 1-dimensional subspace of $W$, and denote the span of a non-zero vector $w\in W$ by $\langle w \rangle$. We say a condition holds for a \textit{generic} point or \textit{generically} if there exists an open non-empty subset $U\subset\mathbb{P}(W)$ such that the condition holds at each point of $U$.

\begin{defn}\label[defn]{fs def}
Let $m$ be a positive integer. An injective linear map $\varphi:\mathfrak{h}\to V^*$ of an $(m+1)$-dimensional vector space $\mathfrak{h}$ into $V^*$ is called a \textit{factorization structure} of dimension $m$ if
\begin{align}\label{wfs def condition}
\dim \left( \varphi(\mathfrak{h}) \cap \Sigma_{j,\ell}^0 \right) = 1
\end{align}
holds for every $j\in\{1,\ldots,m\}$ and generic $\ell\in\mathbb{P}(V_j)$.
An isomorphism of factorization structures is the commutative diagram
\begin{center}
\begin{tikzcd}
	\mathfrak{h}_1 \arrow[d, "\varphi_1"'] \arrow[rr, "\Phi"]                      &  & \mathfrak{h}_2 \arrow[d, "\varphi_2"] \\
	V_1^*\otimes\cdots\otimes V_m^* \arrow[rr, "(\phi_1\otimes\cdots\otimes\phi_m)\sigma"] &  & W_1^*\otimes\cdots\otimes W_m^*      
\end{tikzcd},
\end{center}
where $\Phi$ and $\phi_j:V_{\sigma(j)}^*\to W_j^*$ are linear isomorphisms for all $j\in\{1,\ldots,m\}$, and $\sigma$ is a permutation of $\{1,\ldots,m\}$ viewed as the braiding map
$V_1^*\otimes\cdots\otimes V_m^*\to V_{\sigma(1)}^*\otimes\cdots\otimes V_{\sigma(m)}^*$.
\end{defn}

\begin{rem}\label[rem]{the remark}
Setting $\sigma=\text{id}$ and $\phi_j=\text{id}$, $j=1,\ldots,m$, shows that any two factorization structures with the same images are indistinguishable up to a choice of $\Phi$, which does not play a role in the defining condition \eqref{wfs def condition}. Thus, a factorization structure $\varphi$ can be identified with the subspace $\varphi(\mathfrak{h}) \subset V^*$.
\end{rem}

The defining equation \eqref{wfs def condition} of a factorization structure induces an assignment: a generic point $\ell\in\mathbb{P}(V_j)$ is mapped to
$\varphi^{-1} \left( \varphi(\mathfrak{h}) \cap \Sigma_{j,\ell}^0 \right) \in \mathbb{P}(\mathfrak{h})$. It can be shown that this map extends uniquely to a regular map $\psi_j: \mathbb{P}(V_j) \to \mathbb{P}(\mathfrak{h})$, and thus each factorization structure of dimension $m$ induces $m$ projective algebraic curves called \textit{factorization curves}.
Recall that the tautological section $\mathbb{P}(\mathfrak{h}) \to \mathcal{O}_\mathfrak{h}(1) \otimes \mathfrak{h}$ assigns to each class $[z] \in \mathbb{P}(\mathfrak{h})$ the canonical inclusion of the corresponding 1-dimensional vector space $\langle z \rangle$ into $\mathfrak{h}$ viewed as an element of $\langle z \rangle^* \otimes \mathfrak{h}$.
By pulling the tautological section back via $\psi_j$, the curve $\psi_j$ can be viewed as the section of $\psi_j^*\mathcal{O}_\mathfrak{h}(1) \otimes \mathfrak{h} \cong \mathcal{O}_{V_j}(d_j) \otimes \mathfrak{h}$ over $\mathbb{P}(V_j)$ for some integer $d_j$, and we can define the degree $\deg \psi_j = d_j$. 

\begin{example}\label[example]{Segre2}
For 1-dimensional subspace $a_j \subset V_j^*$, $j=1,2$, the 2-dimensional \textit{Segre factorization structure} is defined to be the canonical inclusion
\begin{gather}
V_1^* \otimes a_2 + a_1 \otimes V_2^* \hookrightarrow V_1^* \otimes V_2^*.
\end{gather}
To verify that it is a genuine factorization structure, one observes that for $\ell \in \mathbb{P}(V_1)$ such that $\ell^0 \neq a_1$,
\begin{gather}
\left( V_1^* \otimes a_2 + a_1 \otimes V_2^* \right) \cap \ell^0 \otimes V_2^* = \ell^0 \otimes  a_2
\end{gather}
holds, and thus \eqref{wfs def condition} is satisfied generically for $j=1$, and similarly for $j=2$.
In turn, the associated factorization curves are
\begin{align}
\psi_1: \mathbb{P}(V_1) &\to \mathbb{P}(V_1^* \otimes a_2 + a_1 \otimes V_2^*) \nonumber \\
\ell &\mapsto \ell^0 \otimes a_2
\end{align}
and
\begin{align}
\psi_2: \mathbb{P}(V_2) &\to \mathbb{P}(V_1^* \otimes a_2 + a_1 \otimes V_2^*) \nonumber \\
\ell &\mapsto a_1 \otimes \ell^0,
\end{align}
and both have degree one.
Geometrically, these are two distinct lines in a projective plane.
\end{example}

\begin{example}\label[example]{Veronese2}
For a 2-dimensional vector space $W$, the 2-dimensional \textit{Veronese factorization structure} is defined to be the canonical inclusion of the symmetric tensors
\begin{gather}
S^2W^* \hookrightarrow W^* \otimes W^*.
\end{gather}
Since
\begin{gather}
S^2W^* \cap \ell^0 \otimes W^* = S^2W^* \cap W^* \otimes \ell^0 = \ell^0 \otimes \ell^0,
\end{gather}
it is a factorization structure.
Its factorization curves $\psi_j: \mathbb{P}(W) \to \mathbb{P}(S^2W^*)$, $j=1,2$, coincide,
\begin{gather}
\psi_1(\ell) = \psi_2(\ell) = \ell^0 \otimes \ell^0
\end{gather}
for any 1-dimensional $\ell \subset W$, and have degree two.
Geometrically, this is a quadric in a projective plane.
\end{example}

Recall that 2-dimensional Segre and Veronese factorization structures, both originally found in \cite{apostolov2015ambitoric}, classify 2-dimensional factorization structures up to isomorphism of factorization structures.

\subsection{Segre-Veronese factorization structures}
This subsection presents all known examples of factorization structures.

For $i\in\{1,\ldots,m\}$ we say that the term $a_i$ in $a_1\otimes\cdots\otimes a_m$ is in the $i$th \textit{slot}.
If a partition of $m$ is given, $m=d_1+\cdots+d_k$, $d_j\geq1$, slots group into $k$ groups with the $j$th group containing $d_j$ slots, $j\in\{1,\ldots,k\}$.
Slots belonging to the $j$th group are referred to as \textit{grouped $j$-slots}.
In fact, positions in such a tensor product can be labelled by pairs $(j,r)$, where $j\in\{1,\ldots,k\}$ and $r\in\{1,\ldots,d_j\}$.
For a partition of $m$ as above and a fixed $j\in\{1,\ldots,k\}$ we define the linear operator
$$ins_j:
(W_j^*)^{\otimes d_j}\otimes\bigotimes_{\substack{i=1\\i\neq j}}^k (W_i^*)^{\otimes d_i}
\to
\bigotimes_{i=1}^k (W_i^*)^{\otimes d_i}$$
which acts on decomposable tensors by
\begin{align*}
\left(w_j^1\otimes\cdots\otimes w_j^{d_j}\right)
\otimes
\bigotimes_{\substack{i=1\\i\neq j}}^k
\left(w_i^1\otimes\cdots\otimes w_i^{d_i}\right)
\mapsto
\bigotimes_{i=1}^k
\left(w_i^1\otimes\cdots\otimes w_i^{d_i}\right),
\end{align*}
where $W_j$, $j=1,\ldots,k$, are vector spaces.
Partitions $m=d_1+\cdots+d_p$ and $m=e_1+\cdots+e_q$ are considered to be the same if $\{d_1,\ldots,d_p\}=\{e_1,\ldots,e_q\}$, and distinct if they are not the same.

\begin{defn}\label[defn]{SV def}
For $d_1,\ldots,d_k$ a partition of an integer $m\geq2$ and $W_r$, $r=1,\ldots,k$, 2-dimensional vector spaces, let 
$\Gamma_j
\subset
\bigotimes_{r=1,r\neq j}^k(W_r^*)^{\otimes d_r}$,
$j\in\{1,\ldots,k\}$, be 1-dimensional subspaces such that
\begin{align}\label{SV image}
\sum_{j=1}^{k}
ins_j
\left(
S^{d_j}W_j^*\otimes\Gamma_j
\right)
\end{align}
has dimension $m+1$, where $S^{d_j}W_j^*\subset(W_j^*)^{\otimes d_j}$ is viewed as the subspace of symmetric tensors.
Define vector spaces $V_1,\ldots,V_m$ by
\begin{gather}\label{spaces V_j}
V_{d_1 + \cdots + d_{j-1} + 1} =
V_{d_1 + \cdots + d_{j-1} + 2} =
\cdots = V_{d_1 + \cdots + d_{j-1} + d_j} =
W_j,
\hspace{.5cm}
j=1,\ldots,k,
\end{gather}
where $d_0$ is defined to be zero.
The \textit{standard Segre-Veronese factorization structure} $\varphi: \mathfrak{h} \to V^*$ is defined to be such that $\mathfrak{h}$ is the $(m+1)$-dimensional space \eqref{SV image}, $V^* = \otimes_{j=1}^m V_j^*$ where $V_j$ is defined by \eqref{spaces V_j}, and $\varphi$ is the canonical inclusion of $\mathfrak{h}$ to $V^*$, i.e., it is
\begin{align}\label{SV inclusion}
\sum_{j=1}^{k}
ins_j
\left(
	S^{d_j}W_j^*\otimes \Gamma_j
\right)
\hookrightarrow
\bigotimes_{j=1}^k(W_j^*)^{\otimes d_j}.
\end{align}
Factorization structures corresponding to trivial partitions,
\begin{gather}
\sum_{j=1}^m ins_j \left( W_j^* \otimes \Gamma_j \right)
\hookrightarrow
\bigotimes_{j=1}^m W_j^*
\end{gather}
for $m=1+\cdots+1$, and
\begin{gather}
S^mW^* \hookrightarrow (W^*)^{\otimes m}
\end{gather}
for $m=m$, are respectively called \textit{Segre} and \textit{Veronese}.
An element of the isomorphism class of a standard Segre-Veronese factorization structure is referred to as a Segre-Veronese factorization structure.
\end{defn}
We frequently refer to the 1-dimensional spaces $\Gamma_j$ as \textit{defining tensors} of the standard Segre-Veronese factorization structure, since one does need them to define a given Segre-Veronese factorization structure, and since each $\Gamma_j$ is a linear span of a tensor.
\begin{rem}
Note that when $m=2$, Segre and Veronese factorization structures recover factorization structures from \Cref{Segre2} and \Cref{Veronese2}, respectively.

To verify that \eqref{SV inclusion} defines a factorization structure, we observe that for $i\in\{1,\ldots,k\}$ and generic $\ell\in\mathbb{P}(W_i)$,
\begin{align}\label{standard curves}
\varphi(\mathfrak{h})
\cap
\Sigma_{d_1+\cdots+d_{i-1}+q,\ell}^0
=
ins_i
\left(
	(\ell^0)^{\otimes d_i}
	\otimes
	\Gamma_i
\right),
\end{align}
holds, where $\varphi(\mathfrak{h})$ is \eqref{SV image}, $q\in\{1,\ldots,d_i\}$ and $d_0$ is defined to be $0$.
Thus, for each $i=1,\ldots,k$, exactly $d_i$ factorization curves coincide,
\begin{gather}
\psi_{d_1+\cdots+d_{i-1}+1}(\ell) = \cdots = \psi_{d_1+\cdots+d_{i-1}+d_i}(\ell) =
ins_i \left( (\ell^0)^{\otimes d_i} \otimes	\Gamma_i \right),
\end{gather}
and each can be viewed as the rational normal curve in its span,
\begin{gather}
\text{span} \{ \psi_{d_1+\cdots+d_{i-1}+q}(\ell) \mid \ell \in \mathbb{P}(W_i) \} =
ins_i \left( S^{d_i}W_i^* \otimes \Gamma_i \right),
\end{gather}
$q=1,\ldots,d_i$.
In particular, the factorization curve of the Veronese factorization structure is the genuine rational normal curve.
\end{rem}

Determining in general which choices of $\Gamma_j$, $j=1,\ldots,k$, give rise to a factorization structure, i.e., make \eqref{SV image} an $(m+1)$-dimensional vector space, is a challenging task.
Instead, in the following we exemplify particular choices which effortlessly guarantee the correct dimension. \par
\begin{example}\label[example]{k=2 example}
We examine the standard Segre-Veronese factorization structure for $k=2$. To this end, let $m=d_1+d_2$ be a partition, and $\Gamma_1\subset (W_2^*)^{\otimes d_2}$ and $\Gamma_2\subset (W_1^*)^{\otimes d_1}$ be 1-dimensional subspaces.
Observe that the dimension of the image of
\begin{align}\label{k=2}
S^{d_1}W_1^*\otimes\Gamma_1
+
\Gamma_2\otimes S^{d_2}W_2^*
\hookrightarrow
(W_1^*)^{\otimes d_1}\otimes (W_2^*)^{\otimes d_2}
\end{align}
is $m+1$ if and only if $\Gamma_1\subset S^{d_2}W_2^*$ and $\Gamma_2\subset S^{d_1}W_1^*$, which completely characterises choices of $\Gamma_1$ and $\Gamma_2$ leading to a factorization structure.
\end{example}

\begin{example}\label[example]{one intersection example}
For a partition $m=d_1+\cdots+d_k$ and 1-dimensional subspaces $ a^r \subset W_r^*$, $r=1,\ldots,k$, we define \textit{the product Segre-Veronese factorization structure} as the standard Segre-Veronese factorization structure such that
\begin{align}\label{product tensors}
\Gamma_j=
\bigotimes_{\substack{r=1\\r\neq j}}^k (a^r)^{\otimes d_r},
\hspace{.5cm}
j=1,\ldots,k.
\end{align}
These data ensure that any two summands of \eqref{SV image} intersect in
$\bigotimes_{r=1}^k(a^r)^{\otimes d_r}$, which implies that the dimension of \eqref{SV image} is $m+1$. Therefore, the product Segre-Veronese factorization structure is indeed a factorization structure.
The product Segre-Veronese factorization structure with partition $m=1+\cdots+1$ is called the \textit{product Segre factorization structure}.
\end{example}

\begin{rem} \label[rem]{decomposes in grouped p-slots}
It turns out that if $\Gamma_j \subset \bigotimes_{r=1,r\neq j}^k(W_r^*)^{\otimes d_r}$, $j=1,\ldots,k$, determine a Segre-Veronese factorization structure, then $\Gamma_j \subset \bigotimes_{r=1,r\neq j}^k S^{d_r}W_r^*$, $j=1,\ldots,k$.
In particular, if such $\Gamma_j$ decomposes at the $(p,q)$-th slot, $p\neq j$, i.e., there exists a 1-dimensional subspace $P \subset W_p$ so that the contraction at the $(p,q)$-th slot $\langle \Gamma_j, P \rangle$ is zero, then $\Gamma_j$ \textit{decomposes in grouped $p$-slots} which, by definition, means that there is a 1-dimensional subspace $e_j^p \subset W_p^*$ so that $\Gamma_j = ins_p \left( (e_j^p)^{\otimes d_p} \otimes \tilde{\Gamma}_j \right)$ for some $\tilde{\Gamma}_j$.
\end{rem}

\begin{example}
We conclude this section with an example of 3-dimensional Segre factorization structure whose all defining tensors indecomposable. Observe that only in dimension 3, the annihilator $\mathfrak{h}^0\xhookrightarrow{} V$ of a factorization structure $\mathfrak{h}\xhookrightarrow{} V^*$ has the right dimension for being a factorization structure. The Veronese $S^3W^*\xhookrightarrow{} (W^*)^{\otimes 3}$ has the annihilator
\begin{align}
W\otimes\bigwedge^2W +
ins_2\left( W\otimes\bigwedge^2W \right) +
\bigwedge^2W\otimes W
\xhookrightarrow{}
W\otimes W\otimes W,
\end{align}
a factorization structure with indecomposable defining tensors.
\end{example}

\subsection{Products and decomposable elements}
Determining the defining tensors for Segre-Veronese factorization structure, i.e., finding $\Gamma_j$, $j=1,\ldots,k$, such that \eqref{SV image} is $(m+1)$-dimensional, is generally a complex task.
We introduce below the product of factorization structures that, among other uses, allows explicit construction of defining tensors for plenitude Segre-Veronese examples.
In particular, this product gives a description and complete characterisation of an important subclass: Segre-Veronese factorization structures with decomposable defining tensors.

\begin{defn}\label[defn]{prod}
Let $\chi:\mathfrak{g}\to W_1^*\otimes\cdots\otimes W_n^*$ and $\varphi:\mathfrak{h}\to V_1^*\otimes\cdots\otimes V_m^*$ be two factorization structures and $T\subset\chi(\mathfrak{g})$ and $S\subset\varphi(\mathfrak{h})$ any two 1-dimensional subspaces.
We define the \textit{product} of $\varphi$ and $\chi$ (with respect to $T$ and $S$) to be the factorization structure given by the canonical inclusion
\begin{align}\label{product}
\varphi(\mathfrak{h})
\otimes
T+
S
\otimes
\chi(\mathfrak{g})
\hookrightarrow
V_1^*\otimes\cdots\otimes V_m^*\otimes W_1^*\otimes\cdots\otimes W_n^*.
\end{align}
\end{defn}

Examples of the product include the Segre factorization structure of dimension 2 from \Cref{Segre2}, the Segre-Veronese factorization structure corresponding to the partition $m=d_1+d_2$ \eqref{k=2}, and the product Segre-Veronese factorization structure \eqref{product tensors}.
In fact, the latter is a product in multiple ways. Indeed, for $I:=\{1,\ldots,k_0\}\subset\{1,\ldots,k\}$, $1\leq k_0<k$, we have
\begin{gather}\nonumber
\Bigg(
	\sum_{j\in I}
	ins_j
	\bigg(
		S^{d_j}W_j^*
		\otimes
		\bigotimes_{\substack{r\in I\\r\neq j}} (a^r)^{\otimes d_r}
	\bigg)
\Bigg)
\otimes
\bigotimes_{r\in I^c}
(a^r)^{\otimes d_r}+\\
+
\bigotimes_{r\in I}
(a^r)^{\otimes d_r}
\otimes
\Bigg(
	ins_j
	\bigg(
	\sum_{j\in I^c}
	S^{d_j}W_j^*
	\otimes
	\bigotimes_{\substack{r\in I^c\\r\neq j}} (a^r)^{\otimes d_r}
	\bigg)
\Bigg)\label{product SV is product fs}
\end{gather}
where $I^c$ denotes the complement of $I$. Clearly, such a decomposition exists for any non-trivial $I\subset\{1,\ldots,k\}$. \\

Additionally, we have

\begin{lemma}\label[lemma]{tensors split}
Let
$\varphi(\mathfrak{h})\otimes T+
S\otimes\chi(\mathfrak{g})
\hookrightarrow
V_1^*\otimes\cdots\otimes V_m^*\otimes W_1^*\otimes\cdots\otimes W_n^*$
be a product factorization structure.
Then
\begin{align}
I
\otimes
K  \subset
\varphi(\mathfrak{h})\otimes T+
S\otimes\chi(\mathfrak{g})
\end{align}
for some 1-dimensional subspaces $ I \subset V_1^*\otimes\cdots\otimes V_m^*$ and $ K  \subset W_1^*\otimes\cdots\otimes W_n^*$ if and only if
\begin{align}\label{split elements in product fs}
\bigg[
	I=
	S
	\text{ and }
	K \subset \chi(\mathfrak{g})
\bigg]
\hspace{.4cm}\text{or}\hspace{.4cm}
\bigg[
	K=
	T
	\text{ and }
	I \subset \varphi(\mathfrak{h})
\bigg]
\end{align}
\end{lemma}

\begin{example}
Using products we can construct new examples of Segre-Veronese factorization structure as follows. Let $\Gamma_3$ be a 1-dimensional subspace of the image of the inclusion \eqref{k=2} and $\Gamma$ a 1-dimensional subspace of $S^{d_3}W_3^*$, then
\begin{align}\label{k=3}
S^{d_1}W_1^* \otimes \Gamma_1\otimes\Gamma
+
\Gamma_2 \otimes S^{d_2}W_2^* \otimes \Gamma
+
\Gamma_3 \otimes S^{d_3}W_3^*
\hookrightarrow
\bigotimes_{j=1}^3 (W_j^*)^{\otimes d_j}
\end{align}
is a Segre-Veronese factorization structure of dimension $d_1+d_2+d_3$. We could continue further and make a product of \eqref{k=3} with $S^{d_4}W_4^*$ or form a product of two factorization structures of type \eqref{k=2} to obtain a Segre-Veronese factorization structure with $k=4$. And so on.
\end{example}

The former approach of iterative products happens to completely characterise decomposable Segre-Veronese factorization structures:

\begin{example}\label[example]{decomposable example}
We generalise \Cref{one intersection example}.
\textit{The decomposable Segre-Veronese factorization structure} is defined as the standard Segre-Veronese factorization structure such that $\Gamma_j$ are decomposable, i.e., 
$\Gamma_j
=
\bigotimes_{\substack{r=1\\r\neq j}}^k
\bigotimes_{p=1}^{d_r}
a_j^{r,p}$
for some 1-dimensional subspaces $a_j^{r,p} \subset W_r^*$, $j=1,\ldots,k$.
It can be shown that if such decomposable $\Gamma_j$, $j=1,\ldots,k$, determine a factorization structure, then it must be that
$a_j^{r,1}=
\cdots=
a_j^{r,d_r} =: a_j^r$,
and hence
\begin{align}\label{dec tensors}
\Gamma_j=
\bigotimes_{\substack{r=1\\r\neq j}}^k (a^r_j)^{\otimes d_r}
\end{align}
necessarily.
However, it remains unclear which choices of $a_j^r$ lead to a factorization structure.
With some effort, these structures and hence $a_j^r$s can be characterised as iterative products of Veronese factorization structures as described above, where each product is formed by completely decomposable, i.e., decomposable in each slot, tensors $S$ and $T$ (see \Cref{prod} and \Cref{tensors split}).
\end{example}

\subsection{Contact geometry} \label{contact geom}

Let $\mathcal{D}$ be a co-rank 1 distribution on a smooth manifold $N$, and $\eta: TN \to TN / \mathcal{D}$ the canonical projection.
The pair $(N,\mathcal{D})$ is called \textit{contact} manifold or \textit{contact geometry}, if the associated \textit{Levi form} $L: \mathcal{D} \times \mathcal{D} \to TN / \mathcal{D}$, $L(X,Y) := - \eta([X,Y])$, is non-degenerate at each point: the expression $L(X_p,Y_p)$, $X_p,Y_p \in \mathcal{D}_p$, is independent of the choice of extensions $X,Y \in \mathcal{D}$ of $X_p$ and $Y_p$, and can be viewed as a fibre-wise bilinear form, which is required to be non-degenerate at each point.
Therefore, for a local non-zero section $\chi \in TN / \mathcal{D}$, the 1-form $\eta_\chi := \chi^{-1} \eta$ is \textit{contact} since $\ker \eta_\chi = \mathcal{D}$ and $d \eta_\chi|_\mathcal{D} = \chi^{-1} L$ is non-degenerate, which also implies that $\dim(N)$ is odd.

A vector field $X$ on $N$ is \textit{contact} if $\mathcal{L}_X Y \in \mathcal{D}$ for any $Y \in \mathcal{D}$.
A straightforward computation \cite{apostolov2021weighted} shows that $\eta$ induces a linear isomorphism between contact vector fields and sections of $TN / \mathcal{D}$, whose inverse is denoted here by $\chi \mapsto X_\chi$.
Using this isomorphism to transfer the Lie algebra structure of contact vector fields onto sections of $TN / \mathcal{D}$ by defining $[\xi, \chi] = \eta([X_\xi, X_\chi])$, we obtain the \textit{contact Lie algebra} $\mathfrak{con}(N,\mathcal{D})$.
Observe that $X_\chi$ is the \textit{Reeb vector field} of $\eta_\chi$, i.e., the unique vector field $X$ such that $\langle \eta_\chi, X \rangle = 1$ and $\mathcal{L}_X \eta_\chi = 0$.
A contact geometry $(N,\mathcal{D})$ is \textit{co-orientable} if $TN / \mathcal{D}$ is orientable; note that a global contact form exists if and only if $TN / \mathcal{D}$ is trivialisable, or equivalently orientable, and any two differ by a nowhere vanishing function on $N$.
The group $\text{Aut}(N,\mathcal{D})$ of automorphisms of the contact geometry consists of \textit{contactomorphisms}; elements $\phi \in \text{Diff}(N)$ such that $\phi_* \mathcal{D} \subset \mathcal{D}$, or equivalently, $\phi$-pullback of a contact form is a contact form.

\subsection{Sasaki geometry}

A \textit{CR structure} on a contact manifold $(N,\mathcal{D})$ is an endomorphism $J$ of $\mathcal{D}$ satisfying $J^2 = -id$ such that the distribution of $(1,0)$-vectors in $\mathcal{D} \otimes \mathbb{C}$ is involutive.
It is \textit{strictly pseudo-convex} if $L(\cdot,J\cdot)$ is (positive/negative) definite at each point, which makes $(N,\mathcal{D})$ co-oriented.
In that case, an orientation may be fixed by declaring sections $\chi \in TN/\mathcal{D}$ to be \textit{positive} if $\chi^{-1}L(\cdot,J\cdot)$ is positive definite, thus obtaining the open cone $\mathfrak{con}_+(N,\mathcal{D}) \subset \mathfrak{con}(N,\mathcal{D})$ of positive sections.
The Lie algebra of vector fields preserving the CR structure reads
\begin{gather}
\mathfrak{cr}(N,\mathcal{D},J) = \{ \chi \in \mathfrak{con}(N,\mathcal{D}) \mid \mathcal{L}_{X_\chi} J = 0 \}.
\end{gather}

\begin{defn}[\cite{apostolov_cr_2020}]
Let $(N,\mathcal{D},J)$ be a strictly pseudo-convex CR manifold. The \textit{Sasaki cone} of $(N,\mathcal{D},J)$ is
$\mathfrak{cr}_+(N,\mathcal{D},J) := \mathfrak{cr}(N,\mathcal{D},J) \cap \mathfrak{con}_+(N,\mathcal{D})$.
If $\mathfrak{cr}_+(N,\mathcal{D},J) \neq \emptyset$, then $(N,\mathcal{D},J)$ is said to be of \textit{Sasaki type}, an element $\chi \in \mathfrak{cr}_+(N,\mathcal{D},J)$ is called a \textit{Sasaki structure} on $(N,\mathcal{D},J)$, with a \textit{Sasaki-Reeb vector field} $X_\chi$, and $(N,\mathcal{D},J,\chi)$ is called a \textit{Sasaki manifold}. We say $\chi$ is \textit{quasi-regular} if the flow of $X_\chi$ generates an $\mathbb{S}^1$-action on $N$, and moreover \textit{regular} if this action is free.
\end{defn}

\subsection{Toric contact geometry}
A $(2m+1)$-dimensional contact connected manifold $(N,\mathcal{D})$ is called \textit{toric} if it is equipped with an effective action of the torus $\mathbb{T}^{m+1}$, i.e., with an injective group homomorphism $\mathbb{T}^{m+1} \to \text{Aut}(N,\mathcal{D})$.
If $N$ is further assumed compact, then $m+1$ is the maximal possible dimension of such a torus (see \cite{lerman2002maximaltoricontactomorphismgroups, BOYER2013190}).
As the derivative of this action
\begin{align}
\text{Lie}(\mathbb{T}^{m+1}) &\to \mathcal{X}(N)  \\
a &\mapsto \tilde{X}_a, \nonumber
\end{align}
given by the fundamental vector fields, takes values in contact vector fields, we use $\eta$ (see \Cref{contact geom}) to view the action of the Lie algebra $\mathfrak{h} := \text{Lie}(\mathbb{T}^{m+1})$ on $N$ as a Lie algebra morphism $\mathfrak{h} \to \mathfrak{con}(N,\mathcal{D})$, $a \mapsto \xi_a$.
Hence, $\eta(\tilde{X}_a) = \xi_a$.
Throughout this paper we reserve the letter $\beta$ for a fixed (background or reference) element of $\mathfrak{h}$. \par
The $\mathbb{T}^{m+1}$-action on $(N,\mathcal{D})$ has its associated \textit{momentum section} $\mu : N \to TN / \mathcal{D} \otimes \mathfrak{h}^*$ defined by $\langle a, \mu \rangle = \eta (\tilde{X}_a)$ (see \cite{apostolov_cr_2020}).
The background element $\beta \in \mathfrak{h}$, as any other element of $\mathfrak{h}$, determines a section $\langle \beta, \mu \rangle \in TN/\mathcal{D}$, which trivialises the defining equation of $\mu$ on the open set $U$ where $\langle \beta, \mu\rangle \neq 0$,
\begin{gather}
\eta_\beta (\tilde{X}_a) =
\langle a, \mu_\beta \rangle,
\end{gather}
where $\eta_\beta = \eta / \langle \beta, \mu \rangle$ and $\mu_\beta = \mu / \langle \beta, \mu \rangle$.
In particular, the function $\mu_\beta$ is valued in the affine chart $\{ v \in \mathfrak{h}^* \mid \langle \beta, v \rangle =1 \}$.
One finds that $\mathcal{L}_{\tilde{X}_a} \eta_\beta = 0$ for any $a \in \mathfrak{h}$. \par

The stratification of compact $N$ by the orbit dimension or the generalised Delzant construction \cite{lerman2001contact} show the existence of an open dense subset $N^0 \subset N$ where the $\mathbb{T}^{m+1}$-action is free, which renders $N^0$ as a $\mathbb{T}^{m+1}$-principal bundle whose vertical fields are restricted $\tilde{X}_a$, $a \in \mathfrak{h}$.
The construction also yields angular coordinates $\tau: N^0 \to \mathbb{T}^{m+1}$ satisfying $d\tau(\tilde{X}_a) = a$, which, together with the trivialised momentum section $\mu_\beta$, form a coordinate system on $N^0 \cap U$. \\

Since our interest lies solely in the 1-form $d\tau$ rather than the coordinates $\tau$ themselves, we present an alternative construction for $d\tau$.
Knowing that the quotient manifold $N^0/\mathbb{T}^{m+1}$ is simply-connected \cite{lerman2001contact}, being the interior of a polyhedral cone, the general bijective correspondence \cite{morita2001geometry}, stating that conjugacy classes of a morphism from the fundamental group $\pi_1(N^0/\mathbb{T}^{m+1})$ into a Lie group $G$ are one-to-one with isomorphism classes of flat $G$-principal bundles, shows the existence of a unique isomorphism class of flat $\mathbb{T}^{m+1}$-bundles over $N^0/\mathbb{T}^{m+1}$.
Necessarily, these flat connections $\kappa$, i.e., $d\kappa = 0$, are pullbacks of the Maurer-Cartan connection of the trivial $\mathbb{T}^{m+1}$-bundle; the latter connection being the pullback of the Maurer-Cartan form on $\mathbb{T}^{m+1}$ via the projection from the trivial bundle onto $\mathbb{T}^{m+1}$.
One finds that $\mathcal{L}_{\tilde{X}_a} \kappa = 0$ for all $a \in \mathfrak{h}$, and that $\mathcal{L}_X \kappa = 0$ for all horizontal vector fields $X$ if and only if $d\kappa = 0$. \\

As $d \langle \beta, \mu_\beta \rangle =0$, distributions of 1-forms determined by $d\tau$ and $\mu_\beta$ satisfy
\begin{gather}
\left\langle d \langle b, \mu_\beta \rangle \mid b \in \mathfrak{h} \right\rangle|_U  \oplus
\left\langle \langle d\tau, q \rangle \mid q \in \mathfrak{h}^* \right\rangle|_U =
T^*U,
\end{gather}
which follows as $d\tau$ reproduces fundamental vector fields and
\begin{gather}
\left\langle d \langle b, \mu_\beta \rangle \mid b \in \mathfrak{h} \right\rangle =
\text{Ann } \left\langle \tilde{X}_a \mid a \in \mathfrak{h} \right\rangle
\end{gather}
by
\begin{gather}
( d \langle b, \mu_\beta \rangle ) (\tilde{X}_a) =
\mathcal{L}_{\tilde{X}_a} \langle b, \mu_\beta \rangle =
\mathcal{L}_{\tilde{X}_a} \left( \eta_\beta (\tilde{X}_{b}) \right) =
0.
\end{gather}
By fixing dual bases of $\mathfrak{h}$ and $\mathfrak{h}^*$ one obtains a frame and the dual coframe such that $[\partial_{\mu_j}, \tilde{X}_a] = 0$ for all $j = 1,\ldots, m$ and $a \in \mathfrak{h}$, and $\partial_{\mu_j}$, $j=1,\ldots,m$, are horizontal. \par

Since $d\eta_\beta$ is $\tilde{X}_\beta$-basic, non-degenerate, and
\begin{gather}
d \langle a, \mu_\beta \rangle =
d \left( \eta_\beta (\tilde{X}_a) \right) =
d \iota_{\tilde{X}_a} \eta_\beta =
\left( - \iota_{\tilde{X}_a} d + \mathcal{L}_{\tilde{X}_a} \right) \eta_\beta =
- \iota_{\tilde{X}_a} d \eta_\beta
\end{gather}
holds, it is a \textit{transversal symplectic form} on $U$.

\section{Separable geometries and their scalar curvatures}
This section introduces toric separable contact, CR and Kähler geometries.
A toric separable contact geometry $N$ is defined as a toric contact geometry which, on the dense open subset where the action is free, is isomorphic to the model contact geometry, determined by a factorization structure.
In particular, for $N$ compact, it is a compactification of the model geometry.
Then, a separable contact geometry is equipped with a toric CR structure $J$, which is required to 'separate' variables in the sense of \Cref{J initial}.
This requirement forces the involved functions $\zeta_j$ to be the factorization curves $\psi_j$.
At last, separable Kähler geometries are obtained as quotients by Sasaki-Reeb vector fields of separable CR geometries.
Additionally, for any separable Kähler geometry, explicit formulae for the Laplace operator acting on invariant functions, the Ricci potential, symplectic and holomorphic volume forms, and the scalar curvature are computed.

\subsection{Separable contact geometries}\label{contact geoms}
In this subsection, for the sake of clarity, we denote a toric contact geometry $(N, \mathcal{D})$ as the triple $(N,\mathcal{D},\mu)$, where $\mu$ is the momentum section, uniquely associated with geometry.\\

Let $(\tilde{N},\tilde{\mathcal{D}},\tilde{\mu})$, $\dim \tilde{N} = 2m+1$, be a toric contact geometry with the natural projection $\tilde{\eta}:T\tilde{N} \to T\tilde{N} / \tilde{\mathcal{D}}$ and the momentum section $\tilde{\mu}$.
To define separable contact geometries we impose conditions on $(\tilde{N},\tilde{\mathcal{D}},\tilde{\mu})$ concerning the action and its momentum map, which we discuss now and collect later in \Cref{contact def}. 

We assume that the $\mathbb{T}^{m+1}$-action is free and that the base of the associated principal bundle is simply connected.
Thus, the bundle carries a unique flat connection $\theta \in \Omega^1(\tilde{N}, \mathfrak{h})$, called the angle coordinates, and we further assume that $\tilde{\mathcal{D}}$ is so that $\ker \theta \subset \tilde{\mathcal{D}}$, which forces $\tilde{\eta} = \langle \tilde{\mu}, \theta \rangle$.\par
Motivated by the projective invariance and shapes of Kähler structures found in ambitoric geometry \cite{apostolov2015ambitoric,apostolov2016ambitoric}, and by the impactful use of separable coordinates in \cite{apostolov_cr_2020}, we define \textit{separable coordinates for the contact geometry} $(\tilde{N},\tilde{\mathcal{D}},\tilde{\mu})$ as coordinates $[\textbf{x}_j]:\tilde{N}\to I_j^0\subset\mathbb{P}(V_j)$, $\dim V_j =2$, $j=1,\ldots,m$, $I_j^0$ being the interior of a simply connected region in the projective line, for which there exists a factorization structure $\varphi: \mathfrak{h} \to V^*$ such that $\tilde{\mu} = \varphi^t \textbf{x}$, where $\textbf{x}=\textbf{x}_1 \otimes \cdots \otimes \textbf{x}_m$ and $\textbf{x}_j$ is the pullback of the tautological section $\mathbb{P}(V_j) \to \mathcal{O}_{V_j}(1) \otimes V_j$ by $[\textbf{x}_j]$, and thus
\begin{gather}\label{sep coor}
\varphi^t \textbf{x}: \tilde{N} \to \left( \bigotimes_{j=1}^m [\textbf{x}_j]^* \mathcal{O}_{V_j}(1) \right) \otimes \mathfrak{h}^*.
\end{gather}
In particular,
\begin{gather}\label{sep coor cons}
\eta = \langle \varphi^t \textbf{x}, \theta \rangle
\hspace{.4cm}
\text{and}
\hspace{.4cm}
T\tilde{N} / \tilde{\mathcal{D}} =
\bigotimes_{j=1}^m [\textbf{x}_j]^* \mathcal{O}_{V_j}(1).
\end{gather}

We obtained the model $(\tilde{N},\tilde{\mathcal{D}},\tilde{\mu})$, $\tilde{\mu} = \varphi^t \textbf{x}$, of separable contact geometries.

\begin{defn}\label[defn]{contact def}
A \textit{(toric) separable contact geometry} $(N,\mathcal{D},\varphi)$ is a toric contact geometry $(N, \mathcal{D},\mu)$ which, on the open dense set where the action is free, is isomorphic with $(\tilde{N}, \tilde{\mathcal{D}}, \varphi^t \textbf{x})$ for some factorization structure $\varphi$ and $\tilde{\mathcal{D}}$ as above.
\end{defn}

\begin{rem}
One could consider global (separable) coordinates $\textbf{x}_j$ for a toric contact compact/non-compact geometry and demand $\mu = \varphi^t \textbf{x}$ globally, and only then require $\ker \theta \subset \mathcal{D}$ on the dense open subset where the action is free.
In fact, for compact examples of separable Kähler geometries as defined below, these two approaches are equivalent since separable coordinates extend from the dense open subset to global coordinates in the compactification.
\end{rem}

Note that as $[\textbf{x}_j]$, $j=1,\ldots,m$, is valued in simply connected $I_j^0$, the line bundle $T\tilde{N} / \tilde{\mathcal{D}}$ is trivialisable.
Indeed, we can introduce a basis $e_j^1, e_j^2$ of $V_j^*$, $j=1,\ldots,m$, and affine charts on $\mathbb{P}(V_j)$ by trivialising each $\mathcal{O}_{V_j}(1)$ by the section $\langle \textbf{x}_j, e_j^1 \rangle$ so that under these identifications $\textbf{x}_j = (1, x_j)$ for a function $x_j$ on $\tilde{N}$, and $[\textbf{x}_j] = [ 1: x_j]$.
Respectively, the momentum section $\mu$ and the section $\langle \mu, \beta \rangle$ become in this trivialisation the map $\mu = \varphi^t \otimes_{j=1}^m (1,x_j): N \to \mathfrak{h}^*$ and a function.
In particular, the fields $\partial_{x_j}$, $j=1,\ldots,m$, are horizontal. \\

To address the relation between geometries associated to isomorphic factorization structures, let $g_j:V_j^* \to V_j^*$, $j=1,\ldots,m$, be linear isomorphisms providing an isomorphism of factorization structures, i.e., $\varphi_2 = g \varphi_1$, $g=\otimes_{j=1}^m g_j$, and let $\textbf{x}$ and $\textbf{y}$ be respective separable coordinates on $(\tilde{N},\tilde{\mathcal{D}},\tilde{\mu})$ corresponding to $\varphi_1$ and $\varphi_2$, i.e., $\tilde{\mu} = \varphi_1^t \textbf{x}$ and $\tilde{\mu} = \varphi_2^t \textbf{y} = \varphi_1^t g^t \textbf{y}$.
Clearly, a projective change of coordinates shows that these geometries are the same, and on the other hand, a projective transformation of coordinates in a separable contact geometry gives rise to isomorphic factorization structures.
Therefore, isomorphisms of factorization structures encode individual $GL(V_j)$-transformation of coordinates $\textbf{x}_j$, $j=1,\ldots,m$.
More concretely, one can use the basis vector $e_j^1 \in V_j^*$ to trivialise the transformed coordinate $\textbf{x}_j = g_j^t \textbf{y}_j$ as above, and find the $\mathbb{P}GL(V_j)$-action in terms of the Möbius transformation; if we denote
\begin{gather}\label{g_j}
g_j =
\begin{bmatrix}
a_j & b_j \\
c_j & d_j
\end{bmatrix},
\end{gather}
then
\begin{gather}
(1,x_j) =
\frac{\textbf{x}_j}{\langle \textbf{x}_j, e_j^1 \rangle} =
\frac{[g_j^t]\textbf{y}_j}{\langle [g_j^t]\textbf{y}_j, e_j^1 \rangle} =
\left( 1,\frac{b + d y_j}{a + cy_j} \right),
\end{gather}
where $[g_j^t] \in \mathbb{P}GL(V_j)$.
In particular, the projective coordinate changes are one-to-one with projectivised isomorphisms of factorization structures.

Additionally, if we use the section $\langle \textbf{y}, v \rangle$, $v \in V^*$, to trivialise $T\tilde{N} / \tilde{\mathcal{D}}$, we obtain a correspondence between contact forms
\begin{gather}
\left\langle \frac{(\varphi')^t \textbf{y}}{\langle \textbf{y}, v \rangle}, \theta \right\rangle
\hspace{1cm} \text{ and } \hspace{1cm}
\left\langle \frac{\varphi^t \textbf{x}}{\langle \textbf{x}, g^{-1} v \rangle}, \theta \right\rangle.
\end{gather}

\subsection{Separable CR geometries} \label{CR geoms}
We wish to equip the separable contact geometry $(N,\mathcal{D},\varphi)$ with a strictly pseudo-convex toric CR structure $J$ that separates variables
\begin{align}\label{J initial}
Jd\tau\big|_\mathcal{D}=
\sum_{j=1}^m
\zeta_j([\textbf{x}_j])
d[\textbf{x}_j]\big|_\mathcal{D}.
\end{align}
We have $d[\textbf{x}_j]\in\Omega_N^1([\textbf{x}_j]^*\mathcal{O}(2)\otimes\Lambda^2V_j)$, and thus $\zeta_j \in [\textbf{x}_j]^*\mathcal{O}_{V_j}(-2)\otimes\Lambda^2V_j^* \otimes \mathfrak{h}$, $j=1,\ldots,m$.
Since $\langle \mu, d\tau \rangle|_\mathcal{D} = 0$, for $J$ to be well-defined it must be
\begin{gather}\label{compatibility condition}
0 =
\langle \mu, \zeta_j([\textbf{x}_j]) \rangle =
\left\langle \textbf{x}, \varphi \circ \zeta_j([\textbf{x}_j]) \right\rangle,
\hspace{.6cm}
j=1,\ldots,m.
\end{gather}
Now, for any $j\in\{1,\ldots,m\}$, $\textbf{x}$ takes values in $V_1\otimes\cdots\otimes\langle \textbf{x}_j\rangle\otimes\cdots\otimes V_m$.
Fixing $\textbf{x}_j$ and letting $\textbf{x}_i$, $i\neq j$, vary shows that $\varphi\circ\zeta_j([\textbf{x}_j])$ belongs to 
$\varphi(\mathfrak{h})
\cap
V_1^*\otimes\cdots\otimes \textbf{x}_j^0\otimes\cdots\otimes V_m^*$,
which, because $\varphi$ is a factorization structure, is generically 1-dimensional and spanned by globally defined $\varphi\circ\psi_j([\textbf{x}_j])$ (see \Cref{fs subsection}).
Thus, $\zeta_j([\textbf{x}_j])$ must be a point-wise scalar multiple of $\psi_j([\textbf{x}_j])\in C^\infty(N^0,[\textbf{x}_j]^*\mathcal{O}_{V_j}(d_j)\otimes\mathfrak{h})$ and we write
\begin{align}
\zeta_j([\textbf{x}_j]) = \frac{\psi_j([\textbf{x}_j])}{A_j(\textbf{[x}_j])},
\end{align}
where $A_j$ is a non-vanishing section of $\mathcal{O}_{V_j}(d_j+2)$ over $I_j^0$. \par

For each $j=1,\ldots,m$, we fix an area form on $V_j$ to trivialise $\wedge^2 V_j^*$ and its dual, and the trivialisation of $\mathcal{O}_{V_j}(1)$ as in \Cref{contact geoms}. Then, by differentiating \eqref{compatibility condition} in a trivialisation $\chi \in TN / \mathcal{D}$ we find an essential identity used in many of the following computations,
\begin{gather}\label{crutial identity}
\langle \partial_{x_r} \mu_\chi, \zeta_j(x_j) \rangle = - \langle \mu_\chi, \partial_{x_r} \zeta_j(x_j) \rangle.
\end{gather}
We obtain
\begin{gather}
J \mathcal{L}_{\partial_{x_j}} \langle \mu_\chi, d\tau \rangle |_\mathcal{D} =
\langle \partial_{x_j} \mu_\chi, \zeta_j(x_j) \rangle dx_j |_\mathcal{D},
\end{gather}
and hence, assuming $J^2 = - id_{\mathcal{D}^*}$,
\begin{gather}
J d x_j |_\mathcal{D} =
- \frac{\mathcal{L}_{\partial_{x_j}} \langle \mu_\chi, d\tau \rangle |_\mathcal{D}}{\langle \partial_{x_j} \mu_\chi, \zeta_j(x_j) \rangle}.
\end{gather}
\begin{rem}\label[rem]{cpx str}
To verify $J^2 = -id_{\mathcal{D}^*}$, one observes that $T^*N^0 = \langle dx_1, \ldots, dx_m \rangle \oplus \langle d\tau \rangle$, and thus $\mathcal{D}^*$ is isomorphic with $\langle dx_1,\ldots, dx_m \rangle \oplus \langle d\tau \rangle / \langle \eta_\chi \rangle$. Then, splitting $J$ accordingly,
\begin{align}
\begin{bmatrix}
0 & J_x\\
J_\tau & 0
\end{bmatrix},
\end{align}
gives $J_\tau J_x = - id$ and $J_x J_\tau = - id$. The former holds by the definition of $J$ while for the latter one argues by the fact that in the ring of $m \times m$ matrices, a left inverse is also a right inverse.
\end{rem}

\begin{lemma}
The endomorphism $J: \mathcal{D}^* \to \mathcal{D}^*$ is a (strictly pseudo-convex) CR structure.
\end{lemma}
\begin{proof}
To show that $J$ is integrable we use the characterisation that $J$ is integrable if and only if for any 1-form $\alpha\in T^*N$ such that $\alpha|_\mathcal{D}$ is of type $(1,0)$ we have that $d\alpha|_\mathcal{D}$ is a 2-form in the direct sum of $(2,0)$ and $(1,1)$-forms.
Note that components of $d\tau|_\mathcal{D}+iJd\tau|_\mathcal{D}$ form a frame for such $(1,0)$-forms, hence $\alpha|_\mathcal{D}$ is their combination. We find
\begin{align}
d(d\tau|_\mathcal{D}+iJd\tau|_\mathcal{D})=
i \sum_{r=1}^{m}
d
\left(
	\frac{\psi_r(x_r)}{A_r(x_r)}
	dx_r
\right) \bigg|_\mathcal{D} =
0,
\end{align}
which gives the claim.
\end{proof}

\subsection{Separable Kähler geometries}\label{sep K geoms}

We define \textit{separable Kähler geometry} as the quotient of the separable CR geometry of Sasaki type by a Sasaki-Reeb vector field $X_\beta$ corresponding to a Sasaki structure $\beta \in \mathfrak{h} \subset \mathfrak{cr}_+(N,\mathcal{D},J)$.
In particular, we obtain a toric Kähler geometry with the torus algebra $\mathfrak{h} / \langle \beta \rangle$.

To this end, we note that $\mathcal{L}_{\partial_{x_j} \eta_\beta}$
is a $\tilde{X}_\beta$-basic 1-form, and hence
\begin{gather}
\mathcal{L}_{\partial_{x_j}}\eta_\beta=\pi_\beta^*\theta_j
\end{gather}
for some 1-form $\theta_j$ on the (local) quotient $\pi_\beta: N^0 \to  M^0_\beta$ by $\tilde{X}_\beta$. Thus $(d\eta_\beta,J)$ is the pullback of the Kähler structure
\begin{align}
\omega_\beta=
\sum_{j=1}^mdx_j\wedge\theta_j,\hspace{1cm}
J_\beta dx_j=
-
\frac{A_j(x_j) \theta_j}
	{\langle\partial_{x_j}\mu_\beta,\psi_j(x_j)\rangle},
\hspace{1cm}
\pi_\beta^*\theta_j=
\mathcal{L}_{\partial_{x_j}}
\langle\mu_\beta,d\tau\rangle,
\end{align}
with the associated metric
\begin{align}
g_\beta=
-
\sum_{j=1}^m
\left(
	\frac{\langle\partial_{x_j}\mu_\beta,\psi_j(x_j)\rangle}
		{A_j(x_j)}
	dx_j^2+
	\frac{A_j(x_j)}
		{\langle\partial_{x_j}\mu_\beta,\psi_j(x_j)\rangle}
	\theta_j^2
\right).
\end{align}
Let $\mathfrak{t}=\mathfrak{h}/\langle\beta\rangle$, observe that $\partial_{x_j} \mu_\beta \in \beta^0 \cong \mathfrak{t}^*$ and that it is $\tilde{X}_{\beta}$-invariant.
Furthermore, since $d\tau\text{ mod }\beta$ is  $\tilde{X}_{\beta}$-basic, it descends to a $\mathfrak{t}$-valued 1-form $dt$ on the quotient $M^0_{\beta}$, where the exactness of $dt$ follows from simply-connectedness.
Thus, we may write $\theta_j=\langle\partial_{x_j}\mu_\beta,dt\rangle$, and
\begin{align}\label{sep K-geom}
\begin{split}
g_\beta&=
-
\sum_{j=1}^m
\left(
	\frac{\langle\partial_{x_j}\mu_\beta,\psi_j(x_j)\rangle}
		{A_j(x_j)}
	dx_j^2 +
	\frac{A_j(x_j)}
		{\langle\partial_{x_j}\mu_\beta,\psi_j(x_j)\rangle}
	\langle\partial_{x_j}\mu_\beta,dt\rangle^2
\right)\\
\omega_\beta&=
\sum_{j=1}^mdx_j\wedge\langle\partial_{x_j}\mu_\beta,dt\rangle\\
J_\beta dx_j&=
-
\frac{A_j(x_j)}{\langle\partial_{x_j}\mu_\beta,\psi_j(x_j)\rangle}\langle\partial_{x_j}\mu_\beta,dt\rangle
\hspace{1cm}
J_\beta dt=
\sum_{j=1}^m\frac{\psi_j(x_j)\text{ mod }\beta}{A_j(x_j)}dx_j.
\end{split}
\end{align}

In the case of the Segre-Veronese factorization structure (see \Cref{SV def}),
\begin{align}
\varphi(\mathfrak{h})=
\sum_{j=1}^{k}
ins_j
\left(
	S^{d_j}W_j^*
	\otimes
	\langle\Gamma_j\rangle
\right),
\end{align}
we adapt labelling of variables to grouped slots; the variables corresponding to grouped $j$-slots are $x_{j1},\ldots,x_{jd_j}$.
We fix an affine chart on $V$ so that 
\begin{gather}
\varphi \circ \psi_{ir}(x_{ir})=
ins_i
\left(
	(x_{ir},-1)^{\otimes d_i}\otimes\Gamma_i
\right),
\end{gather}
and recast the identity \eqref{crutial identity} explicitly in the trivialisation by $\langle \mu, \beta \rangle$, being a function, as
\begin{gather}\label{crutial identity S-V}
\langle \partial_{x_{ir}} \mu_\beta, \psi_{js}(x_{js}) \rangle =
\langle \partial_{x_{ir}} \mu_\beta, \psi_{js}(x_{js}) \text{ mod } \beta \rangle =
-
\delta_i^j
\delta_r^s
\frac{\langle\hat{\textbf{x}}_i,\Gamma_i\rangle\Delta_{ir}}
	{\langle\mu,\beta\rangle},
\end{gather}
where
\begin{align}
\hat{\textbf{x}}_i=
\bigotimes_{\substack{j=1\\\ j\neq i}}^k
\bigotimes_{s=1}^{d_j}
(1,x_{js}),
\hspace{1cm}
\Delta_{ir}=
\prod_{\substack{s=1\\s\neq r}}^{d_i}
(x_{ir}-x_{is}),
\end{align}
$\delta_a^b$ is the Kronecker delta, and if $d_i=1$, then $\Delta_{ir}$ is defined to be 1. In the case of the Veronese factorization structure, \eqref{crutial identity S-V} yields also Vandermonde identities (see \Cref{Vandermonde remark}). \par
From now on we work with the general Segre-Veronese factorization structure. The corresponding separable Kähler geometry has the explicit form
\begin{align}\label{geometry tensors}
\begin{split}
g_\beta&=
\sum_{i=1}^k
\sum_{r=1}^{d_i}
\left(
	\frac{\Delta_{ir}\langle\hat{\textbf{x}}_i,\Gamma_i\rangle}
		{A_{ir}(x_{ir})\langle\mu,\beta\rangle}
	dx_{ir}^2 +
	\frac{A_{ir}(x_{ir})\langle\mu,\beta\rangle}
		{\Delta_{ir}\langle\hat{\textbf{x}}_i,\Gamma_i\rangle}
	\langle\partial_{x_{ir}}\mu_\beta,dt\rangle^2
\right)\\
\omega_\beta&=
\sum_{i=1}^k
\sum_{r=1}^{d_i}
dx_{ir}\wedge\langle\partial_{x_{ir}}\mu_\beta,dt\rangle\\
J_\beta dx_{ir}&=
\frac{A_{ir}(x_{ir})\langle\mu,\beta\rangle}
	{\Delta_{ir}\langle\hat{\textbf{x}}_i,\Gamma_i\rangle}
\langle\partial_{x_{ir}}\mu_\beta,dt\rangle
\hspace{1cm}
J_\beta dt=
\sum_{i=1}^k
\sum_{r=1}^{d_i}
\frac{\psi_{ir}(x_{ir})\text{ mod }\beta}
	{A_{ir}(x_{ir})}
dx_{ir}.
\end{split}
\end{align}
Let $T\in C^\infty(M^0_\beta,TM\otimes\mathfrak{t}^*)$ be the angular vector fields dual to $dt$, i.e. $dt(T)=Id_\mathfrak{t}$ and $\langle dt,T\rangle=Id_{TM^0_\beta}$. Using \eqref{crutial identity S-V} and the identity $J_x J_\tau = - id$ from \Cref{cpx str} we find
\begin{align}\label{geometry inverse tensors}
\begin{split}
g^{-1}_\beta&=
\sum_{i=1}^k
\sum_{r=1}^{d_i}
\left(
	\frac{A_{ir}(x_{ir})\langle\mu,\beta\rangle}
		{\Delta_{ir}\langle\hat{\textbf{x}}_i,\Gamma_i\rangle}
	\left(\partial_{ir}\right)^2 +
	\frac{\langle \mu, \beta \rangle}{A_{ir}(x_{ir}) \Delta_{ir} \langle\hat{\textbf{x}}_i,\Gamma_i\rangle}
	\langle\psi_{ir}(x_{ir})\text{ mod }\beta,T\rangle^2
\right)\\
\omega^{-1}_\beta&=
\sum_{i=1}^k
\sum_{r=1}^{d_i}
\frac{\langle\mu,\beta\rangle
	\partial_{ir}
	\wedge
	\langle \psi_{ir}(x_{ir})\text{ mod }\beta,T\rangle}
		{\Delta_{ir}\langle\hat{\textbf{x}}_i,\Gamma_i\rangle}\\
J_\beta T&=
-
\sum_{i=1}^k
\sum_{r=1}^{d_i}
\frac{A_{ir}(x_{ir})\langle\mu,\beta\rangle}
	{\Delta_{ir}\langle\hat{\textbf{x}}_i,\Gamma_i\rangle}
(\partial_{x_{ir}}\mu_\beta)
\partial_{x_{ir}}
\hspace{1cm}
J_\beta \partial_{x_{ir}}=
-
\frac{\langle\psi_{ir}(x_{ir})\text{ mod }\beta, T\rangle}
	{A_{ir}(x_{ir})}.
\end{split}
\end{align}

\begin{rem}
To maintain clarity, we wish to remark on the trivialisation of a factorization curve $\varphi \circ \psi_j: \mathbb{P}(V_j) \to \mathcal{O}_{V_j}(d_j)\otimes V^*$ of a Segre-Veronese factorization structure.
By definition, $\varphi \circ \psi_j$ is the $\varphi \circ [\psi_j]$-pullback of the tautological section $\tau$ of $\mathcal{O}_{V^*}(1) \otimes V^*$, where the latter is trivialised by $\otimes_{j=1}^m \langle \textbf{x}_j, e_j^1 \rangle$ as in \Cref{contact geoms}, and  $\varphi \circ [\psi_j](\ell) = ins_j \left( (\ell^0)^{\otimes d_j} \otimes \langle \Gamma_j \rangle \right)$, where $\ell^0$ is really the image of $\ell$ under the area form $\epsilon_j$ fixed in \Cref{CR geoms}, which comes from a trivialisation of $\bigwedge^2 V_j^*$.
It might be illuminating to view the pullback $\varphi \circ \psi_j$ as being valued in $\mathcal{O}_{V_j}(d_j) \otimes \langle \Gamma_j \rangle \otimes V^*$, since its value at $\ell$ is the canonical inclusion of $ins_j \left( (\ell^0)^{\otimes d_j} \otimes \langle \Gamma_j \rangle \right)$ into $V^*$. Now it should be easy to see that, in the case of a Segre-Veronese factorization structure, $\varphi \circ [\psi_j]$-pullback of the trivialised tautological section
$\tau / \langle \otimes_{j=1}^k \textbf{x}_j^{\otimes d_j}, \otimes_{j=1}^k (e_j^1)^{\otimes d_j} \rangle$ is
\begin{gather}
\frac{ins_j \left( (\epsilon_j \textbf{x}_j)^{\otimes d_j} \otimes \langle \Gamma_j \rangle \right)}
{\left\langle ins_j \left( \textbf{x}_j^{\otimes d_j} \otimes \langle \Gamma_j \rangle \right), \otimes_{j=1}^k (e_j^1)^{\otimes d_j} \right\rangle} =
ins_j \left( (x_j,-1)^{\otimes d_j} \otimes \Gamma_j \right),
\end{gather}
since $\textbf{x}_i$, $i \neq j$, are constant, where
\begin{gather}
\Gamma_j =
\frac{\langle \Gamma_j \rangle}{\left\langle \langle \Gamma_j \rangle, \otimes_{\substack{i=1 \\ i \neq j }}^k (e_i^1)^{\otimes d_i} \right\rangle}.
\end{gather}
\end{rem}

We study the relation between geometries arising from isomorphic factorization structures.
Note, similarly to \Cref{contact geoms}, that individual coordinate $\mathbb{P}GL_2(\mathbb{R})$-transformations of $(1,x_j)$, $j=1,\ldots,m$, in the momentum map
\begin{gather}
\mu_\beta =
\frac{\varphi^t \otimes_{j=1}^m (1,x_j)}{\langle \varphi^t \otimes_{j=1}^m (1,x_j), \beta \rangle}
\end{gather}
correspond to isomorphisms of factorization structures.
Consequently, a direct computation,
\begin{gather}
\omega_\beta =
\sum_{j=1}^m dy_j \wedge \partial_{y_j} \left\langle \frac{\varphi_2^t \otimes_{j=1}^m (1,y_j)}{\langle \varphi_2^t \otimes_{j=1}^m (1,y_j), \beta \rangle} dt \right\rangle =
\sum_{j=1}^m dx_j \wedge \partial_{x_j} \left\langle \frac{\varphi_1^t \otimes_{j=1}^m (1,x_j)}{\langle \varphi_1^t \otimes_{j=1}^m (1,x_j), \beta \rangle} dt \right\rangle,
\end{gather}
shows that the individual projective change of coordinates $x_j$, $j=1,\ldots,m$, in the symplectic form $\omega_\beta$ from \eqref{sep K-geom} results in the symplectic form of the separable Kähler geometry corresponding to the isomorphic factorization structure and the same $\beta$.
Thus, for a fixed $\beta$, separable Kähler geometries corresponding to isomorphic factorization structures have the same symplectic forms.
One easily verifies
\begin{proposition}\label[proposition]{coordinate change general}
The projective change of coordinates
\begin{gather}
x_{ir} = \frac{b_i + d_i y_{ir}}{a_i + c_i y_{ir}}, \hspace{.2cm}
r=1,\ldots,d_i,
i=1,\ldots,k,
\end{gather}
associated with matrices $g_i$ from \eqref{g_j}, in the separable Kähler geometry corresponding to a Segre-Veronese factorization structure \eqref{geometry tensors}, results in
\begin{gather}
\begin{split}
g_\beta&=
\sum_{i=1}^k
\sum_{r=1}^{d_i}
\left(
	\frac{\Delta_{ir}\langle\hat{\textbf{y}}_i,\hat{g}_i \Gamma_i\rangle}
		{\tilde{A}_{ir}(y_{ir})\langle\mu,\beta\rangle}
	dy_{ir}^2 +
	\frac{\tilde{A}_{ir}(y_{ir})\langle\mu,\beta\rangle}
		{\Delta_{ir}\langle\hat{\textbf{y}}_i,\hat{g}_i \Gamma_i\rangle}
	\langle\partial_{y_{ir}}\mu_\beta,dt\rangle^2
\right)\\
\omega_\beta&=
\sum_{i=1}^k
\sum_{r=1}^{d_i}
dy_{ir}\wedge\langle\partial_{y_{ir}}\mu_\beta,dt\rangle\\
J_\beta dy_{ir}&=
\frac{\tilde{A}_{ir}(y_{ir})\langle\mu,\beta\rangle}
	{\Delta_{ir}\langle\hat{\textbf{y}}_i,\hat{g}_i \Gamma_i\rangle}
\langle\partial_{y_{ir}}\mu_\beta,dt\rangle
\hspace{1cm}
J_\beta dt=
\sum_{i=1}^k
\sum_{r=1}^{d_i}
\frac{\psi_{ir}(y_{ir})\text{ mod }\beta}
	{\tilde{A}_{ir}(y_{ir})}
dy_{ir},
\end{split}
\end{gather}
where $\mu = \varphi_2^t \left( \otimes_{i=1}^k \otimes_{r=1}^{d_i} (1,y_{ir}) \right)$, $\hat{g}_i = \otimes_{\substack{j=1 \\ j \neq i}}^k (g_j)^{\otimes d_j}$, $\Delta_{ir} = \prod_{\substack{q=1 \\ q \neq r}}^{d_i} (y_{ir} - y_{iq})$, and
\begin{gather}
\tilde{A}_{ir}(y_{ir}) =
\frac{(a_i+c_iy_{ir})^{d_i+2}}{(a_id_i-b_ic_i)^{d_i+1}} A_{ir}\left( \frac{b_i+d_iy_{ir}}{a_i+c_iy_{ir}} \right).
\end{gather}
Moreover, these projective transformations are encoded in isomorphisms of Segre-Veronese factorization structures $(g_1)^{\otimes d_1} \otimes \cdots \otimes (g_k)^{\otimes g_k}: V^* \to V^*$, where $g_i \in GL(W_j)$, $i=1,\ldots,k$.
\end{proposition}
\begin{proof}
The proof is a direct verification. We offer here only an essential computation,
\begin{gather}
\left\langle \partial_{x_{jp}} \frac{\varphi_1^t \otimes_{i=1}^k \otimes_{r=1}^{d_i} (1,x_{ir})}{\langle \varphi_1^t \otimes_{i=1}^k \otimes_{r=1}^{d_i} (1,x_{ir}), \beta \rangle}, \psi_{jp}(x_{jp}) \right\rangle = \nonumber \\
\frac{(a_jd_j-b_jc_j)^{d_j-1}}{(a_j+c_jy_{jp})^{d_j-2}} \left\langle \partial_{y_{jp}} \frac{\otimes_{i=1}^k \otimes_{r=1}^{d_i} g_i^t (1,y_{ir})}{\langle \varphi_1^t \otimes_{i=1}^k \otimes_{r=1}^{d_i}  g_i^t (1,y_{ir}), \beta \rangle}, (g_j^{-1}(y_{jp},-1))^{\otimes d_j} \otimes \Gamma_j \right\rangle = \nonumber \\
\frac{(a_jd_j-b_jc_j)^{d_j-1}}{(a_j+c_jy_{jp})^{d_j-2}} \left\langle \partial_{y_{jp}} \frac{\mu}{\langle \mu, \beta \rangle}, \psi_j(y_{jp}) \right\rangle.
\end{gather}
\end{proof}

We prepare the following lemma for further use.

\begin{lemma}\label[lemma]{computational lemma}
For each $i\in\{1,\ldots,k\}$ and $r\in\{1,\ldots,d_i\}$, we have
\begin{gather}
\sum_{j=1}^k\sum_{s=1}^{d_j}
\left\langle
	\partial_{x_{ir}}\partial_{x_{js}}\textbf{x},
	\frac{\varphi\circ\psi_{js}(x_{js})\text{ mod }\varphi(\beta)}
	{\Delta_{js}
		\langle\hat{\textbf{x}}_j,\Gamma_j\rangle}
\right\rangle= \nonumber \\
\sum_{j=1}^k\sum_{s=1}^{d_j}
\left\langle
	\partial_{x_{ir}}\partial_{x_{js}}\textbf{x},
	\frac{\varphi\circ\psi_{js}(x_{js})}
		{\Delta_{js}
			\langle\hat{\textbf{x}}_j,\Gamma_j\rangle}
\right\rangle=
-
\sum_{j=1}^k
d_j\frac{\partial_{x_{ir}}
		\langle\hat{\textbf{x}}_j,\Gamma_j\rangle}
	{\langle\hat{\textbf{x}}_j,\Gamma_j\rangle} -
\sum_{\substack{s=1\\ s\neq r}}^{d_i}
\frac{\partial_{x_{ir}}\Delta_{is}}
	{\Delta_{is}}
\end{gather}
In addition,
\begin{align}\label{Delta identity}
\sum_{\substack{s=1\\ s\neq r}}^{d_i}
\frac{\partial_{x_{ir}}\Delta_{is}}
	{\Delta_{is}}=
\frac{\partial_{x_{ir}}\Delta_{ir}}
	{\Delta_{ir}}.
\end{align}
\end{lemma}
\begin{proof}
By differentiating \eqref{crutial identity S-V} one observes
\begin{align}
\langle
	\partial_{x_{ir}}\partial_{x_{js}}\textbf{x},
	\varphi\circ\psi_{js}(x_{js})
\rangle=
-
\Delta_{js}
\langle\partial_{x_{ir}}\hat{\textbf{x}}_j,\Gamma_j\rangle -
\delta_i^j(1-\delta_r^s)
\langle\hat{\textbf{x}}_j,\Gamma_j\rangle
\partial_{x_{ir}}\Delta_{js}.
\end{align}
Note that the first term is zero if $i=j$. Now divide by $\Delta_{js}\langle\hat{\textbf{x}}_j,\Gamma_j\rangle$ and sum over $j,s$.
The identity \eqref{Delta identity} follows from the fact that $(\prod_{s\neq r}\Delta_{is})/\Delta_{ir}$ is independent of $x_{ir}$, and thus by taking the $\partial_{x_{ir}}$-logarithmic derivative we prove the claim.
\end{proof}

\subsection{Examples} \label{examples of geoms}
The separable Kähler geometry \eqref{sep K-geom} corresponding to the product Segre-Veronese factorization structure (see \Cref{one intersection example}) reads
\begin{align}\label{pSV geom}
\begin{split}
g_\beta&=
\sum_{i=1}^k
\sum_{r=1}^{d_i}
\left(
	\frac{\Delta_{ir}}
		{A_{ir}(x_{ir})\langle\mu,\beta\rangle}
	dx_{ir}^2 +
	\frac{A_{ir}(x_{ir})\langle\mu,\beta\rangle}
		{\Delta_{ir}}
	\langle\partial_{x_{ir}}\mu_\beta,dt\rangle^2
\right)\\
\omega_\beta&=
\sum_{i=1}^k
\sum_{r=1}^{d_i}
dx_{ir}\wedge\langle\partial_{x_{ir}}\mu_\beta,dt\rangle\\
J_\beta dx_{ir}&=
\frac{A_{ir}(x_{ir})\langle\mu,\beta\rangle}
	{\Delta_{ir}}
\langle\partial_{x_{ir}}\mu_\beta,dt\rangle
\hspace{1cm}
J_\beta dt=
\sum_{i=1}^k
\sum_{r=1}^{d_i}
\frac{\psi_{ir}(x_{ir})\text{ mod }\beta}
	{A_{ir}(x_{ir})}
dx_{ir},
\end{split}
\end{align}
where the first basis vector, used to trivialise $\textbf{x}_j$ as in \Cref{contact geoms}, of each vector space $W_r^*$, $r=1,\ldots,k$, is chosen to lie on $a^r \subset W_r^*$ (see \eqref{product tensors} in \Cref{one intersection example}), i.e. $a^r = \langle (1,0) \rangle$. 

Depending on the partition $m = d_1 + \cdots + d_k$, the geometry \eqref{pSV geom} interpolates between twisted orthotoric geometry \cite{apostolov_cr_2020,apostolov_hamiltonian_2006,apostolov_hamiltonian_2004}, corresponding to the Veronese factorization structure ($k=1$), and twisted product of Riemann surfaces \cite{apostolov_cr_2020}, associated with the product Segre factorization structure ($d_1=\cdots=d_k$=1).

The function $\langle \mu, \beta \rangle$ is thought of as a twisting function, and when it is constant, the corresponding geometries are viewed as 'untwisted'.

When $\varphi(\beta)=(1,0)^{\otimes m}$ and the partition is $d_1=\cdots=d_k=1$, the family of Kähler structures \eqref{pSV geom} reduces to the product of Riemann surfaces, hence the terminology 'product Segre'.
Similarly, for a general partition and $\varphi(\beta)=(1,0)^{\otimes m}$, \eqref{pSV geom} specialises to the product of orthotoric geometries, each being a separable geometry corresponding to the Veronese factorization structure, hence 'the product Segre-Veronese'.

When $m=2$, meaning that the separable Kähler geometries are 4-dimensional, the classification of 2-dimensional factorization structures from \cite{pucekfs} (see \Cref{fs subsection}) shows that, up to isomorphism, only the 2-dimensional Veronese and Segre factorization structures occur.
These together correspond to (local) ambitoric geometries \cite{apostolov2016ambitoric,apostolov2015ambitoric,apostolov2017levi}, which, when compactified, classify extremal toric 4-orbifold with the second Betti number two.

Locally, extremal separable Kähler geometries associated with the product Segre-Veronese structure not only unify all known explicit extremal toric Kähler geometries but also yield a wealth of new extremal examples, all accessible through a uniform approach. This statement remains valid in the compact case as well; however, the compactification of separable geometries is not considered in this paper.

\subsection{Laplace operator}
We evaluate the Laplace operator $\Delta$ on a $\mathfrak{t}$-invariant function $f$, i.e., a function depending only on variables  $x_{ir}$, $i=1,\ldots,k$, $r=1,\ldots,d_i$.
To this end, we use the standard expression for the Laplace operator acting on functions on a Kähler manifold, $\Delta f = - \langle \omega^\sharp, dJdf \rangle$, where $\omega^\sharp = - \omega$, and \eqref{geometry tensors} and \eqref{geometry inverse tensors};
\begin{align}
d^cf=&
\sum_{i=1}^k\sum_{r=1}^{d_i}
\frac{A_{ir}(x_{ir})\langle\mu,\beta\rangle}
	{\Delta_{ir}
		\langle\hat{\textbf{x}}_i,\Gamma_i\rangle}
(\partial_{x_{ir}}f)
\langle\partial_{x_{ir}}\mu_\beta,dt\rangle\\
dd^cf=&
\sum_{i=1}^k\sum_{r=1}^{d_i}\sum_{j=1}^k\sum_{s=1}^{d_i}
\partial_{x_{js}}
\left(
	\frac{A_{ir}(x_{ir})}
		{\langle\hat{\textbf{x}}_i,\Gamma_i\rangle}
	(\partial_{x_{ir}}f)
\right)
dx_{js}\wedge
\left\langle\
	\frac{\langle\mu,\beta\rangle\partial_{x_{ir}}\mu_\beta}
		{\Delta_{ir}},
	dt
\right\rangle+\nonumber\\
&+\sum_{i=1}^k\sum_{r=1}^{d_i}
\frac{A_{ir}(x_{ir})}
	{\langle\hat{\textbf{x}}_i,\Gamma_i\rangle}
(\partial_{x_{ir}}f)
\sum_{j=1}^k\sum_{s=1}^{d_j}
dx_{js}\wedge
\left\langle
	\partial_{x_{js}}
	\frac{\langle\mu,\beta\rangle\partial_{x_{ir}}\mu_\beta}
		{\Delta_{ir}},
	dt
\right\rangle\\
\Delta f = &
\omega^{-1}_\beta(dd^cf)=
-
\sum_{i=1}^k\sum_{r=1}^{d_i}
\frac{\langle\mu,\beta\rangle}
	{\Delta_{ir}
		\langle\hat{\textbf{x}}_i,\Gamma_i\rangle}
\partial_{x_{ir}}
\left(
	A_{ir}(x_{ir})(\partial_{x_{ir}}f)
\right)+\nonumber\\
& +
\sum_{i=1}^k\sum_{r=1}^{d_i}
\frac{A_{ir}(x_{ir})}
	{\langle\hat{\textbf{x}}_i,\Gamma_i\rangle}
(\partial_{x_{ir}}f)
\sum_{j=1}^k\sum_{s=1}^{d_j}
\frac{\langle\mu,\beta\rangle}
	{\Delta_{js}\langle\hat{\textbf{x}}_j,\Gamma_j\rangle}
\left\langle
	\partial_{x_{js}}
	\frac{\langle\mu,\beta\rangle\partial_{x_{ir}}\mu_\beta}
		{\Delta_{ir}},
	\psi_{js}(x_{js})\text{ mod }\beta
\right\rangle
\end{align}

\begin{rem}
Note that
\begin{align}
\begin{split}
&\sum_{i=1}^k\sum_{r=1}^{d_i}
\frac{A_{ir}(x_{ir})}
	{\langle\hat{\textbf{x}}_i,\Gamma_i\rangle}
(\partial_{x_{ir}}f)
\sum_{j=1}^k\sum_{s=1}^{d_j}
\frac{\langle\mu,\beta\rangle}
	{\Delta_{js}\langle\hat{\textbf{x}}_j,\Gamma_j\rangle}
\left\langle
	\partial_{x_{js}}
	\frac{\langle\mu,\beta\rangle\partial_{x_{ir}}\mu_\beta}
		{\Delta_{ir}},
	\psi_{js}(x_{js})\text{ mod }\beta
\right\rangle=\\
&\sum_{i=1}^k\sum_{r=1}^{d_i}
A_{ir}(x_{ir})
(\partial_{x_{ir}}f)
\sum_{j=1}^k\sum_{s=1}^{d_j}
\frac{\langle\mu,\beta\rangle}
	{\Delta_{js}\langle\hat{\textbf{x}}_j,\Gamma_j\rangle}
\left\langle
	\partial_{x_{js}}
	\frac{\langle\mu,\beta\rangle\partial_{x_{ir}}\mu_\beta}
		{\Delta_{ir}
			\langle\hat{\textbf{x}}_i,\Gamma_i\rangle},
	\psi_{js}(x_{js})\text{ mod }\beta
\right\rangle.
\end{split}
\end{align}
\end{rem}
To simplify the expression for $\Delta f$, we start with
\begin{gather}
\left\langle
	\partial_{x_{js}}
	(\langle\mu,\beta\rangle\partial_{x_{ir}}\mu_\beta),
	\psi_{js}(x_{js})\text{ mod }\beta
\right\rangle=\nonumber\\
\left\langle
	\partial_{x_{ir}}\partial_{x_{js}}\textbf{x}-
	\frac{\langle\partial_{x_{ir}}\mu,\beta\rangle
			\partial_{x_{js}}\textbf{x}}
		{\langle\mu,\beta\rangle},
	\varphi\circ\psi_{js}(x_{js})\text{ mod }\varphi(\beta)
\right\rangle=\nonumber\\
\left\langle
	\partial_{x_{ir}}\partial_{x_{js}}\textbf{x},
	\varphi\circ\psi_{js}(x_{js})\text{ mod }\varphi(\beta)
\right\rangle
+
\frac{\langle\partial_{x_{ir}}\mu,\beta\rangle
		\langle\hat{\textbf{x}}_j,\Gamma_j\rangle
		\Delta_{js}}
	{\langle\mu,\beta\rangle},
\end{gather}
so \Cref{computational lemma} implies
\begin{gather}
\sum_{j=1}^k\sum_{s=1}^{d_j}
\left\langle
	\partial_{x_{js}}
	(\langle\mu,\beta\rangle
	\partial_{x_{ir}}\mu_\beta),
	\frac{\psi_{js}(x_{js})\text{ mod }\varphi(\beta)}
		{\Delta_{js}
			\langle\hat{\textbf{x}}_j,\Gamma_j\rangle}
\right\rangle=\nonumber\\
-
\sum_{j=1}^kd_j
\frac{\langle\mu,\beta\rangle}
	{\langle\hat{\textbf{x}}_j,\Gamma_j\rangle}
\partial_{x_{ir}}
\frac{\langle\hat{\textbf{x}}_j,\Gamma_j\rangle}
	{\langle\mu,\beta\rangle} -
\sum_{\substack{s=1\\ s\neq r}}^{d_i}
\frac{\partial_{x_{ir}}
		\Delta_{is}}
	{\Delta_{is}}.
\label{quick Laplace calculation}
\end{gather}
Hence
\begin{gather}
\sum_{j=1}^k\sum_{s=1}^{d_j}
\left\langle
	\Delta_{ir}
	\partial_{x_{js}}
	\frac{\langle\mu,\beta\rangle
			\partial_{x_{ir}}\mu_\beta}
		{\Delta_{ir}
			\langle\hat{\textbf{x}}_i,\Gamma_i\rangle},
	\frac{\psi_{js}(x_{js})\text{ mod }\varphi(\beta)}
		{\Delta_{js}
			\langle\hat{\textbf{x}}_j,\Gamma_j\rangle}
\right\rangle= \nonumber \\
-
\frac{1}
	{\langle\hat{\textbf{x}}_i,\Gamma_i\rangle}
\sum_{j=1}^k
d_j
\frac{\langle\mu,\beta\rangle}
	{\langle\hat{\textbf{x}}_j,\Gamma_j\rangle}
\partial_{x_{ir}}
\frac{\langle\hat{\textbf{x}}_j,\Gamma_j\rangle}
	{\langle\mu,\beta\rangle};
\end{gather}
when the numerator of the first slot is differentiated, the contraction is given by a $\langle \hat{\textbf{x}}_i, \Gamma_i \rangle^{-1}$-multiple of \eqref{quick Laplace calculation}, while if the denominator is differentiated, the contraction is nonzero iff $j=i$ and $s=r$, which leaves us with
$\frac{1}{\langle\hat{\textbf{x}}_i,\Gamma_i\rangle}
\frac{\partial_{x_{ir}}\Delta_{ir}}{\Delta_{ir}}$ as $\langle\hat{\textbf{x}}_i,\Gamma_i\rangle$ is constant in this case (see \eqref{Delta identity}). \par
We conclude
\begin{gather}
\sum_{j=1}^k\sum_{s=1}^{d_j}
\frac{\langle\mu,\beta\rangle}
	{\Delta_{js}
		\langle\hat{\textbf{x}}_j,\Gamma_j\rangle}
\left\langle
	\partial_{x_{js}}
	\frac{\langle\mu,\beta\rangle
			\partial_{x_{ir}}\mu_\beta}
		{\Delta_{ir}
			\langle\hat{\textbf{x}}_i,\Gamma_i\rangle},
	\psi_{js}(x_{js})\text{ mod }\beta
\right\rangle= \nonumber \\
=
-
\frac{\langle\mu,\beta\rangle}
	{\Delta_{ir}
		\langle\hat{\textbf{x}}_i,\Gamma_i\rangle}
\sum_{j=1}^k
d_j
\frac{\langle\mu,\beta\rangle}
	{\langle\hat{\textbf{x}}_j,\Gamma_j\rangle}
\partial_{x_{ir}}
\frac{\langle\hat{\textbf{x}}_j,\Gamma_j\rangle}
	{\langle\mu,\beta\rangle}= \nonumber \\
=
-
\frac{\langle\mu,\beta\rangle}
	{\Delta_{ir}
		\langle\hat{\textbf{x}}_i,\Gamma_i\rangle}
\frac{\partial_{x_{ir}}
		H}
	{H},
\end{gather}
where
\begin{align}
H=
\frac{\prod_{j=1}^k
		\langle\hat{\textbf{x}}_j,\Gamma_j\rangle^{d_j}}
	{\langle\mu,\beta\rangle^m}.
\end{align}
Finally
\begin{gather}
\Delta f=\nonumber\\
-
\sum_{i=1}^k\sum_{r=1}^{d_i}
\frac{\langle\mu,\beta\rangle}
	{\Delta_{ir}
		\langle\hat{\textbf{x}}_i,\Gamma_i\rangle}
\partial_{x_{ir}}
\left(
	A_{ir}(x_{ir})(\partial_{x_{ir}}f)
\right) -
\sum_{i=1}^k\sum_{r=1}^{d_i}
A_{ir}(x_{ir})
(\partial_{x_{ir}}f)
\frac{ \langle\mu,\beta\rangle }{ \Delta_{ir} \langle\hat{\textbf{x}}_i,\Gamma_i\rangle }
\frac{\partial_{x_{ir}}	H}{H}
=\nonumber\\
-
\sum_{i=1}^k\sum_{r=1}^{d_i}
\frac{\langle\mu,\beta\rangle}
	{
		\Delta_{ir}
		\langle\hat{\textbf{x}}_i,\Gamma_i\rangle H
	}
\partial_{x_{ir}}
\bigg(
	A_{ir}(x_{ir})(\partial_{x_{ir}}f)H
\bigg).\label{Laplace separable}
\end{gather}

\subsection{Scalar curvature}
Using standard formulae \cite{besse2007einstein}, we compute the scalar curvature of a separable Kähler geometry as the value of the Laplace operator on the negative of the Ricci potential, $Scal(g) = \frac{1}{2} \Delta ln \frac{v_g}{v_0}$, where $v_g$ and $v_0$ are respectively the symplectic and a holomorphic volume forms.

For the symplectic volume we have
\begin{align}
\wedge^m
\omega_\beta=
det(\partial_{x_{ir}}\mu_\beta)
\wedge_{i,r}dx_{ir}\wedge\wedge^mdt,
\end{align}
while the holomorphic volume determined by $dt-iJ_\beta dt$ is
\begin{align}
\frac{det(\psi_{ir}(x_{ir})\text{ mod }\beta)}
	{\prod_{i=1}^k\prod_{r=1}^{d_i}
		A_{ir}(x_{ir})}
\wedge_{i,r}dx_{ir}\wedge\wedge^mdt.
\end{align}
Furthermore, \eqref{crutial identity S-V} implies
\begin{gather}
det(\partial_{x_{ir}}\mu_\beta)\hspace{.1cm}
det(\psi_{ir}(x_{ir})\text{ mod }\beta)= \nonumber \\
=
det
\left(
	(\partial_{x_{ir}}\mu_\beta)
	(\psi_{ir}(x_{ir})\text{ mod }\beta)
\right)=
(-1)^m
\frac{\prod_{i=1}^k
		\langle\hat{\textbf{x}}_i,\Gamma_i\rangle^{d_i}}
	{\langle\mu,\beta\rangle^m}
\prod_{i=1}^k\prod_{r=1}^{d_i}
\Delta_{ir}.
\end{gather}
Therefore, the Ricci potential is $-1/2$ times the logarithm of the ratio of these volumes, the last being
\begin{align}\label{Ricci potential}
(-1)^m
\frac{
	det\left(
			\langle\mu,\beta\rangle
			\partial_{x_{ir}}\mu_\beta
		\right)^2
	\left(
		\prod_{i=1}^k\prod_{r=1}^{d_i}
		A_{ir}(x_{ir})
	\right)
}
{
	\langle\mu,\beta\rangle^m
	\left(
		\prod_{j=1}^k
		\langle\hat{\textbf{x}}_j,\Gamma_j\rangle^{d_j}
	\right)
	\left(
		\prod_{i=1}^k\prod_{r=1}^{d_i}
		\Delta_{ir}
	\right)
},
\end{align}
where we changed the determinant $det(\partial_{x_{ir}}\mu_\beta)$ by $\langle\mu,\beta\rangle$ for convenience in the following computations. In order to use the formula \eqref{Laplace separable} for Laplace and calculate the scalar curvature, we need to understand $\partial_{x_{ir}}$-log derivatives of \eqref{Ricci potential}. We have
\begin{align}
\partial_{x_{ir}}
ln
\prod_{j=1}^k\prod_{s=1}^{d_j}
A_{js}(x_{js})=&
\frac{\partial_{x_{ir}}
		A_{ir}(x_{ir})}
	{A_{ir}(x_{ir})}\\
\partial_{x_{ir}}
ln
\langle\mu,\beta\rangle=&
\frac{\langle\partial_{x_{ir}}\mu,\beta\rangle}
	{\langle\mu,\beta\rangle}\\
\partial_{x_{ir}}
ln
\prod_{j=1}^k
\langle\hat{\textbf{x}}_j,\Gamma_j\rangle^{d_j}=&
\sum_{j=1}^k
\frac{\langle
		\partial_{x_{ir}}\hat{\textbf{x}}_j,
		\Gamma_j
	\rangle}
	{\langle\hat{\textbf{x}}_j,\Gamma_j\rangle}\\
\partial_{x_{ir}}
ln
\prod_{j=1}^k\prod_{s=1}^{d_j}
\Delta_{js}=&
\sum_{s=1}^{d_j}
\frac{\partial_{x_{ir}}\Delta_{is}}
	{\Delta_{is}}=
2\sum_{\substack{s=1\\ s\neq r}}^{d_j}
\frac{\partial_{x_{ir}}\Delta_{is}}
	{\Delta_{is}},
\end{align}
where the last equality follows from \eqref{Delta identity}. The last factor to $\partial_{x_{ir}}$-log differentiate is\newline
$
det\left(
		\langle\mu,\beta\rangle
		\partial_{x_{ir}}\mu_\beta
	\right)
$. To this end, we use the general formula
$d\hspace{.1cm}
log\hspace{.1cm}
det\hspace{.1cm}
A=
tr
A^{-1}dA$
and \eqref{crutial identity S-V} to get
\begin{gather}
\partial_{x_{ir}}
ln\hspace{.1cm}
det\left(
		\langle\mu,\beta\rangle
		\partial_{x_{ir}}\mu_\beta
	\right)=
\sum_{j=1}^k\sum_{s=1}^{d_j}
\left\langle \partial_{x_{ir}}\left(\langle\mu,\beta\rangle
	\partial_{x_{js}}\mu_\beta\right),
	- \frac{\psi_{js}(x_{js})\text{ mod }\beta}
		{\langle\hat{\textbf{x}}_j,\Gamma_j\rangle
		\Delta_{js}}
\right\rangle= \nonumber \\
\sum_{j=1}^k\sum_{s=1}^{d_j}
\left\langle\partial_{x_{ir}}\partial_{x_{js}}\textbf{x}-
	\frac{\langle\partial_{x_{js}}\mu,\beta\rangle
		\partial_{x_{ir}}\textbf{x}}
	{\langle\mu,\beta\rangle},
	- \frac{\varphi\circ\psi_{js}(x_{js})\text{ mod }\beta}
		{\langle\hat{\textbf{x}}_j,\Gamma_j\rangle
			\Delta_{js}}
\right\rangle= \nonumber \\
-
\sum_{j=1}^k\sum_{s=1}^{d_j}
\left\langle
	\partial_{x_{ir}}\partial_{x_{js}}\textbf{x},
	\frac{\varphi\circ\psi_{js}(x_{js})\text{ mod }\beta}
	{\langle\hat{\textbf{x}}_j,\Gamma_j\rangle
		\Delta_{js}}
\right\rangle-
\frac{\langle\partial_{x_{ir}}\mu,\beta\rangle}
	{\langle\mu,\beta\rangle},
\end{gather}
which, by \Cref{computational lemma}, is
\begin{align}
\sum_{j=1}^k
d_j\frac{\partial_{x_{ir}}
	\langle\hat{\textbf{x}}_j,\Gamma_j\rangle}
{\langle\hat{\textbf{x}}_j,\Gamma_j\rangle}+
\sum_{\substack{s=1\\ s\neq r}}^{d_i}
\frac{\partial_{x_{ir}}\Delta_{is}}
{\Delta_{is}}-
\frac{\langle\partial_{x_{ir}}\mu,\beta\rangle}
	{\langle\mu,\beta\rangle}.
\end{align}
Moreover, since
\begin{gather}
\partial_{x_{ir}}
ln\hspace{.1cm}
det^2\left(
		\langle\mu,\beta\rangle
		\partial_{x_{ir}}\mu_\beta
	\right)=
2\sum_{j=1}^k
d_j\frac{\partial_{x_{ir}}
	\langle\hat{\textbf{x}}_j,\Gamma_j\rangle}
{\langle\hat{\textbf{x}}_j,\Gamma_j\rangle}+
2\sum_{\substack{s=1\\ s\neq r}}^{d_i}
\frac{\partial_{x_{ir}}\Delta_{is}}
{\Delta_{is}}-
2\frac{\langle\partial_{x_{ir}}\mu,\beta\rangle}
	{\langle\mu,\beta\rangle}= \nonumber \\
=
\partial_{x_{ir}}
ln
\frac{\left(\prod_{j=1}^k\langle\hat{\textbf{x}}_j,\Gamma_j\rangle^{d_j}
		\right)^2
		\left(\prod_{j=1}^k\prod_{s=1}^{d_j}
		\Delta_{js}\right)}
	{\langle\hat{\textbf{x}},\beta\rangle^2}
\end{gather}
for all $i=1,\ldots,k$ and $r=1,\ldots,d_i$, where the second sum was split and sum via \eqref{Delta identity}, we obtain that
$
det^2\left(
		\langle\mu,\beta\rangle
		\partial_{x_{ir}}\mu_\beta
	\right)
$
is a constant multiple of
\begin{align}
\frac{\left(\prod_{j=1}^k\langle\hat{\textbf{x}}_j,\Gamma_j\rangle^{d_j}
		\right)^2
		\left(\prod_{j=1}^k\prod_{s=1}^{d_j}
		\Delta_{js}\right)}
	{\langle\hat{\textbf{x}},\beta\rangle^2},
\end{align}
and hence the Ricci potential reads (up to a constant multiple)
\begin{align}
(-1)^m
\frac{\left(\prod_{j=1}^k\langle\hat{\textbf{x}}_j,\Gamma_j\rangle^{d_j}
	\right)
	\left(\prod_{j=1}^k\prod_{s=1}^{d_j}
	A_{js}(x_{js})\right)}
{\langle\hat{\textbf{x}},\beta\rangle^{m+2}},
\end{align}
whose $\partial_{x_{ir}}$-log derivative is
\begin{align}
\begin{split}
\sum_{j=1}^k
d_j\frac{\partial_{x_{ir}}
	\langle\hat{\textbf{x}}_j,\Gamma_j\rangle}
{\langle\hat{\textbf{x}}_j,\Gamma_j\rangle}+
\frac{\partial_{x_{ir}}A_{ir}(x_{ir})}
	{A_{ir(x_{ir})}}-
(m+2)\frac{\langle\partial_{x_{ir}}\mu,\beta\rangle}
{\langle\mu,\beta\rangle}=\\
\frac{\langle\mu,\beta\rangle^{m+2}}
	{A_{ir}(x_{ir})}
\frac{1}
	{\prod_{j=1}^k
		\langle\hat{\textbf{x}}_j,\Gamma_j\rangle^{d_j}}
\partial_{x_{ir}}
\left(
	\frac{A_{ir}(x_{ir})}
		{\langle\mu,\beta\rangle^{m+2}}
	\prod_{j=1}^k
		\langle\hat{\textbf{x}}_j,\Gamma_j\rangle^{d_j}
\right).
\end{split}
\end{align}
Thus, the scalar curvature is
\begin{align}
-
\sum_{i=1}^k\sum_{r=1}^{d_i}
\frac{\langle\mu,\beta\rangle^{m+1}}
	{	\Delta_{ir}
		\langle\hat{\textbf{x}}_i,\Gamma_i\rangle
		\left(\prod_{p=1}^k\langle\hat{\textbf{x}}_p,\Gamma_p\rangle^{d_p}
		\right)
	}
\partial_{x_{ir}}
\left(
	\langle\mu,\beta\rangle^2
	\partial_{x_{ir}}
	\left(
		\frac{A_{ir}(x_{ir})}
			{\langle\mu,\beta\rangle^{m+2}}
		\prod_{j=1}^k
		\langle\hat{\textbf{x}}_j,\Gamma_j\rangle^{d_j}
	\right)
\right).
\end{align}
Now, if $f''(x)=0$, then
\begin{gather}
\frac{d}{dx}
\left(
	f^2(x)
	\frac{d}{dx}
	\frac{h(x)}{f^{m+2}(x)}
\right)= \nonumber \\
=
\frac{h''(x)}{f^m(x)}-
2(m+1)
\frac{f'(x)h'(x)}{f^{m+1}(x)}+
(m+1)(m+2)
\frac{(f'(x))^2h(x)}{f^{m+2}(x)}= \nonumber \\
=f(x)\frac{d^2}{dx^2}
\frac{h(x)}{f^{m+1}(x)}.
\end{gather}
Therefore, we obtained
\begin{thm} \label{scalar curvature}
The scalar curvature of separable Kähler geometry \eqref{geometry tensors} is
\begin{align}
-
\sum_{i=1}^k\sum_{r=1}^{d_i}
\frac{\langle\mu,\beta\rangle^{m+2}}
	{	\Delta_{ir}
		\langle\hat{\textbf{x}}_i,\Gamma_i\rangle
		\left(\prod_{p=1}^k\langle\hat{\textbf{x}}_p,\Gamma_p\rangle^{d_p}
		\right)
	}
\partial_{x_{ir}}^2
\left(
	\frac{A_{ir}(x_{ir})}
		{\langle\mu,\beta\rangle^{m+1}}
	\prod_{j=1}^k
	\langle\hat{\textbf{x}}_j,\Gamma_j\rangle^{d_j}
\right),
\end{align}
and the geometry is extremal if and only if there exists $\alpha\in\mathfrak{h}$ such that
\begin{align}\label{extremality equation}
\sum_{i=1}^k\sum_{r=1}^{d_i}
\frac{\langle\mu,\beta\rangle^{m+3}}
	{	\Delta_{ir}
		\langle\hat{\textbf{x}}_i,\Gamma_i\rangle
		\left(
			\prod_{p=1}^k\langle\hat{\textbf{x}}_p,\Gamma_p\rangle^{d_p}
		\right)
	}
\partial_{x_{ir}}^2
\left(
	\frac{A_{ir}(x_{ir})}
		{\langle\mu,\beta\rangle^{m+1}}
	\prod_{j=1}^k
	\langle\hat{\textbf{x}}_j,\Gamma_j\rangle^{d_j}
\right)=
\langle\mu,\alpha\rangle.
\end{align}
\Cref{extremality equation} is referred to as the \textit{extremality equation} of a given separable Kähler geometry.
\end{thm}

\begin{rem}\label[rem]{extr eq rem}
In the case of the product Segre-Veronese factorization structure the extremality equation reads
\begin{align}
\sum_{i=1}^k\sum_{r=1}^{d_i}
\frac{\langle\mu,\beta\rangle^{m+3}}{\Delta_{ir}}
\partial_{x_{ir}}^2
\left(
	\frac{A_{ir}(x_{ir})}{\langle\mu,\beta\rangle^{m+1}}
\right)=
\langle\mu,\alpha\rangle.
\end{align}
In particular, the Segre and Veronese factorization structures recover extremality equations formulated in \cite{apostolov_cr_2020} for twisted products of Riemann surfaces and twisted orthotoric geometries, respectively.
In the case when $k=1$ and $m=2$ or when $k=2$ and $d_1=d_2=1$, we recover the corresponding equations for ambitoric geometries (see \cite{apostolov2016ambitoric,apostolov2015ambitoric}).
\end{rem}

\section{Extremality equation}

As outlined in the introduction, a major focus of this paper is the explicit study of Calabi’s extremality problem in the toric setting.
The Euler-Lagrange equation of the Calabi functional can be expressed as the scalar curvature being a Killing potential, which, in the setting of separable Kähler geometries \eqref{geometry tensors}, is \Cref{extremality equation}.
\Cref{solutions 1} proves that a solution $A_{pq}$, $q=1,\ldots,d_p$, $p=1,\ldots,k$, of the extremality equation consists of rational functions $A_{pq}$ whose denominators are determined solely by the factorization structure and behave uniformly in grouped slots, and whose numerators' degrees are bound explicitly.
We show that when the defining tensors decompose in grouped $p$-slots, the difference $A_{pq}(x) - A_{pr}(x)$, $q\neq r$, is an explicit rational function (in some cases the zero function), which provides an exact relation between $A_{pq}$ and $A_{pr}$.
Many of these findings are summarised in \Cref{main 1}.
Lastly, in \Cref{dec summed lemma}, we introduce a practical technique for expressing scalar curvature in momentum coordinates for a large class of separable geometries.
\\

More concretely, the extremality equation \eqref{extremality equation} is a PDE in unknowns $A_{ir}$, $r=1,\ldots,d_i$, $i=1,\ldots,k$, with rational functions as coefficients.
The strategy for solving it is following:
\begin{enumerate}
\item For a fixed variable $x_{ir}$, clear denominators of the PDE's coefficients so that the new coefficients are polynomials in $x_{ir}$.
\item For $\ell$, the highest degree with respect to $x_{ir}$ of these polynomial coefficients, apply $\partial_{x_{ir}}^{\ell+1}$ onto the PDE with the $x_{ir}$-polynomial coefficients to obtain an equation with a single unknown $A_{ir}$, often called a necessary condition, compatibility condition or prolongation.
\item For fixed $i \in \{1,\ldots,k\}$, relate $A_{ir},A_{ip}$, $r,p=1,\ldots,d_i$, through another necessary condition derived by multiplying \eqref{extremality equation} with $x_{ir}-x_{ip}$ and evaluating at $x_{ir}=x_{ip}=x$.
\item Solve the necessary conditions and verify which solutions satisfy the original PDE.
\end{enumerate}

In general, a necessary condition on $A_{ir}$, as described above, is an ODE that depends rationally on parameters, typically being all variables $x_{jp}$, $(j,p) \neq (i,r)$.
Once we prove that each $A_{ir}$ is a rational function, this allows us to interpret the necessary condition on $A_{ir}$ (the ODE) as a linear combination in a polynomial algebra over the filed of rational functions $\mathbb{R}(x_{ir})$, where the algebra is determined by the given factorization structure.
Empirically, if the necessary condition is a seemingly complicated ODE, this often indicates that the corresponding linear combination is in fact a linear combination of linearly independent vectors, implying that coefficients must vanish, which means solving yet another (frequently trivial) ODEs.\\

Throughout this section we work with the separable Kähler geometry \eqref{geometry tensors} corresponding to, unless specified otherwise, a fixed Segre-Veronese factorization structure $\varphi: \mathfrak{h} \to V^*$ defined in \Cref{SV def}.
Its scalar curvature and extremality equation are given by \Cref{scalar curvature}.
By \Cref{decomposes in grouped p-slots}, we have $\Gamma_j\in\bigotimes^k_{\substack{i=1\\i\neq j}}S^{d_i}W_i^*$, hence $\varphi(\beta) \in \bigotimes_{i=1}^k S^{d_i}W_i^*$ for any $\beta \in \mathfrak{h}$, and the polynomials $\langle \hat{\textbf{x}}_j, \Gamma_j \rangle$ and $\langle \mu, \beta \rangle$ are symmetric in grouped slots.
In particular, $deg_{x_{pq}}(\langle\hat{\textbf{x}}_j,\Gamma_j\rangle) = deg_{x_{pr}}(\langle\hat{\textbf{x}}_j,\Gamma_j\rangle)$ for any $q,r=1,\ldots,d_p$ and $j=1,\ldots,k$, and similarly for $\langle \mu, \beta \rangle$.
In addition, because $\Gamma_j$ and $\varphi(\beta)$ are symmetric in grouped slots, if it decomposes in the $(p,q)$-th slot, then it decomposes in grouped $p$-slots.

\subsection{Solutions are rational functions}\label{method 1}

Following the general strategy as outlined above, we multiply the extremality equation \eqref{extremality equation} by
\begin{align}
\Delta_{pq}
\prod_{b=1}^k
\langle\hat{\textbf{x}}_b,\Gamma_b\rangle^{d_b},
\hspace{.5cm}
\text{with the convention}
\hspace{.2cm}
\prod_{b=1}^1
\langle\hat{\textbf{x}}_b,\Gamma_b\rangle^{d_b}=1,
\end{align}
to obtain
\begin{gather}\label{multiplied extremality equation}
\sum_{i=1}^k\sum_{r=1}^{d_i}
\frac{\Delta_{pq}}{\Delta_{ir}}
\langle\mu,\beta\rangle^{m+3}
\langle\hat{\textbf{x}}_i,\Gamma_i\rangle^{d_i-1}
\partial_{x_{ir}}^2
\left(
	\frac{A_{ir}(x_{ir})}
		{\langle\mu,\beta\rangle^{m+1}}
	\prod_{\substack{b=1\\b\neq i}}^k
	\langle\hat{\textbf{x}}_b,\Gamma_b\rangle^{d_b}
\right)=
\langle\mu,\alpha\rangle
\Delta_{pq}
\prod_{b=1}^k
\langle\hat{\textbf{x}}_b,\Gamma_b\rangle^{d_b}.
\end{gather}

Note that the coefficients of the PDE \eqref{multiplied extremality equation} are polynomials in $x_{pq}$.
Differentiation of \eqref{multiplied extremality equation} provides useful relations between functions $A_{ir}$ of a solution, which are often specific to a class of factorization structures. We analyse two relations valid for any separable Kähler geometry.
Fix $(p,q)$ and denote by $\ell_p$ the maximum of $x_{pq}$-degrees of coefficients at $A_{ir}(x_{ir}),\partial_{x_{ir}}A_{ir}(x_{ir})$ and $\partial_{x_{ir}}^2A_{ir}(x_{ir})$, $(i,r)\neq(p,q)$, and of the right hand side of \eqref{multiplied extremality equation}.
Since $A_{ir}$, $(i,r) \neq (p,q)$, is constant with respect to $\partial_{x_{pq}}$ we find
\begin{lemma} \label[lemma]{minimal}
For each $p \in \{1,\ldots,k\}$ there exists a unique minimal non-negative integer $\ell_p$ such that for all $q = 1,\ldots,d_p$ we have
\begin{gather}
\partial_{x_{pq}}^{\ell_p+1}
\left(
	\langle\mu,\beta\rangle^{m+3}
	\partial_{x_{pq}}^2
	\left(
		\frac{A_{pq}(x_{pq})}
			{\langle\mu,\beta\rangle^{m+1}}
		\prod_{\substack{j=1\\j\neq p}}^k
		\langle\hat{\textbf{x}}_j,\Gamma_j\rangle^{d_j}
	\right)
\right)=
0.\label{extr eq consequence}
\end{gather}
Furthermore, $d_p \leq \ell_p \leq m$.
\end{lemma}

To handle \eqref{extr eq consequence}, note by direct computation that operators $D_{j,p}$ defined by
\begin{align}
D_{j,p}g=f^{p+1}\partial_{x_j}\frac{g}{f^p}\hspace{1cm}p=0,1,\ldots,
\end{align}
commute $[D_{j,p},D_{j,q}]=0$, where $f$ and $g$ are multivariate functions and $f$ is such that $\partial_{x_j}^2f=0$.
Thus, 
\begin{equation}\label{extr eq finer consequence}
\begin{split}
\partial_{x_{pq}}^{l+1}
\left(
	\langle\mu,\beta\rangle^{m+3}
	\partial_{x_{pq}}^2
	\left(
		\frac{A_{pq}(x_{pq})}
			{\langle\mu,\beta\rangle^{m+1}}
		\prod_{j=1}^k
		\langle\hat{\textbf{x}}_j,\Gamma_j\rangle^{d_j}
	\right)
\right)=\\
=\frac{1}
	{\langle\mu,\beta\rangle^{l+1}}
D_{pq,l}\circ\cdots\circ D_{pq,0}\circ D_{pq,m+2}\circ D_{pq,m+1}
\left(
	A_{pq}(x_{pq})
	\prod_{j=1}^k
	\langle\hat{\textbf{x}}_j,\Gamma_j\rangle^{d_j}
\right)=\\
=\frac{1}
	{\langle\mu,\beta\rangle^{l+1}}
D_{pq,m+2}\circ D_{pq,m+1}\circ D_{pq,l}\circ\cdots\circ D_{pq,0}
\left(
	A_{pq}(x_{pq})
	\prod_{j=1}^k
	\langle\hat{\textbf{x}}_j,\Gamma_j\rangle^{d_j}
\right)=\\
=
\langle\mu,\beta\rangle^{m+2-l}
\partial_{x_{pq}}^2
\left(
\frac{1}{\langle\mu,\beta\rangle^{m-l}}
\partial_{x_{pq}}^{l+1}
\left(
	A_{pq}(x_{pq})
	\prod_{j=1}^k
	\langle\hat{\textbf{x}}_j,\Gamma_j\rangle^{d_j}
\right)
\right)
\end{split}
\end{equation}
In particular, for $l=m$ we reformulated the upper bound on $\ell_p$ in \eqref{extr eq consequence} and see that
$
A_{pq}(x_{pq})
	\prod_{j=1}^k
	\langle\hat{\textbf{x}}_j,\Gamma_j\rangle^{d_j}
$ is a polynomial in $x_{pq}$ of degree at most $m+2$.
We found

\begin{proposition}\label[proposition]{solutions 1}
Let $A_{pq}$, $q=1,\ldots,d_p$, $p=1,\ldots,k$, be a solution of the extremality equation \eqref{extremality equation}.
Let $\bar{O}_p \subset \{ 1,\ldots,k \} \backslash \{ p \}$ be the set of those $j$ for which $\deg_{x_{pq}} \langle \hat{\textbf{x}}_j, \Gamma_j \rangle = 1$, and $O_p \subset \bar{O}_p$ the set of those $j$ for which $\Gamma_j$ decomposes in grouped $p$-slots (see \Cref{decomposes in grouped p-slots}).
Then,
\begin{align}
A_{pq}(x_{pq}) = \frac{N_{pq}(x_{pq})}{\prod_{j \in O_p} \langle (1,x_{pq}, e_j^p) \rangle^{d_j}},
\end{align}
where $N_{pq}$ is a polynomial such that $\deg(N_{pq}) \leq m + 2 - \sum_{j \in \bar{O}_p \backslash O_p} d_j$.
Furthermore, if $\partial_{x_{pq}} \langle \mu, \beta \rangle = 0$, then $\deg(N_{pq}) \leq \ell_p + 2 - \sum_{j \in \bar{O}_p \backslash O_p} d_j$.
\end{proposition}

\subsection{Case of decomposable defining tensors}

We assume that $\Gamma_j$, $j\neq p$, decompose in grouped $p$-slots, i.e., $\Gamma_j = ins_p \left( (e_j^p)^{\otimes d_p} \otimes \tilde{\Gamma}_j \right)$ for some $\tilde{\Gamma}_j$ (see \Cref{decomposes in grouped p-slots}), and study $A_{pq}$, $q=1,\ldots,d_p$, a part of a solution of the corresponding extremality equation:
\begin{gather}
\langle \mu, \alpha \rangle =
\sum_{r=1}^{d_p} \frac{\langle \mu, \beta \rangle^{m+3}}{\Delta_{pr} \langle \hat{\textbf{x}}_p, \Gamma_p \rangle \prod_{\substack{n=1 \\ n\neq p}}^k \langle (1,x_{pr}), e^p_n \rangle^{d_n}}
\partial_{x_{pr}}^2 \left( \frac{A_{pr}(x_{pr})}{\langle \mu, \beta \rangle^{m+1}} \prod_{\substack{n=1 \\ n \neq p}}^k \langle (1,x_{pr}), e_n^p \rangle^{d_n} \right) + \nonumber \\
\sum_{\substack{i=1 \\ i \neq p}}^k \sum_{r=1}^{d_i}
\frac{\langle \mu, \beta \rangle^{m+3}}{\Delta_{ir} \langle \hat{\textbf{x}}_i, \Gamma_i \rangle 
\langle \hat{\textbf{x}}_p, \Gamma_p \rangle^{d_p} \prod_{\substack{n=1 \\ n \neq p,i}}^k \langle \hat{\textbf{x}}_{np}, \tilde{\Gamma}_n \rangle^{d_n}}
\partial_{x_{ir}}^2 \left( \frac{A_{ir}(x_{ir})}{\langle \mu, \beta \rangle^{m+1}}
\langle \hat{\textbf{x}}_p, \Gamma_p \rangle^{d_p} \prod_{\substack{n=1 \\ n \neq p,i}}^k \langle \hat{\textbf{x}}_{np}, \tilde{\Gamma}_n \rangle^{d_n} \right), \label{extr eq dec}
\end{gather}
where $\hat{\textbf{x}}_{np} = \bigotimes_{\substack{j=1 \\ j\neq n,p}}^k \bigotimes_{s=1}^{d_j} (1,x_{js})$.
Note that because of this decomposability, in order to turn the coefficients of the extremality equation to polynomials with respect to $x_{pq}$, it is enough to multiply it by
\begin{gather}\label{poly}
\Delta_{pq} \prod_{\substack{j=1 \\ j\neq p}}^k \langle (1,x_{pq}), e_j^p \rangle.
\end{gather}

Then, the highest $x_{pq}$-degree $L_p$ of the polynomial coefficients is $o_p := |O_p|$, where $O_p \subset \{1,\ldots,k\} \backslash \{ p \}$ is the set of those $j$ for which $e_j^p$ does not lie on the span of $(1,0)$ or, equivalently, $\deg_{x_{pq}} \langle (1,x_{pq}), e_j^p \rangle = 1$.
This definition extends to the case of Veronese factorization structure, where $k=1$, and $p=1$, $d_1=m$, and $O_p = \emptyset$ follow from definitions.
We find

\begin{lemma}\label[lemma]{Lp lemma}
Let $\Gamma_j$, $j \neq p$, decompose in grouped $p$-slots, or the involved factorization structure be the Veronese factorization structure.
If $\partial_{x_{pq}} \langle \mu, \beta \rangle = 0$, then $L_p = d_p + o_p$.
Further, if $\partial_{x_{pq}} \langle \mu, \beta \rangle \neq 0$ and
\begin{enumerate}
\item $O_p \cup \{p\} = \{1,\ldots,k\}$, then $L_p = d_p + o_p$ and $\ell_p = \sum_{b \in O_p \cup \{p\}} d_b = m$. (This covers Veronese case.)
\item $O_p \cup \{p\} \neq \{1,\ldots,k\}$, then $L_p = d_p + o_p + 1$ and $\ell_p = 1 + \sum_{b \in O_p \cup \{p\}} d_b$.
\end{enumerate}
\end{lemma}
Hence, multiplying the extremality equation by \eqref{poly} and then applying $\partial_{x_{pq}}^{L_p+1}$ yields
\begin{gather}
\partial_{x_{pq}}^{L_p+1}
\left(
	\frac{\langle\mu,\beta\rangle^{m+3}}
		{\prod_{j \in O_p}\langle (1,x_{pq}),e_j^p \rangle^{d_j-1}}
	\partial_{x_{pq}}^2
	\left(
		\frac{A_{pq}(x_{pq})}
			{\langle\mu,\beta\rangle^{m+1}}
		\prod_{j \in O_p}\langle (1,x_{pq}),e_j^p \rangle^{d_j}
	\right)
\right) = 0,
\end{gather}
which, using \Cref{solutions 1}, can be rewritten as equality of two polynomials
\begin{gather}
\langle \mu, \beta \rangle^{m+3} \partial_{x_{pq}}^2 \frac{N_{pq}(x_{pq})}{\langle \mu, \beta \rangle^{m+1}}
= \prod_{j \in O_p}\langle (1,x_{pq}),e_j^p \rangle^{d_j-1} \Upsilon_{pq}(x_{pq}),
\end{gather}
where $\deg_{x_{pq}} \Upsilon_{pq} \leq L_p$, and so $\deg_{x_{pq}} N_{pq} -2 \leq \sum_{j \in O_p \cup \{p\}} d_j$ if $\partial_{x_{pq}} \langle \mu, \beta \rangle = 0$. \\

For $\varphi(\beta)$ decomposable in grouped $p$-slots, i.e., $\varphi(\beta) = ins_p \left( (a,b)^{\otimes d_p} \otimes \tilde{\beta} \right)$ for some $\tilde{\beta}$ and non-zero $(a,b) \in W_p^*$, we find

\begin{corollary}\label[corollary]{deccor}
Let $\varphi(\beta)$ and all $\Gamma_j$, $j\neq p$, decompose in grouped $p$-slots, or let the involved factorization structure be the Veronese factorization structure.
Then, for $A_{pq}$, $q=1,\ldots,d_p$, a part of a solution of the corresponding extremality equation, we have
\begin{gather}
A_{pq}(x_{pq})=
\frac{(a+bx_{pq})^{m+1}}{\prod_{j \in O_p} \langle(1,x_{pq}),e_j^p\rangle^{d_j}}
\int \int \Upsilon_{pq}(x_{pq}) \frac{\prod_{j \in O_p} \langle(1,x_{pq}),e_j^p\rangle^{d_j-1}}{(a+bx_{pq})^{m+3}} dx dx,
\hspace{.2cm}
q=1,\ldots,d_p,
\end{gather}
where $\Upsilon_{pq}$ is a polynomial of degree at most $L_p$.
Note that the double-integral produces also a linear function from constants of integration, and that when $b=0$ we recover the corresponding case from \Cref{solutions 1}.
In the Veronese case, the same formula applies not only for $\varphi(\beta) = (a,b)^{\otimes m}$, but also for $\varphi(\beta) = t (a,b)^{\otimes m}$, where $t \in \mathbb{R} \backslash \{0\}$.
\end{corollary}

To address indecomposable $\varphi(\beta)$, we first observe that $\varphi(\beta)$ decomposes in grouped $p$-slots if and only if $\langle \mu, \beta \rangle$ and $\partial_{x_{pq}} \langle \mu, \beta \rangle$ are linearly dependent over the rational functions $\mathbb{R}(x_{pq})$ for any $q=1,\ldots,d_p$ (see \Cref{linear independence} for more details).
With this new perspective of linear dependence relations in mind, we expand the last line of \eqref{extr eq finer consequence} for $l=\ell_p$ from \Cref{minimal} and obtain
\begin{gather}\nonumber
\left[
(m-\ell_p)(m-\ell_p+1) \left( \partial_{x_{pq}} \langle \mu, \beta \rangle \right)^2 \partial_{x_{pq}}^{\ell_p+1} -
2(m-\ell_p) \langle \mu, \beta \rangle \partial_{x_{pq}} \langle \mu, \beta \rangle
\partial_{x_{pq}}^{\ell_p+2}+
\langle \mu, \beta \rangle^2
\partial_{x_{pq}}^{\ell_p+3}
\right]\\
\bigg( A_{pq}(x_{pq}) {\prod_{j \in O_p} \langle(1,x_{pq}),e_j^p\rangle^{d_j}} \bigg) = 0, \label{LIeq}
\end{gather}
which we regard as a linear combination relation as follows.
Since $A_{pq}$ was shown to be a rational function in \Cref{solutions 1}, the equation \eqref{LIeq} is a relation in the algebra over $\mathbb{R}(x_{pq})$ generated by the elementary symmetric polynomials $\sigma_l(x_{j1},\ldots,x_{jd_j})$, $l=1,\ldots,d_j$, $j\neq p$, and $\sigma_l(x_{p1},\ldots,\hat{x}_{pq},\ldots,x_{pd_p})$, $l=1,\ldots,d_p-1$, where $\hat{x}_{pq}$ represents the omission of the variable $x_{pq}$, and for $d_p=1$ the latter set of generators is omitted.
In fact, it can be viewed as a relation in a module over polynomials $\mathbb{R}[x_{pq}]$.
We have
\begin{corollary}\label[corollary]{final piece}
If $\ell_p < m$ and if
\begin{gather}\label{t1}
\langle \mu, \beta \rangle^2, \hspace{.1cm} \langle \mu, \beta \rangle \partial_{x_{pq}} \langle \mu, \beta \rangle, \hspace{.1cm} (\partial_{x_{pq}} \langle \mu, \beta \rangle)^2
\end{gather}
are linearly independent over the rational functions $\mathbb{R}(x_{pq})$, then the solution
\begin{gather}
A_{pq}(x_{pq}) = \frac{N_{pq}(x_{pq})}{{\prod_{\substack{j=1\\j\neq p}}^k \langle(1,x_{pq}),e_j^p\rangle^{d_j}}},
\end{gather}
must be so that $\deg N_{pq} \leq \ell_p$.
\end{corollary}
\begin{proof}
The assumption ensures that \eqref{t1} are linearly independent, which means that coefficients of \eqref{LIeq} are trivial.
\end{proof}

We split our analysis into two cases, depending on whether $\langle \mu, \beta \rangle$ and $\partial_{x_{pq}} \langle \mu, \beta \rangle$ are linearly independent over $\mathbb{R}(x_{pq})$.
The first case, when they are linearly dependent, is resolved in \Cref{deccor}.
The difficulty lies in the second case.
To use \Cref{final piece}, one attempts to argue that the linear independence of $\langle \mu, \beta \rangle$ and $\partial_{x_{pq}} \langle \mu, \beta \rangle$ implies the linear independence of \eqref{t1}.
Although this implication does not hold in general, the following construction ensures it in many interesting cases.

\begin{rem}\label[rem]{algebra remark}
Note that if one can construct an algebra over $\mathbb{R}(x_{pq})$ which is generated in degree one, contains the relation \eqref{LIeq}, and in which $\langle \mu, \beta \rangle$ and $\partial_{x_{pq}} \langle \mu, \beta \rangle$ are linearly independent degree-one elements, then these elements may be extended to a basis of the algebra's first graded piece, i.e., the basis forms a minimal generating set, which in turn ensures that \eqref{t1} is a linearly independent set.

Interpreting equation \eqref{LIeq} as a linear combination in an algebra is a powerful technique, whose efficacy becomes most apparent in concrete examples (see \Cref{cs}).
The construction below associates an algebra to any given factorization structure, but does not always guarantee the desired properties.
However, for the specific factorization structures considered here, it does yield an appropriate algebra.

To begin, note that in general
\begin{gather}
\langle \mu, \beta \rangle =
\sum_{j=1}^k
\langle (1,x_{j1}) \otimes \cdots \otimes (1,x_{jd_j}), \beta_j \rangle
\langle \hat{\textbf{x}}_j, \Gamma_j \rangle
\end{gather}
for some $\beta_j \in S^{d_j} W_j^*$, $j=1,\ldots,k$.
Therefore, $\langle \mu, \beta \rangle$ can be viewed as a degree one polynomial in the $\mathbb{R}$-algebra generated by polynomials $\theta_{rj} := \sigma_r(x_{j1},\ldots,x_{jd_j}) \langle \hat{\textbf{x}}_j, \Gamma_j \rangle$, $r=0,\ldots,d_j$, $j=1,\ldots,k$.
Going one step further, every $\theta_{rj}$ can be expanded into a linear combination with coefficients in the polynomial ring $\mathbb{R}[x_{pq}]$ of some polynomials $\theta_{rj}^a$ which do not depend on $x_{pq}$.
Then, the algebra generated by all $\theta_{rj}^a$ over $\mathbb{R}(x_{pq})$ is the algebra in which \eqref{LIeq} is viewed as a linear combination.
\end{rem}

\subsection{Relating $A_{pq}$ and $A_{pr}$ of a solution when $d_p\geq2$}\label{method 2}
For $p\in\{1,\ldots,k\}$ such that $d_p\geq2$, this subsection relates $A_{p1},\ldots,A_{pd_p}$ of a solution of the extremality equation as follows.
We multiply the extremality equation \eqref{extremality equation} by $x_{pq}-x_{pr}$, where $q,r\in\{1,\ldots,d_p\}$ are distinct, evaluate the resulting equation at $x_{pq}=x_{pr}=x$, and obtain
\begin{align}\label{diagonal extr eq}
\begin{gathered}
A
\left[
	\langle\mu,\beta\rangle_x^2
	\mathcal{P}''
	-2(m+1)\langle\mu,\beta\rangle_x
	\mathcal{\langle\mu,\beta\rangle}'
	\mathcal{P}'
	+(m+1)(m+2)(\mathcal{\langle\mu,\beta\rangle}')^2
	\mathcal{P}
\right]+\\
2A'
\Big[
	\langle\mu,\beta\rangle_x^2
	\mathcal{P}'
	-(m+1)\mathcal{P}
	\langle\mu,\beta\rangle_x
	\langle\mu,\beta\rangle'
\Big]+
\\
A''
\mathcal{P}
\langle\mu,\beta\rangle_x^2=
0,
\end{gathered}
\end{align}
where
\begin{align}
A=&
A_{pq}(x)-A_{pr}(x)\\
\mathcal{P}^{(l)}=&
\partial_{x_{pq}}^l
\Big|_{x_{pq}=x_{pr}=x}
\prod_{\substack{j=1 \\ j\neq p}}^k
\langle\hat{\textbf{x}}_j,\Gamma_j\rangle^{d_j}=
\partial_{x_{pr}}^l
\Big|_{x_{pq}=x_{pr}=x}
\prod_{\substack{j=1 \\ j\neq p}}^k
\langle\hat{\textbf{x}}_j,\Gamma_j\rangle^{d_j},
\hspace{.5cm}
l=0,1,2, \\
\langle\mu,\beta\rangle'=&
\partial_{x_{pq}}
\Big|_{x_{pq}=x_{pr}=x}
\langle\mu,\beta\rangle=
\partial_{x_{pr}}
\Big|_{x_{pq}=x_{pr}=x}
\langle\mu,\beta\rangle,\\
\langle \mu, \beta \rangle_x =& \langle \mu, \beta \rangle \Big|_{x_{pq}=x_{pr}=x}.
\end{align}

From now on, suppose that $\Gamma_j$, $j\neq p$, decompose in grouped $p$-slots.
Then, dividing \eqref{diagonal extr eq} by $\mathcal{P}$ yields
\begin{align}\label{diagonal extr eq 1}
\begin{gathered}
A
\left[
	( \mathcal{S}^2+\mathcal{S}' ) \langle\mu,\beta\rangle_x^2
	-2(m+1)
	\langle\mu,\beta\rangle_x \langle\mu,\beta\rangle'
	\mathcal{S}
	+(m+1)(m+2)
	\left( \langle\mu,\beta\rangle' \right)^2
\right]+\\
2A'
\Big[
	\mathcal{S} \langle\mu,\beta\rangle_x^2
	-(m+1)
	\langle\mu,\beta\rangle_x \langle\mu,\beta\rangle'
\Big]+
A'' \langle\mu,\beta\rangle^2_x =
0,
\end{gathered}
\end{align}
where
\begin{align}
S=
\sum_{\substack{j=1\\j\neq p}}^k
d_j
\frac{\langle(0,1),e_j^p\rangle}
	{\langle(1,x),e_j^p\rangle}
\in\mathbb{R}(x)
\end{align}
is the $\partial_{x_{pq}}$-logarithmic derivative of
$\prod_{j=1}^k
\langle\hat{\textbf{x}}_j,\Gamma_j\rangle^{d_j}$
evaluated at $x_{pq}=x_{pr}=x$, i.e. $\mathcal{P}'=\mathcal{P}\mathcal{S}$.
Note, $\mathcal{P}''=\mathcal{P} (\mathcal{S}^2 + \mathcal{S}')$.

\begin{rem}\label[rem]{constant twist solution remark}
If $\langle\mu,\beta\rangle'=0$, then \eqref{diagonal extr eq 1} reduces to
\begin{align}\label{ODE}
A''+
2A'\mathcal{S}
+
A(\mathcal{S}^2+\mathcal{S}')=
0,
\end{align}
which can be viewed as $(e^{\int S}A)''=0$, and whose solutions are
\begin{align}\label{const diag solution}
A(x)=\frac{\alpha^0+\alpha^1x}{\prod_{\substack{j=1\\j\neq p}}^k	\langle(1,x),e_j^b\rangle^{d_j}},
\end{align}
where $\alpha^0, \alpha^1 \in \mathbb{R}$.
We remark that $\langle\mu,\beta\rangle'=0$ is equivalent with $\partial_{x_{pq}}\langle\mu,\beta\rangle=0$ for any $q=1,\ldots,d_p$.
Using \Cref{deccor} we conclude
\begin{gather}
A_{pq}(x_{pq})=
\frac{1}{\prod_{j \in O_p} \langle(1,x_{pq}),e_j^p\rangle^{d_j}}
\int \int \Upsilon_p(x_{pq}) \prod_{j \in O_p} \langle(1,x_{pq}),e_j^p\rangle^{d_j-1} dx dx,
\end{gather}
where $\Upsilon_p$ is $q$-independent polynomial of degree at most $L_p$.
Note that for each $q=1,\ldots,d_p$, the double-integral produces a linear term whose coefficients in general depend on both $p$ and $q$, as required by \eqref{const diag solution}.
\end{rem}

To address $\langle\mu,\beta\rangle' \neq 0$ case, we use the fact that $A$ is a rational function depending only on $x$, and view \eqref{diagonal extr eq 1} as a linear relation in the algebra over the field of rational functions $\mathbb{R}(x)$ generated by the elementary symmetric polynomials $\sigma_l(x_{j1}, \ldots, x_{jd_j})$, $l=1,\ldots,d_j$, $j\neq p$, and $\sigma_l(x_{p1},\ldots,\hat{x}_{pq},\ldots,\hat{x}_{pr}, \ldots, x_{pd_p})$, $l=1,\ldots,d_p-2$, where $\hat{x}_{pq}$ and $\hat{x}_{pr}$ represent the omission of the variables $x_{pq}$ and $x_{pr}$.

\begin{proposition} \label[proposition]{LI - solutions}
If
\begin{gather} \label{LI}
\left\{ \langle\mu,\beta\rangle_x^2, \langle\mu,\beta\rangle_x \langle\mu,\beta\rangle', \left( \langle\mu,\beta\rangle' \right)^2 \right\}
\end{gather}
is a linearly independent set over the rational functions $\mathbb{R}(x)$, then \eqref{diagonal extr eq 1} has only the trivial solution $A=0$, i.e., $A_{pq}(x) = A_{pr}(x)$.
\end{proposition}

Recall that in the previous subsection, we found it natural to distinguish cases based on the linear independence of $\langle \mu, \beta \rangle$ and $\partial_{x_{pq}} \langle \mu, \beta \rangle$.
Strictly speaking, because the case when these are linearly independent is challenging to handle in full generality, we instead focused on the scenario where
$\langle \mu, \beta \rangle^2, \hspace{.1cm} \langle \mu, \beta \rangle \partial_{x_{pq}} \langle \mu, \beta \rangle$ and $(\partial_{x_{pq}} \langle \mu, \beta \rangle)^2$
are linearly independent, and claimed that, for the examples covered here, this condition is equivalent to the linear independence of $\langle \mu, \beta \rangle$ and $\partial_{x_{pq}} \langle \mu, \beta \rangle$.

By analogy, in this subsection, we addressed the case when elements of \eqref{LI} are linearly independent.
The next step is to analyse the case when $\langle \mu, \beta \rangle_x$ and $\langle \mu, \beta \rangle'$ are linearly dependent.
In doing so, we find that the situation when $d_p=2$ is atypical and requires further distinction according to whether $\langle \mu, \beta \rangle$ and $\partial_{x_{pq}} \langle \mu, \beta \rangle$ are linearly independent.
We summarise the behaviour of linear in/dependencies in the following

\begin{lemma} \label[lemma]{linear independence}
Let $W$ and $V_{ji}$, $i=1,\ldots,d_j$, $j=1,\ldots,k$, be 2-dimensional vector spaces, $\textbf{x}_j = (1,x_{j1}) \otimes \cdots \otimes (1,x_{jd_j})$, $p \in \{1,\ldots,k\}$, and
\begin{gather}
\Gamma \in ins_p ( S^{d_p}W_p^* \otimes \bigotimes_{\substack{j=1 \\ j \neq p}}^k \bigotimes_{i=1}^{d_j} V_{ji}^* ),
\hspace{.2cm} \text{ and } \hspace{.2cm}
\gamma = \langle ins_p ( \textbf{x}_p \otimes \bigotimes_{\substack{j=1 \\ j \neq p}}^k \textbf{x}_j ), \Gamma \rangle.
\end{gather}
\begin{enumerate}
\item
For $p$ such that $d_p \geq 3$ and some fixed $q,r \in \{1,\ldots,d_p\}$, $q\neq r$,
\begin{align}\label{same derivative due symmetricity}
\gamma\bigg|_{x_{pq}=x_{pr}=x}
\hspace{1cm}\text{and}\hspace{1cm}
\partial_{x_{pq}}
\bigg|_{x_{pq}=x_{pr}=x}\gamma =
\partial_{x_{pr}}
\bigg|_{x_{pq}=x_{pr}=x}\gamma
\end{align}
are linearly dependent over rational functions $\mathbb{R}(x)$ if and only if they are linearly dependent over polynomials $\mathbb{R}[x]$ if and only if $\Gamma$ decomposes in grouped $p$-slots, i.e., $\Gamma = ins_p \left( (a,b)^{\otimes d_p} \otimes K \right)$ for some $K$, and thus
\begin{align}\label{decomposable gamma}
\gamma=
\langle \bigotimes_{\substack{j=1 \\ j\neq p}}^k \textbf{x}_j, K \rangle
\prod_{i=1}^{d_p}
(a+bx_{pi}).
\end{align}
In particular, there exist $q$ and $r$, $q\neq r$, such that \eqref{same derivative due symmetricity} are linearly dependent over $\mathbb{R}(x)$ if and only if for all $q$ and $r$, $q\neq r$, the functions \eqref{same derivative due symmetricity} are linearly dependent over $\mathbb{R}(x)$.
In addition, $\gamma$ and $\partial_{x_{pq}} \gamma$ are linearly dependent over $\mathbb{R}(x_{pq})$ if and only if they are linearly dependent over $\mathbb{R}[x_{pq}]$ if and only if for any $r\neq q$ \eqref{same derivative due symmetricity} are linearly dependent.
\item
For $d_p=2$, \eqref{same derivative due symmetricity} are linearly dependent over $\mathbb{R}(x)$ if and only if $\gamma = \kappa^p (a+b(x_{p1}+x_{p2})+cx_{p1}x_{p2})$, where $\kappa^p$ does not depend on $x_{p1}, x_{p2}$, and $a,b,c\in\mathbb{R}$.
Moreover, $\gamma$ and $\partial_{x_{pq}} \gamma$ are linearly dependent over $\mathbb{R}[x_{pq}]$ if and only if they are linearly dependent over $\mathbb{R}(x_{pq})$ if and only if the polynomial $a+2bx+cx^2$ has a real root, in which case
\begin{align}
\gamma=
\langle \bigotimes_{\substack{j=1 \\ j\neq p}}^k \textbf{x}_j, K \rangle
(a+bx_{p1}) (a+bx_{p2})
\end{align}
for some $a,b \in \mathbb{R}$.
\end{enumerate}
\end{lemma}
\begin{proof}
The proof is a simple but tedious coefficients comparing for \eqref{same derivative due symmetricity} using the expansion
\begin{align}\label{independence lemma gamma}
\gamma=
\sum_{i=0}^{d_p-2}
\left[
	\alpha_i+(x_{pq}+x_{pr})\alpha_{i+1}+x_{pq}x_{pr}\alpha_{i+2}
\right]
\sigma_i(\hat{x}_{pq},\hat{x}_{pr}),
\end{align}
where $\sigma_i(\hat{x}_{pq},\hat{x}_{pr})$ is the $i$-th elementary symmetric polynomial in variables $x_{pj}$, $j\neq q$ and $j\neq r$.
\end{proof}

\begin{rem}\label[rem]{Veronese decomposability}
Note that in the Veronese situation, i.e., when $k=1$, \Cref{linear independence} shows that, under linear dependence over $\mathbb{R}(x_{1q})$, the tensor $\Gamma \in S^m W^*$ must be of the form $(a,b)^{\otimes m}$ up a non-zero scalar, similarly to \Cref{deccor}.
This careful treatment of scalars is characteristic of the Veronese factorization structure, and throughout this work, decomposability in $S^m W^*$ will always be understood up to multiplication by a non-zero scalar.
For all other factorization structures, this issue does not arise, since the scalar can propagate through the tensor product.
In these cases, only the grouped $p$-slots, of the form $(a,b)^{\otimes d_p}$, are relevant.
For instance, if $\Gamma = ins_p \left( t(a,b)^{\otimes d_p} \otimes K \right)$, then one may absorb $t$ into the remaining factor by setting $\tilde{K} = t K$, giving $\Gamma = ins_p \left( (a,b)^{\otimes d_p} \otimes \tilde{K} \right)$.
\end{rem}

\begin{corollary}\label[corollary]{cor}
Vectors $\langle \mu, \beta \rangle_x$ and $\langle \mu, \beta \rangle'$ are linearly dependent over $\mathbb{R}(x)$ if and only if one of the following holds:
\begin{enumerate}
\item[(i)] $d_p \geq 2$ and $\langle \mu, \beta \rangle$ and $\partial_{x_{pq}} \langle \mu, \beta \rangle$ are linearly dependent over $\mathbb{R}(x_{pq})$
\item[(ii)] $d_p = 2$ and $\langle \mu, \beta \rangle$ and $\partial_{x_{pq}} \langle \mu, \beta \rangle$ are linearly independent over $\mathbb{R}(x_{pq})$, and $\langle \mu, \beta \rangle = \kappa (a+b(x_{p1} + x_{p2}) + cx_{p1}x_{p2})$ for some $a,b,c \in \mathbb{R}$.
\end{enumerate}
\end{corollary}

To continue our analysis of \Cref{diagonal extr eq 1} in the case when  $\langle \mu, \beta \rangle_x$ and $\langle \mu, \beta \rangle'$ are linearly dependent over $\mathbb{R}(x)$, we use \Cref{cor} and begin with the part (i): we assume that $\langle \mu, \beta \rangle$ and $\partial_{x_{pq}} \langle \mu, \beta \rangle$ are linearly dependent over $\mathbb{R}[x_{pq}]$ and find
\begin{align} \label{note1}
\frac{\langle\mu,\beta\rangle'}{\langle\mu,\beta\rangle_x} =
\partial_x ln(a+bx).
\end{align}
Therefore, after dividing \Cref{diagonal extr eq 1} by $\langle\mu,\beta\rangle_x$, it becomes an ODE of type \eqref{ODE} with $\mathcal{S}$ replaced by $\mathcal{S}-(m+1)\partial_x ln(a+bx)$, and thus the solutions are
\begin{align}
A=
(\beta^0+\beta^1x)
\frac{(a+bx)^{m+1}}
	{\prod_{\substack{j=1\\j\neq p}}^k
		\langle(1,x),e_j^p\rangle^{d_j}}.
\end{align}
Using \Cref{deccor} we obtain
\begin{gather}\label{a}
A_{pq}(x_{pq})=
\frac{(a+bx_{pq})^{m+1}}{\prod_{j \in O_p} \langle(1,x_{pq}),e_j^p\rangle^{d_j}}
\int \int \Upsilon_p(x_{pq}) \frac{\prod_{j \in O_p} \langle(1,x_{pq}),e_j^p\rangle^{d_j-1}}{(a+bx_{pq})^{m+3}} dx dx,
\hspace{.2cm}
q=1,\ldots,d_p,
\end{gather}
where $\Upsilon_p$ is a $q$-independent polynomial of degree at most $L_p$.
Note that the double integral produces a linear term from constants of integration, which depend on both $p$ and $q$. \\

Now we address the final case.
\begin{rem}\label[rem]{coupled}
Assuming part (ii) of \Cref{cor}, we find
\begin{align}
2
\frac{\langle\mu,\beta\rangle'}{\langle\mu,\beta\rangle_x} =
\partial_x ln(a+2bx+cx^2)
\end{align}
and rewrite \Cref{diagonal extr eq 1} as
\begin{align}\label{the ODE}
\begin{gathered}
A
\left[
	\mathcal{S}^2+\mathcal{S}'
	-(m+1)
	\mathcal{S}
	\partial_x ln(a+2bx+cx^2)
	+\frac{(m+1)(m+2)}{4}
	\left(\partial_x ln(a+2bx+cx^2)\right)^2
\right]+\\
2A'
\Big[
	\mathcal{S}
	-\frac{m+1}{2}
	\partial_x ln(a+2bx+cx^2)
\Big]+
\\
A''=
0.
\end{gathered}
\end{align}
This ODE is not of type \eqref{ODE} and is not solved here in this generality.
\end{rem}

Instead of solving \eqref{the ODE}, we describe the solution for the product Segre-Veronese factorization structure in the spaces of polynomials relevant to the extremality problem.
Therefore, we solve 
\begin{align}\label{eq}
\begin{gathered}
\frac{(m+1)(m+2)}{4}
\left(\partial_x ln(a+2bx+cx^2)\right)^2 A -
(m+1) \partial_x ln(a+2bx+cx^2) A' +
A''=
0,
\end{gathered}
\end{align}
where $a+2bx+cx^2$ does not have a real root and $c\neq0$, in the space of polynomials $\tilde{A}$ (see \Cref{solutions 1}) whose degree is bounded by a constant depending on the following two cases:
(i) $\ell_p = m$ and hence $\deg \tilde{A} \leq m+2$ by \Cref{solutions 1}, and (ii) $\ell_p < m$ with $deg \tilde{A} \leq l_p$, where we used \Cref{final piece} since \eqref{t1} are linearly independent in this case as \Cref{algebra remark} applies here: for full details see \Cref{pSVfs lemma} and its proof.
\Cref{Lp lemma} shows that for the product Segre-Veronese factorization structure and $\langle \mu, \beta \rangle$ non-constant we have $\ell_p = d_p + 1 = 3$ if $k\geq 2$, and $\ell_p = 2$ if $k=1$.
Therefore, in case (i) we have $m=3$ for $k\geq 2$, in which case we look for polynomial solutions of degree at most five, and $m=2$ for $k=1$, in which case we look for polynomial solutions of degree at most four.
Case (ii) concerns polynomial solutions of degree at most three for $m\geq4$ and $k\geq2$.

\begin{lemma}\label[lemma]{eq lemma}
For the \Cref{eq} we have the following.
\begin{enumerate}
\item If $m = 2$, polynomials of degree at most four which solve the equation form a 2-dimensional $\mathbb{R}$-vector space with the basis $(b+cx)(ac-b^2+(b+cx)^2)$ and $(ac-b^2)^2 - (b+cx)^4$.
\item If $m = 3$, the only polynomial of degree at most five which solves the equation is the zero polynomial.
\item If $m \geq 4$, the only polynomial of degree at most three which solves the equation is the zero polynomial.
\end{enumerate}
\end{lemma}
\begin{proof}
By denoting $\gamma = ac - b^2$ and changing coordinates $y=b+cx$ we free ourselves of obscuring constants and obtain
\begin{gather}\label{pSVfs eq}
(\gamma + y^2)^{m+3} \partial_y \frac{\tilde{A}'(y)}{(\gamma + y^2)^{m+1}} + (m+1)(m+2) y^2 \tilde{A}(y) = 0,
\end{gather}
where $\tilde{A}(y) = A(x)$.
To solve this equation in polynomials, one plugs in a general polynomial $\sum_{j=0}^n \alpha_j (\gamma + y^2)^j + \beta_j y(\gamma + y^2)^j$, and compares coefficients.
The resulting linear system in unknowns $\alpha_j, \beta_j$, $j=1,\ldots,n$, de-couples into two linear systems:
\begin{gather}
\alpha_0 = 0, \hspace{.2cm}
(m-2) \alpha_1 = 0, \hspace{.2cm}
(m+1-2n) (m+2-2n) \alpha_n = 0, \\
(m+3-2j)(m+4-2j) \alpha_{j-1} + \gamma [4j(m+2-j) - (m+1)(m+2)] \alpha_j = 0, \hspace{.2cm} j=2,\ldots,n \label{alpha}
\end{gather}
and
\begin{gather}
\beta_0 = 0, \hspace{.2cm}
(m-2n)(m+1-2n) \beta_n = 0, \\
(m+2-2j)(m+3-2j) \beta_{j-1} - \gamma [4j(m+2-j) - (m+1)(m+2)] \beta_j = 0 \hspace{.2cm}
j=1,\ldots,n. \label{beta}
\end{gather}
These systems can be easily solved in general, but in the interest of brevity, one can easily verify cases relevant to our study.
\end{proof}

Gathering together some results of this section we obtain
\begin{thm}\label{main 1}
Let $p \in \{1,\ldots,k\}$ be fixed and $\varphi:\mathfrak{h} \to V^*$ be a Segre-Veronese factorization structure \eqref{SV inclusion} such that $\Gamma_j$, $j\neq p$, decompose in grouped $p$-slots, i.e., $\Gamma_j = ins_j \left( (e_j^p)^{\otimes d_p} \otimes \tilde{\Gamma}_j \right)$ for some $e_j^p \in W_p^*$ (see \Cref{decomposes in grouped p-slots}), or the Veronese factorization structure.
Let $(g_\beta,J_\beta)$ be the separable Kähler geometry \eqref{geometry tensors} associated with this factorization structure and $\beta \in \mathfrak{h}$, $l_p$ be as in \Cref{minimal}, and $L_p$ be as defined below \eqref{poly}.
Let $A_{ir}$, $r=1,\ldots,d_i$, $i=1,\ldots,k$, be a solution of the corresponding extremality equation \eqref{extr eq dec}.
Then, for $d_p\geq1$ we have the following.
\begin{enumerate}
\item
	If $\langle \mu, \beta \rangle$ and $\partial_{x_{pq}} \langle \mu, \beta \rangle$ are linearly dependent over polynomials $\mathbb{R}[x_{pq}]$, i.e., $\varphi(\beta) = ins_p \left( (a,b)^{\otimes d_p} \otimes \tilde{\beta} \right)$ for some $\tilde{\beta}$ and $b\neq0$ (see \Cref{Veronese decomposability}), then for each $q=1,\ldots,d_p$ we have
	\begin{gather}\label{part 1}
	A_{pq}(x_{pq})=
	\frac{(a+bx_{pq})^{m+1}}{\prod_{j \in O_p} \langle(1,x_{pq}),e_j^p\rangle^{d_j}}
	\int \int \Upsilon_p(x_{pq}) \frac{\prod_{j \in O_p} \langle(1,x_{pq}),e_j^p\rangle^{d_j-1}}{(a+bx_{pq})^{m+3}} dx_{pq} dx_{pq},
	\end{gather}
	where $\Upsilon_p$ is a $q$-independent polynomial of degree at most $L_p$ with constant coefficients (with respect to every variable).
	Note that the double integral produces a linear term from constants of integration, which depend on both $p$ and $q$.
	This statements remains valid for $b=0$ when the resulting constant $1/a^2$ is omitted or viewed as a part of $\Upsilon_p$.
\item
	If $d_p \geq 2$, $\ell_p < m$, $\langle \mu, \beta \rangle^2, \langle \mu, \beta \rangle \partial_{x_{pq}} \langle \mu, \beta \rangle$ and $(\partial_{x_{pq}} \langle \mu, \beta \rangle)^2$ are linearly independent over $\mathbb{R}(x_{pq})$, and $\langle\mu,\beta\rangle_x^2, \langle\mu,\beta\rangle_x \langle\mu,\beta\rangle'$ and $\left( \langle\mu,\beta\rangle' \right)^2$ are linearly independent over $\mathbb{R}(x)$, then
	\begin{gather}
	A_{pq}(x_{pq}) = \frac{N_p(x_{pq})}{{\prod_{\substack{j=1\\j\neq p}}^k \langle(1,x_{pq}),e_j^p\rangle^{d_j}}},
	\hspace{.5cm}
	q=1,\ldots,d_p,
	\end{gather}
	where $N_p$ is a $q$-independent polynomial whose degree is at most $\ell_p$.
\item
	If $d_p \geq 2$, $\ell_p = m$ and $\langle\mu,\beta\rangle_x^2, \langle\mu,\beta\rangle_x \langle\mu,\beta\rangle'$ and $\left( \langle\mu,\beta\rangle' \right)^2$ are linearly independent over $\mathbb{R}(x)$, then
	\begin{gather}
	A_{pq}(x_{pq}) = \frac{N_p(x_{pq})}{{\prod_{\substack{j=1\\j\neq p}}^k \langle(1,x_{pq}),e_j^p\rangle^{d_j}}},
	\hspace{.5cm}
	q=1,\ldots,d_p,
	\end{gather}
	where $N_p$ is a $q$-independent polynomial whose degree is at most $m+2$.
\item
	If $d_p = 2$, $\ell_p < m$, and $\langle \mu, \beta \rangle^2, \langle \mu, \beta \rangle \partial_{x_{pq}} \langle \mu, \beta \rangle$ and $(\partial_{x_{pq}} \langle \mu, \beta \rangle)^2$ are linearly independent over $\mathbb{R}(x_{pq})$, and $\langle \mu, \beta \rangle_x$ and $\langle \mu, \beta \rangle'$ are linearly dependent over $\mathbb{R}(x)$ (see \Cref{linear independence} and \Cref{cor}), then
	\begin{gather}
	A_{pq}(x_{pq}) = \frac{N_{pq}(x_{pq})}{{\prod_{\substack{j=1\\j\neq p}}^k \langle(1,x_{pq}),e_j^p\rangle^{d_j}}},
	\hspace{.5cm}
	q=1,2,
	\end{gather}
	and $A(x):=A_{p1}(x) - A_{p2}(x)$ satisfy \eqref{the ODE}, where $N_{pq}$, $q=1,2$ is a polynomial whose degree is at most $\ell_p$.
	Additionally, if the factorization structure is the product Segre-Veronese factorization structure, then this case occurs if and only if $m\geq4$ and $k\geq2$, and we have
	\begin{gather}
	A_{pq}(x_{pq}) = N_p(x_{pq}),
	\hspace{.5cm}
	q=1,2,
	\end{gather}
	where $N_p$ is a $q$-independent polynomial of degree at most $\ell_p = 3$.
\item
	If $d_p = 2$, $\ell_p = m$, $\langle \mu, \beta \rangle$ and $\partial_{x_{pq}} \langle \mu, \beta \rangle$ are linearly independent over $\mathbb{R}[x_{pq}]$, and $\langle \mu, \beta \rangle_x$ and $\langle \mu, \beta \rangle'$ are linearly dependent over $\mathbb{R}(x)$, then
	\begin{gather}
	A_{pq}(x_{pq}) = \frac{N_{pq}(x_{pq})}{{\prod_{\substack{j=1\\j\neq p}}^k \langle(1,x_{pq}),e_j^p\rangle^{d_j}}},
	\hspace{.5cm}
	q=1,2,
	\end{gather}
	and $A(x):=A_{p1}(x) - A_{p2}(x)$ satisfy \eqref{the ODE}, where $N_{pq}$,, $q=1,2$ is a polynomial whose degree is $m+2$.
	Additionally, if the factorization structure is the Veronese factorization structure, then this case occurs if and only if $m=2$, and we have
	\begin{gather}
	A_{pq}(x_{pq}) = N_p(x_{pq}) + \nu_{pq}^0 (b+cx)(ac-b^2 + (b+cx)^2) +
	\nu_{pq}^1 ((ac-b^2)^2 - (b+cx)^4),
	\hspace{.5cm}
	q=1,2,
	\end{gather}
	where $N_p$ is a $q$-independent polynomial of degree at most 4 and $\langle \mu, \beta \rangle = \kappa(a + b(x_{p1}+x_{p2}) + cx_{p1}x_{p2})$, $a,b,c, \nu_{pq}^0, \nu_{pq}^1 \in \mathbb{R}$.
	Furthermore, if the factorization structure is the product Segre-Veronese factorization structure with $k\geq 2$, then this case occurs if and only if the factorization structure is 3-dimensional and corresponds to the partition $3=1+2$, in which case we have
	\begin{gather}
	A_{pq}(x_{pq}) = N_p(x_{pq}),
	\hspace{.5cm}
	q=1,2,
	\end{gather}
	for a $q$-independent polynomial $N_p$ of degree at most $m+2=5$.
\item
	If $d_p=1$, $\ell_p < m$, $\langle \mu, \beta \rangle^2, \langle \mu, \beta \rangle \partial_{x_{p1}} \langle \mu, \beta \rangle$ and $(\partial_{x_{p1}} \langle \mu, \beta \rangle)^2$ are linearly independent over $\mathbb{R}(x_{p1})$, then
	\begin{gather}
	A_{p1}(x_{p1}) = \frac{N_p(x_{p1})}{{\prod_{\substack{j=1\\j\neq p}}^k \langle(1,x_{p1}),e_j^p\rangle^{d_j}}},
	\end{gather}
	where $N_p$ is a polynomial of degree at most $\ell_p$.
\item
	If $d_p=1$, $\ell_p = m$, then
	\begin{gather}
	A_{p1}(x_{p1}) = \frac{N_p(x_{p1})}{{\prod_{\substack{j=1\\j\neq p}}^k \langle(1,x_{p1}),e_j^p\rangle^{d_j}}},
	\end{gather}
	where $N_p$ is a polynomial of degree at most $m+2$.
\end{enumerate}

\end{thm}
\begin{proof}
Note that the assumption of the case (1) allows $\partial_{x_{pq}} \langle \mu, \beta \rangle = 0$ and that \Cref{linear independence} shows that $\varphi(\beta)$ decomposes in grouped $p$-slots.
\Cref{deccor} provides shapes of a solution for any $d_p \geq 1$, and \eqref{a} relates them for $d_p\geq2$, subsuming \Cref{constant twist solution remark}.
Cases (2) and (6) are consequences of \Cref{final piece} and \Cref{LI - solutions}.
To address cases (3) and (7) note that \Cref{solutions 1} provide shapes of a solution and \Cref{LI - solutions} relate them.
Shapes of a solution for cases (4) and (5) follow respectively from \Cref{final piece} and \Cref{solutions 1}, and they are related via \Cref{coupled} and \Cref{eq lemma}.
\end{proof}

\begin{rem} \label[rem]{dec summed lemma}
Let $e_j^p = (e_{j,1}^p, e_{j,2}^p)$.
To establish notation for polynomials and their coefficients, we have $\Upsilon_p(x) = \sum_{j=0}^{\deg \Upsilon_p} \gamma_p^j x^j$ and $\tilde{\Upsilon}_p = \sum_{j=0}^{\deg \tilde{\Upsilon}_p} \tilde{\gamma}_p^j x^j$.
We now evaluate the following part of the scalar curvature
\begin{gather}\label{first line}
\sum_{q=1}^{d_p}
\frac{\langle\mu,\beta\rangle^{m+3}}
	{	\Delta_{pq}
		\langle\hat{\textbf{x}}_p,\Gamma_p\rangle
		\prod_{\substack{j=1 \\ j \neq p}}^k\langle (1,x_{pq}),e_j^p \rangle^{d_j}
	}
\partial_{x_{pq}}^2
\left(
	\frac{A_{pq}(x_{pq})}
		{\langle\mu,\beta\rangle^{m+1}}
	\prod_{\substack{j=1 \\ j \neq p}}^k\langle (1,x_{pq}),e_j^p \rangle^{d_j}
\right)
\end{gather}
under the assumption that $A_{pq}$, $q=1,\ldots,d_p$, satisfies part (1) of \Cref{main 1}.
This depends on whether $\langle \mu, \beta \rangle$ is constant or not.

In the former case, plugging
\begin{gather}
A_{pq}(x_{pq})=
\frac{1}{\prod_{j \in O_p} \langle(1,x_{pq}),e_j^p\rangle^{d_j}}
\int \int \Upsilon_p(x_{pq}) \prod_{j \in O_p} \langle(1,x_{pq}),e_j^p\rangle^{d_j-1} dx_{pq} dx_{pq},
\end{gather}
into \eqref{first line} gives
\begin{gather}
\frac{\langle \mu, \beta \rangle^2}{\langle\hat{\textbf{x}}_p,\Gamma_p\rangle}
\sum_{q=1}^{d_p}
\frac{\Upsilon_p(x_{pq})}{ \Delta_{pq} \prod_{j \in O_p} \langle (1,x_{pq}), e_j^p \rangle },
\end{gather}
which, using the partial fraction decomposition, yields
\begin{gather}
\frac{\langle \mu, \beta \rangle^2}{\langle\hat{\textbf{x}}_p,\Gamma_p\rangle}
\sum_{q=1}^{d_p}
\frac{1}{ \Delta_{pq} }
\left[
	\tilde{\Upsilon}_p(x_{pq}) + \sum_{j \in O_p} \frac{c_j}{ \langle (1,x_{pq}), e_j^p \rangle }
\right],
\end{gather}
where $c_j \in \mathbb{R}$ and $\deg \tilde{\Upsilon}_p = \deg \Upsilon_p - o_p \leq d_p$, since $\deg \Upsilon_p \leq L_p = d_p + o_p$ (see \Cref{Lp lemma} for explicit degrees).
In particular, if $o_p=0$, then $\Upsilon_p = \tilde{\Upsilon}_p$, and if $O_p=\{j\}$, then $c_j = \Upsilon_p(-e^p_{j,1}/e^p_{j,2})$.
Using formulae from \Cref{Vandermonde remark} we arrive at
\begin{gather}\label{dec summed const}
\frac{\langle \mu, \beta \rangle^2}{\langle\hat{\textbf{x}}_p,\Gamma_p\rangle}
\left[
	\tilde{\gamma}_p^{d_p-1} + \tilde{\gamma}_p^{d_p} \sigma_1 (x_{p1},\ldots,x_{pd_p}) +
	\sum_{j \in O_p} \frac{ c_j (-e_{j,2}^p)^{d_p-1} }{ \langle \textbf{x}_p, (e_j^p)^{\otimes d_p} \rangle }
\right].
\end{gather}

To address the case when $\langle \mu, \beta \rangle$ is not constant, we plug \eqref{part 1} into \eqref{first line} and obtain
\begin{gather}
\frac{\langle \mu, \beta \rangle^2}{\langle\hat{\textbf{x}}_p,\Gamma_p\rangle}
\sum_{q=1}^{d_p}
\frac{\Upsilon_p(x_{pq})}{ \Delta_{pq} (a+bx_{pq})^2 \prod_{j \in O_p} \langle (1,x_{pq}), e_j^p \rangle },
\end{gather}
which, using the partial fraction decomposition, equals
\begin{gather}
\frac{\langle \mu, \beta \rangle^2}{\langle\hat{\textbf{x}}_p,\Gamma_p\rangle}
\sum_{q=1}^{d_p}
\frac{1}{ \Delta_{pq} }
\left[
	\tilde{\Upsilon}_p(x_{pq}) + \frac{g_p^1}{(a+bx_{pq})} + \frac{g_p^2}{(a+bx_{pq})^2} + \sum_{j \in O_p} \frac{c_j}{ \langle (1,x_{pq}), e_j^p \rangle }
\right],
\end{gather}
where $g_p^1,g_p^2,c_j \in \mathbb{R}$, and $\deg \tilde{\Upsilon} = \deg \Upsilon - (o_p + 2)$ is at most $d_p-2$ if $o_p=k-1$ and $d_p-1$ otherwise.
Now, using the assumption $\varphi(\beta) = ins_p \left( (a,b)^{\otimes d_p} \otimes \tilde{\beta} \right)$ and formulae from \Cref{Vandermonde remark} we obtain
\begin{gather}\nonumber
\frac{ \langle \hat{\textbf{x}}_p, \tilde{\beta} \rangle^2 }{\langle \hat{\textbf{x}}_p, \Gamma_p \rangle}
\Bigg[
\tilde{\gamma}_p^{d_p-1} \langle \textbf{x}_p, (a,b)^{\otimes d_p} \rangle^2 +
g_p^1 (-b)^{d_p-1} \langle \textbf{x}_p, (a,b)^{\otimes d_p} \rangle +\\
g_p^2 (-b)^{d_p-1} \sum_{i=1}^{d_p} \prod_{\substack{q=1 \\ q \neq i}} (a+bx_{pq}) + 
\sum_{j \in O_p} c_j (-e_{j,2}^p)^{d_p-1} \frac{ \langle \textbf{x}_p, (a,b)^{\otimes d_p} \rangle^2 }{ \langle \textbf{x}_p, (e_j^p)^{\otimes d_p} \rangle }
\Bigg], \label{dec summed}
\end{gather}
where $\tilde{\gamma}_p^{d_p-1} = 0$ if $\deg \tilde{\Upsilon}_p \leq d_p-2$.
In particular, if $o_p=0$, hence $L_p = d_p+1$, then $\gamma_p^{d_p+1} = b^2 \tilde{\gamma}_p^{d_p-1}$, $b g_p^1 = \Upsilon_p'(-a/b)$, $g_p^2 = \Upsilon_p(-a/b)$, and if $o_p=1$ and $k\geq3$, i.e., $L_p = d_p+2$, then $\gamma_p^{d_p+2} = b^2 e_{j,2}^p \tilde{\gamma}_p^{d_p-1}$.
\end{rem}

\section{Complete solution set}\label{cs}
This section recasts all known explicit extremal toric Kähler metrics as extremal separable geometries corresponding to Segre and Veronese factorization structures, and obtains new extremal geometries by determining all extremal separable geometries corresponding to the product Segre-Veronese factorization structure.
In addition, we initiate the study of extremality for a more general class of factorization structures, the decomposable ones, and as a proof of concept for the general technique developed in \Cref{dec summed lemma}, provide examples of extremal geometries for the first non-trivial member of this class.
We also explain how the projective freedom in change of coordinates found in the literature corresponds to an isomorphism of factorization structures.

A particularly powerful aspect of separable geometries is that their scalar curvature can be explicitly expressed in momentum coordinates, offering a powerful new tool to advance the analytic study of these metrics.

\subsection{Preliminary lemma}

For the standard product Segre-Veronese factorization structure $\varphi:\mathfrak{h} \to V^*$ and any $\beta \in \mathfrak{h}$, we have
\begin{gather}
\langle \mu, \beta \rangle = \sum_{j=1}^k \langle (1,x_{j1}) \otimes \cdots \otimes (1,x_{jd_j}), \beta_j \rangle,
\end{gather}
as in \Cref{algebra remark}.
Let $S$ be the set of those $j \in \{1,\ldots,k\}$ for which the $j$-th summand of the above sum is non-constant, or equivalently,
$\beta_j \in ins_j \left( S^{d_j}W_j^* \otimes \langle (1,0)^{\otimes (m-d_j)} \rangle \right)$
does not lie on $\langle (1,0)^{\otimes m} \rangle$.

\begin{lemma}\label[lemma]{pSVfs lemma}
Let $\varphi$ be the $m$-dimensional product Segre-Veronese factorization structure.
If $\langle \mu, \beta \rangle$ and $\partial_{x_{pq}} \langle \mu, \beta \rangle$ are linearly independent over $\mathbb{R}(x_{pq})$, then
\begin{enumerate}
\item for $d_p \geq 1$,
	\begin{gather}\label{i1}
	\langle \mu, \beta \rangle^2, \hspace{.1cm} \langle \mu, \beta \rangle \partial_{x_{pq}} \langle \mu, \beta \rangle, \hspace{.1cm} (\partial_{x_{pq}} \langle \mu, \beta \rangle)^2
	\end{gather}
	are linearly independent over $\mathbb{R}(x_{pq})$.
\item for $d_p \geq 3$
	\begin{gather}\label{i2}
	\langle\mu,\beta\rangle_x^2, \langle\mu,\beta\rangle_x \langle\mu,\beta\rangle', \left( \langle\mu,\beta\rangle' \right)^2 
	\end{gather}
	are linearly independent over $\mathbb{R}(x)$.
\item for $d_p=2$ one of the following occurs:
	\begin{enumerate}
	\item $|S| = 1$, i.e., $\beta \in ins_p \left( S^2 W_p^* \otimes \langle (1,0)^{\otimes (m-2)} \right)$, and $\beta$ is indecomposable, i.e., $\langle \mu, \beta \rangle_x$ and $\langle \mu, \beta \rangle'$ are linearly dependent over $\mathbb{R}(x)$.
	\item $|S| \geq 2$, and \eqref{i2} is linearly independent over $\mathbb{R}(x)$.
	\end{enumerate}
\end{enumerate}
Additionally, a non-zero element $\varphi(\beta)$ is decomposable in $(p,q)$-th slot if and only if it is decomposable in the grouped $p$-slots if and only if
\begin{itemize}
\item[(i)]
$\varphi(\beta) =
ins_p \left( t \cdot (a,b)^{\otimes d_p} \otimes (1,0)^{\otimes (m-d_p)} \right)$ for some non-zero $(a,b)\in W_p^*$ and $t \in \mathbb{R} \backslash \{0\}$, or
\item[(ii)]
$\varphi(\beta) =
ins_p \left( (1,0)^{\otimes d_p} \otimes U \right)$ for some $U \in \sum_{\substack{j=1 \\ j \neq p}}^k ins_j \left( S^{d_j}W_j^* \otimes \langle (1,0)^{\otimes (m-d_p-d_j)} \rangle \right)$.
\end{itemize}
\end{lemma}
\begin{proof}
To prove (1), we use the technique described in \Cref{algebra remark} and view  $\langle \mu, \beta \rangle$ and $\partial_{x_{pq}} \langle \mu, \beta \rangle$ as elements of the algebra $\mathcal{A}$ over the field of rational functions $\mathbb{R}(x_{pq})$ generated by the scalars, and polynomials $\sigma_j(x_{r1},\ldots,x_{rd_r})$, $j=1,\ldots,d_r$, $r \in \{1,\ldots,k\} \backslash \{p\}$, and $\sigma_j(x_{p1}, \ldots, \hat{x}_{pq}, \ldots, x_{pd_p})$, $j=1,\ldots,d_p-1$: for $d_p=1$ the latter generator is omitted.
We equip $\mathcal{A}$ with a grading by declaring these polynomials as degree one elements, thus $\mathcal{A}$ is generated in degree one.
In particular, $\langle \mu, \beta \rangle$ and $\partial_{x_{pq}} \langle \mu, \beta \rangle$ are either degree one elements which are linearly independent in the vector space of degree one elements, or respectively degree one and degree zero elements.
In the former case, they can be completed into a basis, which in turn generate the entire $\mathcal{A}$, showing that vectors \eqref{i1} linearly independent.
In the latter case, vectors from \eqref{i1} belong respectively to degree two, one and zero graded pieces of $\mathcal{A}$, therefore they are independent.

To prove (2), note that \Cref{linear independence} concludes that $\langle \mu, \beta \rangle_x$ and $\langle \mu, \beta \rangle'$ are linearly independent over $\mathbb{R}(x)$.
The claim follows by the same argument as above with the algebra $\mathcal{A}$ being replaced with the algebra $\mathcal{B}$ over $\mathbb{R}(x)$ generated by the scalars, and polynomials $\sigma_j(x_{r1},\ldots,x_{rd_r})$, $j=1,\ldots,d_r$, $r \in \{1,\ldots,k\} \backslash \{p\}$, and $\sigma_j(x_{p1}, \ldots, \hat{x}_{pq}, \ldots, \hat{x}_{pr}, \ldots, x_{pd_p})$, $j=1,\ldots,d_p-1$.

For part (3), because of the assumption that $\langle \mu, \beta \rangle$ and $\partial_{x_{pq}} \langle \mu, \beta \rangle$ are independent, one distinguishes two complementary cases: $|S| = 1$ and $|S| \geq 2$.
In the former case, the assumption forces $\beta$ to be indecomposable.
In the latter case, one observes that for
\begin{gather}
\langle \mu, \beta \rangle =
\langle (1,x_{p1}) \otimes (1,x_{p2}), \beta_p \rangle +
\sum_{\substack{j \in S\\ j \neq p}} \langle (1,x_{j1}) \otimes \cdots \otimes (1,x_{jd_j}), \beta_j \rangle,
\end{gather}
$\langle \mu, \beta \rangle_x$ and $\langle \mu, \beta \rangle'$ are linearly independent over $\mathbb{R}(x)$.
The claim then follows by the same algebra-argument as above.

Finally, regarding the $m$-dimensional product Segre-Veronese factorization structure as a product of factorization structures as follows
\begin{gather}
S^{d_p}W_p^* \otimes \langle (1,0)^{\otimes (m-d_p)} \rangle + \langle (1,0)^{\otimes d_p} \rangle \otimes \sum_{\substack{j=1 \\ j \neq p}}^k ins_j \left( S^{d_j}W_j^* \otimes \langle (1,0)^{\otimes (m-d_p-d_j)} \rangle \right),
\end{gather}
and using \Cref{tensors split} on the 1-dimensional space determined by $ins_p \left( (a,b)^{\otimes d_p} \otimes K \right)$ proves the claim.
\end{proof}

Recall that \Cref{solutions 1} implies that solutions of the extremality equation in Veronese, Segre and the standard product Segre-Veronese cases are polynomials of degrees at most $m+2$, where $m$ is the dimension of the corresponding factorization structure.
We will use this fact throughout the following three subsection without explicit mention.

\subsection{Extremal twisted orthotoric geometries}
The extremality equation of twisted orthotoric geometry (see \Cref{examples of geoms} and \Cref{extr eq rem}) reads
\begin{gather}
\sum_{j=1}^m \frac{\langle \mu, \beta \rangle^{m+3}}{\Delta_j} \partial_{x_j}^2 \left( \frac{A_j(x_j)}{\langle \mu, \beta \rangle^{m+1}} \right) = \langle \mu, \alpha \rangle,
\end{gather}
and thus $l_p = m$ for any $p=1,\ldots,m$ (see \Cref{Lp lemma}).

\begin{thm}\label{veronese extremal}
The twisted orthotoric geometry is extremal if and only if one of the following holds:
\begin{enumerate}
\item
$\varphi(\beta) = t \cdot (a,b)^{\otimes m}$ for some $t \in \mathbb{R} \backslash \{0\}$, i.e., $\langle \mu, \beta \rangle = t \prod_{j=1}^m (a+bx_j),$ and
\begin{gather}
A_j(x_j) = (a+bx_j)^{m+1} \int \int \frac{\Upsilon(x_j)}{(a+bx_j)^{m+3}} dx_j dx_j,
\hspace{.2cm}
j=1,\ldots,m,
\end{gather}
where $\Upsilon$ is a $j$-independent polynomial such that $\deg \Upsilon \leq m$.
In particular, if $b=0$, then $A_j(x) = \tilde{\Upsilon}(x_j) + \nu_j^0 + \nu_j^1x_j$, where $\nu_j^1, \nu_j^2 \in \mathbb{R}$, $\tilde{\Upsilon}$ does not depend on $j$ and $\deg \tilde{\Upsilon} \leq m+2$, and if $b\neq 0$, then $A_j(x) = \bar{\Upsilon}(x_j) + (a+bx_j)^{m+1} (\nu_j^0 + \nu_j^1x_j)$, where $\nu_j^1, \nu_j^2 \in \mathbb{R}$, $\bar{\Upsilon}$ does not depend on $j$ and $\deg \bar{\Upsilon} \leq m$.
\item
$\varphi(\beta)$ does not decompose and exactly one of the following holds.
	\begin{enumerate}
	\item
	$m=2$ and for $j=1,2$ we have
	\begin{gather}
	A_j(x_j) = \Upsilon(x_j) + \nu_j^0 (b+cx)(ac-b^2 + (b+cx)^2) + \nu_j^1 ((ac-b^2)^2 - (b+cx)^4),
	\end{gather}
	where $\Upsilon$ is a $j$-independent polynomial of degree at most 4 and $\nu_j^0, \nu_j^1 \in \mathbb{R}$.
	\item
	$m \geq 3$ and $A_j = \Upsilon$, $j=1,\ldots, m$, where $\Upsilon$ is a $j$-independent polynomial such that $\deg \Upsilon \leq m+2$.
	\end{enumerate}
\end{enumerate}
\end{thm}
\begin{proof}
We distinguish two complementary cases: either $\langle \mu, \beta \rangle$ and $\partial_{x_{pq}} \langle \mu, \beta \rangle$ are linearly independent over $\mathbb{R}(x_j)$, or not.
\Cref{linear independence} shows that they are dependent if and only if $\varphi(\beta) = t\cdot (a,b)^{\otimes m}$ for some $(a,b) \in W^*$ and $t \in \mathbb{R} \backslash \{0\}$.
If they are dependent, part (1) of \Cref{main 1} provides the shape of solutions of the corresponding extremality equation and \Cref{Lp lemma} gives the degree $L_p = m$.
In turn, either \Cref{dec summed lemma} or direct use of identities from \Cref{Vandermonde remark} show that these are indeed solutions (and even that the scalar curvature is constant when $\deg \Upsilon \leq m-1$ and zero when $\deg \Upsilon \leq m-2$).

To address the case (2b), note using \Cref{pSVfs lemma} that \eqref{i2} are linearly independent, and hence by \Cref{LI - solutions} all $A_j$, $j=1,\ldots,m$, are given by the same polynomial $\Upsilon$ whose degree is at most $m+2$ by \Cref{minimal}.
To verify that this solves the equation is a direct but lengthy computation using formulae of \Cref{appendix1}.
We note that the same computation for $m=2$ proves that $A_1$ and $A_2$, given by the same polynomial $\Upsilon$ of degree at most 4, satisfy the corresponding extremality equation.

To prove the remaining case (2a), we work with the shape of solutions provided in part (5) of \Cref{main 1}.
Under the above observation, it suffices to show that polynomials $A_j(x_j) - \Upsilon(x_j)$, $j=1,2$, satisfy the extremality equation.
This is a simple direct verification.
\end{proof}

\begin{rem}
One can easily verify that the polynomial
\begin{gather}
\nu_j^0 (b+cx)(ac-b^2 + (b+cx)^2) + \nu_j^1 ((ac-b^2)^2 - (b+cx)^4),
\end{gather}
a part of $A_1$ and $A_2$ from \Cref{veronese extremal} (2a), can be written as a product of $a+2bx+cx^2$ with a quadratic polynomial, which are orthogonal via the bilinear form $\langle p, q \rangle = p_0q_2-p_1q_1+p_2q_0$.
This matches the description of extremal regular ambitoric geometries originally found in \cite{apostolov2016ambitoric}, and its interpretation in terms of CR twists in \cite{apostolov_cr_2020}.
Therefore, the case (2a) recovers extremal regular ambitoric structures, while the other cases recover extremal twisted orthotoric geometries from \cite{apostolov_cr_2020}.
\end{rem}

\begin{proposition}
Any two extremal twisted orthotoric geometries $(g_{\beta_i},J_{\beta_i})$, $i=1,2$, corresponding to decomposable $\varphi(\beta_1)$ and $\varphi(\beta_2)$ as in \Cref{veronese extremal} part (1), are isomorphic by projective change of coordinates described in \Cref{coordinate change general}.
In particular, every such extremal twisted orthotoric geometry is an extremal orthotoric geometry.
Stated differently, extremal separable geometries corresponding to $\beta \in \im \psi$ are isomorphic, where $\psi$ is the factorization curve of the Veronese factorization structure.
\end{proposition}
\begin{proof}
Using \Cref{coordinate change general} and coordinates change $x_j = 1/(a+by_j)$, $j=1,\ldots,m$, given by the matrix
\begin{gather}
g=
\begin{bmatrix}
a & 1\\
b & 0
\end{bmatrix}
\end{gather}
shows that families of geometries corresponding to $\varphi(\beta_1) = (1,0)^{\otimes m}$ and $\varphi(\beta_2) = (a,b)^{\otimes m}$ are the same.
One proceeds similarly in case of scaling.
\end{proof}

\subsection{Extremal twisted product of Riemann surfaces}
The extremality equation of the twisted product of Riemann surfaces (see \Cref{examples of geoms} and \Cref{extr eq rem}) reads
\begin{gather}
\sum_{j=1}^m \langle \mu, \beta \rangle^{m+3} \partial_{x_j}^2 \left( \frac{A_j(x_j)}{\langle \mu, \beta \rangle^{m+1}} \right) = \langle \mu, \alpha \rangle.
\end{gather}

\begin{thm}
The twisted product of Riemann surfaces is extremal if and only if one of the following holds:
\begin{enumerate}
\item
$\varphi(\beta) = t \cdot (1,0)^{\otimes m}$, $t \in \mathbb{R} \backslash \{0\}$, i.e., $\langle \mu, \beta \rangle = t$, and $A_i$, $i=1,\ldots,m$, is a polynomial of degree at most 3.
\item
$\varphi(\beta) = ins_j \left( (a,b) \otimes (1,0)^{\otimes (m-1)} \right)$, $b \neq 0$, i.e., $\langle \mu, \beta \rangle = a+bx_j$, and
\begin{gather}
A_j(x_j) = (a+bx_j)^{m+1} \int \int \frac{\Upsilon(x_j)}{(a+bx_j)^{m+3}} dx_j dx_j
\end{gather}
and $A_i$, $i\neq j$, is a polynomial of degree at most 2, and
\begin{gather}
\gamma_2 + b^2 \sum_{\substack{i=1 \\ i \neq j}}^m A_i'' = 0,
\end{gather}
where $\Upsilon(x) = \sum_{r=0}^2 \gamma_r(a+bx)^r$.
In particular, $A_j(x_j) = \bar{\Upsilon}(x_j) + (a+bx_j)^{m+1} (\nu_j^0 + \nu_j^1x_j)$, where $\deg \bar{\Upsilon} \leq 2$, and $\nu_j^0, \nu_j^1 \in \mathbb{R}$.
\item
There exists $S \subset \{1,\ldots,m\}$ of cardinality at least two and a vector $(a_j,b_j)$, $b_j \neq 0$, for each $j \in S$ such that $\varphi(\beta) = \sum_{j \in S} ins_j \left( (a_j,b_j) \otimes (1,0)^{\otimes (m-1)} \right)$, i.e., $\langle \mu, \beta \rangle = \beta_0 + \sum_{j \in S} \beta_j x_j$ for some non-zero $\beta_j \in \mathbb{R}$, and either
\begin{enumerate}
\item
	$m=2$ and $A_j(x_j) = \sum_{r=0}^4 a_{jr} x_j^r$, $j=1,2$, are such that
	\begin{align}
	a_{24} &= - \frac{\beta_2^2}{\beta_1^2} a_{14} \\
	a_{23} &= - 4 \frac{\beta_0 \beta_2}{\beta_1^2} a_{14} + \frac{\beta_2}{\beta_1} a_{13} \\
	a_{22} &= - 6 \frac{\beta_0^2}{\beta_1^2} + 3 \frac{\beta_0}{\beta_1} a_{13} - a_{12},
	\end{align}
	where $a_{1r} \in \mathbb{R}$, $r=0,\ldots,4$, and $a_{20},a_{21} \in \mathbb{R}$, or
\item
	$m\geq3$ and $A_j$, $j \in S$, is a polynomial of degree at most 1, $A_r$, $r \notin S$, is a polynomial of degree at most 2 and
	\begin{gather}
	\sum_{r \notin S} A_r'' = 0.
	\end{gather}
\end{enumerate}
\end{enumerate}
\end{thm}

\begin{proof}
Given $\langle \mu, \beta \rangle = \beta_0 + \sum_{j=1}^m \beta_j x_j$, we distinguish cases according to the cardinality $|S|$ of $S = \{ j \in \{1,\ldots,m\} \mid \beta_j \neq 0\}$, thereby motivating the division of the theorem into three distinct cases.

Observe that if $|S| \leq 1$, i.e., $S \subset \{ p \}$ for some $p$, then $\langle \mu, \beta \rangle$ and $\partial_{x_p} \langle \mu, \beta \rangle$ are linearly dependent over $\mathbb{R}[x_p]$, while for cardinalities strictly larger than one, the vectors are independent for any $p \in S$ and dependent otherwise.
Then, by \Cref{pSVfs lemma}, we have that for $S$ such that $|S|>1$ and for $p \in S$, the vectors $\langle \mu, \beta \rangle^2$, $\langle \mu, \beta \rangle \partial_{x_p}\langle \mu, \beta \rangle$ and $\left( \partial_{x_p}\langle \mu, \beta \rangle \right)^2$ are independent.

The part (1) follows from \Cref{main 1} and \Cref{Lp lemma} by a trivial direct computation which, additionally, shows when the geometry has zero or constant scalar curvature.

To address part (2), note that \Cref{Lp lemma} and \Cref{main 1} give
\begin{gather}
A_j(x_j) = (a+bx_j)^{m+1} \int \int \frac{\Upsilon(x_j)}{(a+bx_j)^{m+3}} dx_j dx_j,\\
A_i(x_i) = \int \int \Upsilon_i(x_i) dx_i dx_i,
\hspace{.2cm}
i \in \{1,\ldots,m\} \backslash \{j\},
\end{gather}
where $\Upsilon$ is a polynomial of degree at most 2, and $\Upsilon_i$ for $i\neq j$ are polynomials of degree at most 1, and that for such $A_r$, $r=1,\ldots,m$, the extremality equation takes form
\begin{gather}
\Upsilon_j(x_j) + (a + bx_j)^2 \sum_{\substack{i=1 \\ i \neq j}}^m \Upsilon_i(x_i) =
\alpha_0 + \sum_{r=1}^m \alpha_r x_r.
\end{gather}
Therefore, it must be $\deg \Upsilon_i = 0$, $i \neq j$, and $\gamma_2 + b^2 \sum_{\substack{i=1 \\ i \neq j}}^m \Upsilon_i = 0$.
By integrating we find that $A_j(x_j) = \bar{\Upsilon}(x_j) + (a+bx_j)^{m+1} (\nu_j^0 + \nu_j^1x_j)$, where $\deg \bar{\Upsilon} \leq 2$, and that $A_i$, $i\neq j$, is a polynomial of degree at most 2.
Note that the scalar curvature can be easily found explicitly.

In part (3), by \Cref{Lp lemma}, we have $\ell_j = 2$ for $j\in S$ and $L _r = 1$ otherwise.
This prompts the distinction between $m=2$, in which case $\ell_1=\ell_2=m=2$, and $m\geq 3$, in which case $\ell_j < m$ for all $j=1,\ldots,m$.
Parts (1), (6) and (7) of \Cref{main 1} provides the shape of solutions.

For $m=2$; $A_1$ and $A_2$ are polynomials of degrees at most 4.
By plugging these into the extremality equation and gathering coefficients at $x_1^2 x_2^2$, $x_1^2 x_2$ and $x_1^2$ we obtain conditions from the theorem, respectively.
All the other conditions, which arise as coefficients at higher order terms, are dependent with these.

For $m\geq3$; $A_j$, $j \in S$, is a polynomial of degree at most 2 and $A_r$, $r \notin S$, is a polynomial of degree at most 3, and the extremality equation reads
\begin{gather}
\sum_{j \in S} \langle \mu, \beta \rangle^{m+3} \partial_{x_j}^2 \left( \frac{A_j(x_j)}{\langle \mu, \beta \rangle^{m+1}} \right) + 
\sum_{r \notin S} \langle \mu, \beta \rangle^2 A''_r(x_r) = 
\alpha_0 + \sum_{r=1}^m \alpha_r x_r.
\end{gather}
Clearly, for $r \notin S$ it must be $\deg A_r \leq 2$, and by comparing coefficients at quadratic terms $x_i x_j$, $i,j \in S$, we conclude $\deg A_j \leq 1$ for $j \in S$, and
\begin{gather}
\sum_{r \notin S} A_r'' = 0.
\end{gather}
\end{proof}

A change of coordinates statement similar to the one in the previous subsection appears at the end of the next subsection as part of a more general formulation.

\begin{rem}
The case (3a) recovers ambitoric product, Calabi and negative orthotoric ansatz of \cite{apostolov2016ambitoric,apostolov2015ambitoric}, all of which were found to belong into the same family in \cite{apostolov2017levi}, and reinterpreted in terms of CR twists in \cite{apostolov_cr_2020}.
The remaining cases cover extremal twists of Riemann surfaces from \cite{apostolov_cr_2020}.
\end{rem}

\subsection{Extremal separable geometries: the standard product Segre-Veronese factorization structure} \label{product extremality}
The equation reads
\begin{gather}
\sum_{i=1}^k \sum_{r=1}^{d_i}
\frac{\langle \mu, \beta \rangle^{m+3}}{\Delta_{ir}}
\partial_{x_{ir}}^2 \left( \frac{A_{ir}(x_{ir})}{\langle \mu, \beta \rangle^{m+1}} \right) =
\langle \mu, \alpha \rangle.
\end{gather}

We freely use the notation $\Upsilon_i(x) = \sum_j \gamma_i^j x^j$ for coefficients of given polynomials without further explicit mention.

\begin{thm}\label{main pSVfs}
Let $(g_\beta, J_\beta)$ be the separable geometry associated to the product Segre-Veronese factorization structure which is distinct from the Veronese and Segre factorization structures, and let
\begin{gather}
S = \{ i \in \{1,\ldots,k\} \mid \exists r \in \{1,\ldots,d_i\}: \beta_{ir} \neq 0 \}.
\end{gather}
The geometry is extremal if and only if one of the following holds:
\begin{enumerate}
\item
	$\varphi(\beta) = t \cdot (1,0)^{\otimes m}$ for some $t \in \mathbb{R} \backslash \{0\}$, i.e., $\langle \mu, \beta \rangle = t$, and for $p=1,\ldots,k$ and $q=1,\ldots,d_p$ we have
	\begin{gather}
	A_{pq}(x_{pq}) = \tilde{\Upsilon}_p(x_{pq}) + \nu_{pq}^0 + \nu_{pq}^1 x_{pq},
	\end{gather}
	where $\tilde{\Upsilon}_p$ is a $q$-independent polynomial of degree at most $d_p+2$, and $\nu^0_{pq}, \nu^1_{pq} \in \mathbb{R}$ for any $p,q$.
\item
	$\varphi(\beta) = ins_p \left( t \cdot (a,b)^{\otimes d_p} \otimes (1,0)^{\otimes (m-d_p)} \right)$ for some $t \in \mathbb{R} \backslash \{0\}$, $b \neq 0$, i.e., $\langle \mu, \beta \rangle = t \prod_{q=1}^{d_p} (a+bx_{pq})$, and
	\begin{align}
	A_{pq}(x_{pq}) &= (a+bx_{pq})^{m+1} \int \int \frac{\Upsilon_p(x_{pq})}{(a+bx_{pq})^{m+3}} dx_{pq} dx_{pq},
	\hspace{.5cm}
	q=1,\ldots,d_p,\\
	A_{ir}(x_{ir}) &= \int \int \Upsilon_i(x_{ir}) dx_{ir} dx_{ir},
	\hspace{.5cm}
	i\neq p, \hspace{.1cm} r=1,\ldots,d_i,
	\end{align}
	with
	\begin{gather}
	\deg \Upsilon_p \leq d_p+1\\
	\deg \Upsilon_i \leq d_i-1
	\end{gather}
	are such that
	\begin{gather}
	\frac{\gamma_p^{d_p+1}}{b^2} + \sum_{\substack{i=1 \\ i \neq p}}^k \gamma_i^{d_i-1} = 0.
	\end{gather}
\item 
	The geometry corresponds to the factorization structure
	\begin{gather}\label{quick}
	S^dW_1^* \otimes \langle (1,0) \rangle + \langle (1,0)^{\otimes d} \rangle \otimes W^*_2,
	\end{gather}
	there exists $q \in \{1,\ldots,d\}$ such that $\langle \mu, \beta \rangle$ and $\partial_{x_{1q}} \langle \mu, \beta \rangle$ are linearly independent over $\mathbb{R}[x_{1q}]$, $A_2$ is a polynomial of degree at most two and
	\begin{gather}
	A_{1q}(x_{1r}) = \Upsilon(x_{1r}), \hspace{.2cm} r=1,\ldots,d,
	\end{gather}
	where $\Upsilon$ is a $r$-independent polynomial of degree at most $m$, such that $2\gamma^m + A_2'' = 0$.
\item 
	There exists $p \in \{1,\ldots,k\}$ and $q \in \{1,\ldots, d_p\}$ such that $\langle \mu, \beta \rangle$ and $\partial_{x_{pq}} \langle \mu, \beta \rangle$ are linearly independent over $\mathbb{R}[x_{pq}]$, the factorization structure is not isomorphic with \eqref{quick}, and for
	\begin{gather}
	A_{pq}(x_{pq}) = \Upsilon_p(x_{pq}), \hspace{.2cm} q=1,\ldots,d_p, \hspace{.2cm} p \in S,
	\end{gather}
	where $\Upsilon_p$ is a $q$-independent polynomial of degree at most $d_p+1$,
	and
	\begin{gather}
	A_{ir}(x_{ir}) = \int \int \Upsilon_i(x_{ir}) dx_{ir} dx_{ir}, \hspace{.2cm} r=1,\ldots,d_i, \hspace{.2cm} i \notin S,
	\end{gather}
	where $\Upsilon_i$ is a $r$-independent polynomial of degree at most $d_i-1$, exactly one of the following holds:
	\begin{enumerate}
	\item
		$S = \{ p \}$, and 
		\begin{gather}
		(m-d_p) (m+1-d_p) \gamma_p^{d_p+1} + \sum_{i \notin S} \gamma_i^{d_i-1} = 0.
		\end{gather}
	\item $S = \{ p, \tilde{p} \}$, $m = d_p + d_{\tilde{p}}$, i.e., the geometry corresponds to the factorization structure
		\begin{gather}
		S^{d_p}W_p^* \otimes \langle (1,0)^{d_{\tilde{p}}} \rangle + \langle (1,0)^{d_p} \rangle \otimes S^{d_{\tilde{p}}}W_{\tilde{p}},
		\end{gather}
		and
		\begin{gather}
		\gamma_p^{d_p+1} + \gamma_{\tilde{p}}^{d_{\tilde{p}+1}} = 0.
		\end{gather}
	\item
		$|S| \geq 2$, $k \geq 3$, and $\deg \Upsilon_j \leq d_j$ for $j \in S$, and $\sum_{i \notin S} \gamma_i^{d_i-1} = 0$.
	\end{enumerate}
\end{enumerate}
\end{thm}
\begin{proof}
We distinguish cases according to the linear independence of $\langle \mu, \beta \rangle$ and $\partial_{pq} \langle \mu, \beta \rangle$ over $\mathbb{R}[x_{pq}]$.
By \Cref{pSVfs lemma}, exactly two situations can occur: either $\varphi(\beta) = ins_p ( t \cdot (a,b)^{\otimes d_p} \otimes (1,0)^{\otimes (m-d_p)} )$ for some $a,b \in \mathbb{R}$ not both zero and $t \in \mathbb{R} \backslash \{0\}$, or the two functions are independent.
The first case can be branched further as in part (1) and (2) of the theorem.

To prove (1), note that \Cref{main 1} and \Cref{Lp lemma} give
\begin{gather}
A_{pq}(x_{pq}) = \int \int \Upsilon_p(x_{pq}) dx_{pq} dx_{pq},
\end{gather}
where $\deg \Upsilon_p \leq d_p$, i.e., $A_{pq}(x_{pq}) = \tilde{\Upsilon}_p(x_{pq}) + \nu_{pq}^0 + \nu_{pq}^1 x_{pq}$ such that $\deg \tilde{\Upsilon} \leq d_p+2$. Using formulae \eqref{Vandermonde identity} and \eqref{extended Vandermonde identity} it is plain to see that these satisfy the extremality equation.

To address part (2), \Cref{main 1} and \Cref{Lp lemma} give
\begin{gather}
A_{pq}(x_{pq}) = (a+bx_{pq})^{m+1} \int \int \frac{\Upsilon_p(x_{pq})}{(a+bx_{pq})^{m+3}} dx_{pq} dx_{pq}
\end{gather}
for $q=1,\ldots,d_p$, where $\deg \Upsilon \leq d_p+1$, and for $i \neq p$ and any $r=1,\ldots,d_i$
\begin{gather}
A_{ir}(x_{ir}) = \int \int \Upsilon_i(x_{ir}) dx_{ir} dx_{ir},
\end{gather}
where $\deg \Upsilon_i \leq d_i$.
According to \Cref{dec summed lemma} and formulae from \Cref{Vandermonde remark}, the left hand side of the extremality equation, up to the factor $t^2$, reads
\begin{gather}
	\frac{\gamma_p^{d_p+1}}{b^2} \langle \mu, \beta \rangle^2 +
	g_1 (-b)^{d_p-1} \langle \mu, \beta \rangle +
	g_2 (-b)^{d_p-1} \sum_{i=1}^{d_p} \prod_{\substack{q=1 \\ q \neq i}} (a+bx_{pq}) + \nonumber \\
\prod_{q=1}^{d_p} (a+bx_{pq})^2 \sum_{\substack{i=1 \\ i \neq p}}^k \left( \gamma_i^{d_i-1} +\gamma_i^{d_i} \sigma_1(x_{i1},\ldots,x_{id_i}) \right),
\end{gather}
which forces $\deg \Upsilon_i \leq d_i-1$, $i \neq p$, and
\begin{gather}
\frac{\gamma_p^{d_p+1}}{b^2} + \sum_{\substack{i=1 \\ i \neq p}}^k \gamma_i^{d_i-1} = 0.
\end{gather}

To conclude parts (3) and (4), we now assume that $\langle \mu, \beta \rangle$ and $\partial_{x_{pq}} \langle \mu, \beta \rangle$ are linearly independent, i.e., there exists $p \in S$ and $q \in \{1,\ldots,d_p\}$ such that they are independent.
\Cref{Lp lemma} shows that $\ell_p = d_p+1$ for $p \in S$ and $L_j = d_j$ for $j \notin S$.

Using \Cref{pSVfs lemma} and \Cref{main 1} we find the following cases:
\begin{enumerate}
\item[(i)]
	If there exists $p \in S$ such that $d_p+1 = m$, then the involved factorization structure is $S^dW_1^* \otimes \langle (1,0) \rangle + \langle (1,0)^{\otimes d} \rangle \otimes W^*_2$, where $d_p$ was renamed to $d$.
	Thus, $1\leq |S| \leq 2$, and for any such cardinality we have
	\begin{gather}
	A_{1q}(x_{1q}) = \Upsilon(x_{1q}), \hspace{.2cm} q=1,\ldots,d,
	\end{gather}
	where $\Upsilon$ is a $q$-independent polynomial of degree at most $m+2$, and $A_2$ is a polynomial of degree at most $4 - |S|$.
\item[(ii)]
	If all $p \in S$ are so that $d_p+1 < m$, then for $p \in S$ and $q=1,\ldots,d_p$ we have
	\begin{gather}
	A_{pq}(x_{pq}) = \Upsilon_p(x_{pq})
	\end{gather}
	with $\Upsilon_p$ being $q$-independent polynomial such that $\deg \Upsilon_p \leq d_p+1$, and for $i \notin S$ and $r=1,\ldots,d_i$,
	\begin{gather}
	A_{ir}(x_{ir}) = \int \int \Upsilon_i(x_{ir}) dx_{ir} dx_{ir},
	\end{gather}
	where $\deg \Upsilon_i \leq d_i$.
\end{enumerate}

We solve the case (i) by distinguishing $|S| = 1$ and $|S| = 2$.
To solve the case $|S| = 1$, we substitute $A_{1q}$, $q=1,\ldots,d$, and $A_2$ as above and evaluate the extremality equation using \Cref{lemma m=d+1} to obtain
\begin{gather}
\sum_{q=1}^d \frac{\langle \mu, \beta \rangle^{m+3}}{\Delta_{1q}} \partial_{x_{1q}}^2 \frac{\Upsilon(x_{1q})}{\langle \mu, \beta \rangle^{m+1}} + \langle \mu, \beta \rangle^2 \partial_{x_2}^2 A_2(x_2) =   \nonumber \\ \nonumber
(m+2)(m+1) (\beta_0)^2 \gamma^{m+2} (\sigma_1)^2 - 2(m+1) \beta_0 \gamma^{m+1} \sigma_1 f + 2 \gamma^{m} f^2 \\
 - 2(m+1) \gamma^{m+1} f f^+ +
2(m+1)(m+2) \beta_0 \gamma^{m+2} \sigma_1 f^+ + (m+1)(m+2) \gamma^{m+2} (f^+)^2 +  \nonumber \\
6a_3 x_2 f^2 + 12 a_3 \beta_0 x_2 f + 2a_2 f^2,
\end{gather}
modulo affine-linear terms, where $A_2(x) = a_0 + a_1 x + a_2 x^2 + a_3 x^3$ and $f_p$ and $f_p^+$ from \Cref{lemma m=d+1} were renamed to $f$ and $f^+$, i.e., $\langle \mu, \beta \rangle = \beta_0 + f$.
Clearly, it must be $\deg A_2 \leq 2$, and then the above expression can be viewed as a linear combination in the $\mathbb{R}$-algebra generated by the elementary symmetric polynomials.
Provided $f$ is not a real multiple of $\sigma_1$ or $\sigma_d$, the vectors $(\sigma_1)^2, \sigma_1 f, f^2, f f^+, \sigma_1 f^+$ and $(f^+)^2$ are linearly independent, which yields the result in this generic case.
The rest is a trivial verification, thus proving (3) in the case $|S|=1$.
We wish to remark that throughout the verification one encounters a case when $f$ is a multiple of $\sigma_d$ and $\beta_0 = 0$, whose corresponding solutions would be $\Upsilon$ of degree at most $m+2$ satisfying $\gamma^m+a_2=0$, which, however, is not valid since $\langle \mu, \beta \rangle$ and its derivative are not independent in this case.

For $|S| = 2$ case, we apply $\partial_{x_2}^2$ onto the extremality equation with $A_{1q}$, $q=1,\ldots,d_1$, and $A_2$ as above, which yields
\begin{gather}
2\sum_{q=1}^d \frac{\Upsilon''(x_{1q})}{\Delta_{1q}} + m(m-1)A_2'' = 0
\end{gather}
and is equivalent to
\begin{gather}
2(m+2)(m+1)\gamma^{m+2} h_2(x_{11}, \ldots, x_{1d}) + 2(m+1)m \gamma^{m+1} \sigma_1(x_{11}, \ldots, x_{1d_1}) + \nonumber \\
2m(m-1) \gamma^m + m(m-1)A_2'' = 0.
\end{gather}
Necessarily, $\gamma^{m+2} = \gamma^{m+1} = 0$ and $2\gamma_m + A_2'' = 0$.
Using \Cref{lemma m=}, it is easy to verify that subjected to these conditions the extremality equation is satisfied.
This completes the proof of (3).

We continue with case (ii) for which the extremality equation reads
\begin{gather}\label{aa}
\sum_{p \in S} \sum_{q=1}^{d_p} \frac{\langle \mu, \beta \rangle^{m+3}}{\Delta_{pq}} \partial_{x_{pq}}^2 \frac{\Upsilon_p(x_{pq})}{\langle \mu, \beta \rangle^{m+1}} +
\langle \mu, \beta \rangle^2 \sum_{i \notin S} \sum_{r=1}^{d_i} \frac{\Upsilon_i(x_{ir})}{\Delta_{ir}} =
\langle \mu, \alpha \rangle.
\end{gather}
Now, since the first group of sums does not contain any variables $x_{ir}$ for $i \notin S$, and the second group adds up to
\begin{gather}
\langle \mu, \beta \rangle^2
\sum_{i \notin S}
\left[ \gamma_i^{d_i-1} + \gamma_i^{d_i} \sigma_1(x_{j1},\ldots,x_{jd_j}) \right],
\end{gather}
it must be that $\deg \Upsilon_i \leq d_i-1$ for all $i \notin S$.
Therefore, using \Cref{lemma m=}, the equation \eqref{aa} modulo affine-linear terms reads
\begin{gather}
\sum_{p \in S}
\Bigg[ d_p(d_p+1) \gamma_p^{d_p+1} \left( \sum_{j=1}^k f_j \right)^2 - 2(m+1) (d_p+1) \gamma_p^{d_p+1} \left( \sum_{j=1}^k f_j \right) f_p + \nonumber \\
(m+1)(m+2) \gamma_p^{d_p+1} f_p^2 \Bigg] + \left( \sum_{j=1}^k f_j \right)^2 \sum_{i \notin S} \gamma_i^{d_i-1} = 0.
\end{gather}
Since the functions $f_j$ are independent, we solve this equation by comparing coefficients at $f_if_j$.
For any $j \in S$ the coefficient at $f_j^2$ is
\begin{gather}\label{g1}
\left[ \sum_{p \in S} d_p(d_p+1) \gamma_p^{d_p+1} \right] + (m+1)(m-2d_j) \gamma_j^{d_j+1} + \sum_{i \notin S} \gamma_i^{d_i-1} = 0,
\end{gather}
while for any $i,j\in S$, $i\neq j$, the coefficient at $f_if_j$ is
\begin{gather}\label{g2}
\left[ \sum_{p \in S} d_p(d_p+1) \gamma_p^{d_p+1} \right] - (m+1)(d_j+1) \gamma_j^{d_j+1} - (m+1)(d_i+1) \gamma_i^{d_i+1} + \sum_{i \notin S} \gamma_i^{d_i-1} = 0,
\end{gather}
where the latter condition is void if $|S|=1$.
We proceed according to the cardinality $|S|$.

If $|S| = 1$, then
\begin{gather}
(m-d_p) (m+1-d_p) \gamma_p^{d_p+1} + \sum_{i \notin S} \gamma_i^{d_i-1} = 0,
\end{gather}
which proves (4a).
If $|S| \geq 2$, then for any pair $i,j \in S$, $i \neq j$, by subtracting \eqref{g2} corresponding to $i$ and $j$ from \eqref{g1} corresponding respectively to $j$ and $i$, we obtain
\begin{gather}
(m+1-d_j) \gamma_j^{d_j+1} = - (d_i+1) \gamma_i^{d_i+1}\\
(m+1-d_i) \gamma_i^{d_i+1} = - (d_j+1) \gamma_j^{d_j+1},
\end{gather}
which yields
\begin{gather}
(m-d_i-d_j)\gamma_j^{d_j+1} = 0.
\end{gather}
Therefore, if $m=d_i+d_j$ we find $\gamma_i^{d_i+1} + \gamma_j^{d_j+1} = 0$, hence proving (4b), and otherwise we have $\gamma_j^{d_j+1} = 0$ for any $j \in S$, which forces $\sum_{i \notin S} \gamma_i^{d_i-1} = 0$, thus proving (4c).
\end{proof}

\begin{proposition}
Let $(g_{\beta_i},J_{\beta_i})$, $i=1,2$, be extremal separable geometries corresponding to the product Segre-Veronese factorization structure different from the Veronese factorization structure and such that $\varphi(\beta_i) = ins_p \left( t_i (a_i,b_i)^{\otimes d_p} \otimes (1,0)^{\otimes (m-d_p)} \right)$, $i=1,2$, where $t_1t_2b_1b_2\neq0$, i.e., as in part (2) of \Cref{main pSVfs}.
Then, these geometries are isomorphic by a projective change off coordinates as in \Cref{coordinate change general}.
\end{proposition}
\begin{proof}
Using \Cref{coordinate change general} and coordinates change $x_{pq} = \frac{a_1-a_2+b_2y_{pq}}{b_1}$, $q=1,\ldots,d_p$, given by the matrix
\begin{gather}
g_p=
\begin{bmatrix}
b_1 & a_2-a_1\\
0 & b_2
\end{bmatrix}
\end{gather}
shows that geometries are the same.
One proceeds similarly in case of scaling.
\end{proof}

\subsection{Decomposable Segre-Veronese factorization structure with 2 intersection points}
For $\pi_1, \pi_2 \in \mathbb{R}$, $\pi_2 \neq 0$, factorization curves $\psi_1,\ldots,\psi_{k-1}$, $k \geq 3$, of the decomposable Segre-Veronese factorization structure
\begin{gather}\label{two points fs}
\sum_{i=1}^{k-1}
ins_i \left( S^{d_i} W_i^* \otimes \langle (1,0)^{\otimes (m-d_i)} \rangle \right) +
\langle (\pi_1,\pi_2)^{\otimes d_1} \otimes (1,0)^{\otimes m-d_1-d_k} \rangle \otimes S^{d_k} W_k^*
\end{gather}
intersect mutually at a unique point $\langle (1,0)^{\otimes m} \rangle$, while the factorization curve $\psi_k$ intersects only the curve $\psi_1$, at $\langle (\pi_1,\pi_2)^{\otimes d_1} \otimes (1,0)^{\otimes (m-d_1)} \rangle$.
The corresponding extremality equation reads
\begin{gather}\label{ex eq}
\sum_{r=1}^{d_1}
\frac{\langle \mu, \beta \rangle^{m+3}}{\Delta_{1r} (\pi_1 + \pi_2 x_{1r})^{d_k}} \partial_{x_{1r}}^2 \left( \frac{A_{1r}(x_{1r})}{\langle \mu, \beta \rangle^{m+1}} (\pi_1 + \pi_2 x_{1r})^{d_k} \right) + \nonumber \\
\sum_{i=2}^{k-1} \sum_{r=1}^{d_i}
\frac{\langle \mu, \beta \rangle^{m+3}}{\Delta_{ir}} \partial_{x_{1r}}^2 \left( \frac{A_{ir}(x_{ir})}{\langle \mu, \beta \rangle^{m+1}} \right) +
\sum_{r=1}^{d_k} \frac{\langle \mu, \beta \rangle^{m+3}}{\Delta_{kr} \prod_{q=1}^{d_1} (\pi_1 + \pi_2 x_{1q})} \partial_{x_{kr}}^2 \left( \frac{A_{kr}(x_{kr})}{\langle \mu, \beta \rangle^{m+1}} \right) =
\langle \mu, \alpha \rangle.
\end{gather}

\begin{proposition}
If one of the following holds, then the separable geometry $(g_\beta, J_\beta)$ associated to the factorization structure \eqref{two points fs} is extremal;
\begin{enumerate}
\item $\varphi(\beta) \in span\{(1,0)^{\otimes m}\}$, i.e., $\langle \mu, \beta \rangle$ is constant, and
	\begin{gather}
	A_{1q}(x_{1q}) = \frac{1}{(\pi_1 + \pi_2 x_{1q})^{d_k}} \int \int \Upsilon_1(x_{1q}) (\pi_1 + \pi_2 x_{1q})^{d_k-1} dx_{1q} dx_{1q}, \hspace{.2cm} q=1,\ldots,d_1,
	\end{gather}
	and for $p=2,\ldots,k$ we have
	\begin{gather}
	A_{pq}(x_{pq}) = \int \int \Upsilon_p(x_{pq}) dx_{pq} dx_{pq}, \hspace{.2cm} q=1,\ldots,d_p,
	\end{gather}
	where $\deg \Upsilon_1 \leq d_1+1$, $\deg \Upsilon_p \leq d_p$, $p=2,\ldots,k-1$, $\deg \Upsilon_k \leq d_k-1$, and $\gamma_k^{d_k-1} = - (-\pi_2)^{d_1-1} \Upsilon_1(-\pi_1/\pi_2)$.
\item $\varphi(\beta) \in span\{(\pi_1,\pi_2)^{\otimes d_1} \otimes (1,0)^{\otimes m-d_1}\}$, i.e., $\langle \mu, \beta \rangle = t \prod_{q=1}^{d_1} (\pi_1 + \pi_2 x_{1q})$, $t \in \mathbb{R} \backslash \{0\}$, and for $q=1,\ldots,d_1,$
	\begin{gather}
	A_{1q}(x_{1q}) = (\pi_1 + \pi_2 x_{1q})^{m+1-d_k} \int \int \Upsilon_1(x_{1q}) (\pi_1 + \pi_2 x_{1q})^{-m+d_k-4} dx_{1q} dx_{1q},
	\end{gather}
	and for $p=2,\ldots,k$ we have
	\begin{gather}
	A_{pq}(x_{pq}) = \int \int \Upsilon_p(x_{pq}) dx_{pq} dx_{pq}, \hspace{.2cm} q=1,\ldots,d_p,
	\end{gather}
	where $\deg \Upsilon_1 \leq d_1+2$, $\deg \Upsilon_p \leq d_p-1$, $p=2,\ldots,k$, $\Upsilon_1(-\pi_1/\pi_2) = 0$, and $\frac{\gamma_1^{d_1+2}}{\pi_2^3} + \sum_{i=2}^{k-1} \gamma_i^{d_i-1} = 0$.
\item $\varphi(\beta) \in span\{ (a,b)^{\otimes d_1} \otimes (1,0)^{\otimes (m-d_1)} \}$ is such that $b \neq 0$ and $\langle (a,b) \rangle \neq \langle (\pi_1,\pi_2) \rangle$, hence $\langle \mu, \beta \rangle = t \prod_{q=1}^{d_1} (a + b x_{1q})$, $t \in \mathbb{R} \backslash \{0\}$, and
	\begin{gather}
	A_{1q}(x_{1q}) = \frac{(a+bx_{1q})^{m+1}}{(\pi_1 + \pi_2 x_{1q})^{d_k}} \int \int \Upsilon_1(x_{1q}) \frac{(\pi_1 + \pi_2 x_{1q})^{d_k-1}}{(a+bx_{1q})^{m+3}} dx_{1q} dx_{1q}, \hspace{.2cm} q=1,\ldots,d_1,
	\end{gather}
	and for $p=2,\ldots,k$ we have
	\begin{gather}
	A_{pq}(x_{pq}) = \int \int \Upsilon_p(x_{pq}) dx_{pq} dx_{pq}, \hspace{.2cm} q=1,\ldots,d_p,
	\end{gather}
	where $\deg \Upsilon_1 \leq d_1+2$, $\deg \Upsilon_p \leq d_p-1$, $p=2,\ldots,k$, are such that $\gamma_k^{d_k-1} (a-b\pi_1/\pi_2)^2 = - (-\pi_2)^{d_1-1} \Upsilon_1(-\pi_1/\pi_2)$, and $\gamma_1^{d_1+2}/(\pi_2b^2) + \sum_{i=2}^{k-1} \gamma_i^{d_i-1} = 0$.
\end{enumerate}
\end{proposition}

\begin{proof}
All three possibilities for $\varphi(\beta)$ are covered by part (1) of \Cref{main 1};
\begin{gather}
A_{1q}(x_{1q}) = \frac{(a+bx_{1q})^{m+1}}{(\pi_1 + \pi_2 x_{1q})^{d_k}} \int \int \Upsilon_1(x_{1q}) \frac{(\pi_1 + \pi_2 x_{1q})^{d_k-1}}{(a+bx_{1q})^{m+3}} dx_{1q} dx_{1q}, \hspace{.2cm} q=1,\ldots,d_1,
\end{gather}
and for $p=2,\ldots,k$,
\begin{gather}
A_{pq}(x_{pq}) = \int \int \Upsilon_p(x_{pq}) dx_{pq} dx_{pq}, \hspace{.2cm} q=1,\ldots,d_p,
\end{gather}
where, using \Cref{Lp lemma},
\begin{enumerate}
\item if $\varphi(\beta) \in span\{(1,0)^{\otimes m}\}$, then $\deg \Upsilon_1 \leq d_1+1$ and $\deg \Upsilon_p \leq d_p$, $p=2,\ldots,k$.
\item if  $\varphi(\beta) \in span\{(\pi_1,\pi_2)^{\otimes d_1} \otimes (1,0)^{\otimes m-d_1}\}$, then $\deg \Upsilon_1 \leq d_1+2$ and $\deg \Upsilon_p \leq d_p$, $p=2,\ldots,k$.
\item if $\varphi(\beta) = (a,b)^{\otimes d_1} \otimes (1,0)^{\otimes m-d_1}$ is such that $b \neq 0$ and $\langle (a,b) \rangle \neq \langle (\pi_1,\pi_2) \rangle$, then $\deg \Upsilon_1 \leq d_1+2$ and $\deg \Upsilon_p \leq d_p$, $p=2,\ldots,k$.
\end{enumerate}
Now we verify that these expressions satisfy the extremality equation precisely under the conditions from this proposition.
We proceed case by case and use \Cref{dec summed lemma} to sum the sums up.
\begin{enumerate}
\item We find that the left hand side of the extremality equation equals
\begin{gather}
\langle \mu, \beta \rangle^2
\left( \frac{\gamma_1^{d_1}}{\pi_2}-\frac{\pi_1}{\pi_2^2} \gamma_1^{d_1+1} + \frac{\gamma_1^{d_1+1}}{\pi_2} \sigma_1(x_{11}, \ldots, x_{1d_1}) + \frac{\Upsilon_1(-\pi_1/\pi_2) (-\pi_2)^{d_1-1}}{\langle \textbf{x}_1, (\pi_1,\pi_2)^{\otimes d_1} \rangle} \right) + \nonumber \\
\langle \mu, \beta \rangle^2
\left[ \sum_{i=2}^{k-1} \gamma_i^{d_i-1} + \gamma_i^{d_i} \sigma_1(x_{i1},\ldots,x_{id_i}) \right] +
\frac{\langle \mu, \beta \rangle^2}{\langle \textbf{x}_1, (\pi_1,\pi_2)^{\otimes d_1} \rangle}
\left[ \gamma_k^{d_k-1} + \gamma_k^{d_k} \sigma_1(x_{k1},\ldots,x_{kd_k}) \right],
\end{gather}
and is linear if and only if $\deg \Upsilon_k \leq d_k-1$ and $\gamma_k^{d_k-1} = - (-\pi_2)^{d_1-1} \Upsilon_1(-\pi_1/\pi_2)$.

\item We use an analogous computation to the one from \Cref{dec summed lemma}.
To this end we work with the partial fraction decomposition of $\Upsilon_1(x)/(\pi_1+\pi_2x)^3$, the formula
\begin{gather}
\sum_{j=1}^m \frac{1}{x_j^3 \Delta_j} = \frac{(-1)^{m-1}}{\sigma_m} \cdot \frac{\sigma_{m-1}^2 - \sigma_m \sigma_{m-2}}{\sigma_m^2},
\end{gather}
which can be derived from \eqref{extended Vandermonde identity} for $p=2$ by substituting $x_j \mapsto 1/x_j$, and notation $\sigma_j\{\pi_1+\pi_2x_{ir}\}$ for the $j$-th elementary symmetric polynomial in variables $\pi_1+\pi_2x_{i1}, \ldots, \pi_1+\pi_2x_{id_i}$.
Therefore, the left hand side of the extremality equation evaluates to the expression
\begin{gather}
\frac{\gamma_1^{d_1+2}}{\pi_2^3} \langle \textbf{x}_1, (\pi_1,\pi_2)^{\otimes d_1} \rangle^2 +
g_1^1(-\pi_2)^{d_1-1} \langle \textbf{x}_1, (\pi_1,\pi_2)^{\otimes d_1} \rangle + \nonumber \\
g^2_1(-\pi_2)^{d_1-1} \sum_{i=1}^{d_1} \prod_{\substack{q=1 \\ q \neq i}}^{d_1} (\pi_1+\pi_2 x_{1q}) + \nonumber \\
\frac{g^31 (-\pi_2)^{d_1-1}}{\langle \textbf{x}_1, (\pi_1,\pi_2)^{\otimes d_1} \rangle} \left[ \sigma_{d_1-1}^2 \{\pi_1+\pi_2x_{1r}\} - \sigma_{d_1}\{\pi_1+\pi_2x_{1r}\} \cdot \sigma_{d_1-2}\{\pi_1+\pi_2x_{1r}\} \right] + \nonumber \\
\langle \textbf{x}_1, (\pi_1,\pi_2)^{\otimes d_1} \rangle^2 \sum_{i=2}^{k-1} \left[ \gamma_i^{d_i-1} + \gamma_i^{d_i} \sigma_1(x_{i1},\ldots,x_{id_i}) \right] + \nonumber \\
\langle \textbf{x}_1, (\pi_1,\pi_2)^{\otimes d_1} \rangle \left[ \gamma_k^{d_k-1} + \gamma_k^{d_k} \sigma_1(x_{k1}, \ldots, x_{kd_k}) \right],
\end{gather}
which is linear if and only if $g_3^1=0$, i.e., $\Upsilon_1(-\pi_1/\pi_2) = 0$, $\deg \Upsilon_i \leq d_i-1$, $i=2,\ldots,k$, and $\frac{\gamma_1^{d_1+2}}{\pi_2^3} + \sum_{i=2}^{k-1} \gamma_i^{d_i-1} = 0$.

\item Finally,
\begin{gather}
\frac{\gamma_1^{d_1+2}}{\pi_2 b^2} \langle \textbf{x}_1, (a,b)^{\otimes d_1} \rangle^2 +
g_1^1(-b)^{d_1-1} \langle \textbf{x}_1, (a,b)^{\otimes d_1} \rangle +
g^2_1(-b)^{d_1-1} \sum_{i=1}^{d_1} \prod_{\substack{q=1 \\ q \neq i}}^{d_1} (a+bx_{1q}) + \nonumber \\
c_3(-\pi_2)^{d_1-1} \frac{\langle \textbf{x}_1, (a,b)^{\otimes d_1} \rangle^2}{\langle \textbf{x}_1, (\pi_1,\pi_2)^{\otimes d_1} \rangle} +
\langle \textbf{x}_1, (a,b)^{\otimes d_1} \rangle^2
\left[ \sum_{i=2}^{k-1} \gamma_i^{d_i-1} + \gamma_i^{d_i} \sigma_1(x_{i1},\ldots,x_{id_i}) \right] + \nonumber \\
\frac{\langle \textbf{x}_1, (a,b)^{\otimes d_1} \rangle^2}{\langle \textbf{x}_1, (\pi_1,\pi_2)^{\otimes d_1} \rangle}
\left[ \gamma_k^{d_k-1} + \gamma_k^{d_k} \sigma_1(x_{k1},\ldots,x_{kd_k}) \right]
\end{gather}
is linear if and only if $\deg \Upsilon_i \leq d_i-1$, $i=2,\ldots,k$, $\gamma_k^{d_k-1} = - (-\pi_2)^{d_1-1} c_3$, and $\gamma_1^{d_1+2}/(\pi_2b^2) + \sum_{i=2}^{k-1} \gamma_i^{d_i-1} = 0$, and one infers $c_3 (a- b\pi_1/ \pi_2)^2 = \Upsilon_1(-\pi_1/\pi_2)$.
\end{enumerate}
\end{proof}

\section{Appendix}\label{appendix1}

This section proves various formulae used in this article.

The following remark connects Vandermonde identities (see \cite{apostolov_hamiltonian_2006}) to identities \eqref{crutial identity S-V} provided by factorization structures.
\begin{rem}\label[rem]{Vandermonde remark}
In the case of the Veronese factorization structure $\varphi(\mathfrak{h})=S^mW^*$ we write $x_j$ instead of $x_{1j}$ and $\Delta_j$ instead of $\Delta_{1j}$. Using \eqref{crutial identity S-V} for Veronese factorization structure with $\varphi(\beta)=(1,0)^{\otimes m}$ we get
\begin{align}\label{pre-Vandermonde identity}
\sum_{r=1}^m
W_{ir}
V_{rj}=
\delta_{ij},
\end{align}
where
\begin{center}
$W=
\begin{bmatrix}
\frac{1}{\Delta_1}&\frac{\sigma_1(\hat{x}_1)}{\Delta_1}&\cdots&\frac{\sigma_{m-1}(\hat{x}_1)}{\Delta_1}\\
\vdots&\vdots&\cdots&\vdots\\
\frac{1}{\Delta_m}&\frac{\sigma_1(\hat{x}_m)}{\Delta_m}&\cdots&\frac{\sigma_{m-1}(\hat{x}_m)}{\Delta_m}
\end{bmatrix}
\hspace{1cm}
V=
\begin{bmatrix}
x_1^{m-1}&\cdots&x_m^{m-1}\\
-x_1^{m-2}&\cdots&-x_m^{m-2}\\
\vdots&\cdots&\vdots\\
(-1)^{m-1}&\cdots&(-1)^{m-1}
\end{bmatrix}$,
\end{center}
$\sigma_r$ is the $r$th elementary symmetric polynomial in variables $x_1,\ldots,x_m$, and $\sigma_{r-1}(\hat{x}_j):=\partial_{x_j}\sigma_r$, i.e. $\sigma_{r-1}(\hat{x}_j)$ is the $(r-1)$st elementary symmetric polynomial in variables $x_1, \ldots,$ $x_{j-1}, x_{j+1}, \ldots, x_m$.
Observe, $\sigma_r=\sigma_r(\hat{x}_j)+x_j\sigma_{r-1}(\hat{x}_j)$ with $\sigma_0=1$.

Reading \eqref{pre-Vandermonde identity} as $VW=id$ provides us with the Vandermonde identity
\begin{align}\label{Vandermonde identity}
\sum_{j=1}^m
\frac{x_j^{m-s}\sigma_{r-1}(\hat{x}_j)}{\Delta_j}=
(-1)^{s-1}\delta_{rs}
\hspace{1cm}
\text{for any}
\hspace{.5cm}
r,s=1,\ldots,m.
\end{align}
This identity extends (see Appendix B in \cite{apostolov_hamiltonian_2006})
\begin{align}\label{higher Vandermonde}
\sum_{j=1}^m
\frac{x_j^{m+k}\sigma_{r-1}(\hat{x}_j)}{\Delta_j}=
\sum_{s=0}^k
(-1)^sh_{k-s}\sigma_{r+s},
\end{align}
where $k$ is a non-negative integer, $r=1,\ldots,m$ and $h_k$ is the $k$th complete homogeneous symmetric polynomial ($h_0=1$). In particular, for $r=1$ we have
\begin{align}\label{extended Vandermonde identity}
\sum_{j=1}^m
\frac{x_j^{m-1+p}}{\Delta_j}=h_p,
\end{align}
where $p$ is a nonnegative integer.
In addition, the transformation $x_j\mapsto1/x_j$ for $r=1$ in \eqref{Vandermonde identity} and for $p=1$ in \eqref{extended Vandermonde identity}, $j=1,\ldots,m$, gives
\begin{align}\label{negative Vandermonde identity}
\begin{split}
\sum_{j=1}^m
\frac{x_j^{s-2}}{\Delta_j}&=
(-1)^{m-1}
\frac{\delta_{s1}}{\sigma_m},
\hspace{1cm}
s=1,\ldots,m\\
\sum_{j=1}^m
\frac{x_j^{-2}}{\Delta_j}&=
(-1)^{m-1}
\frac{\sigma_{m-1}}{\sigma_m^2}.
\end{split}
\end{align}
Finally, the transformation $x_j\mapsto a + bx_j$ for $s=1$ in \eqref{negative Vandermonde identity} gives
\begin{align}
\begin{split}
\sum_{j=1}^m
\frac{(a + bx_j)^{-1}}{\Delta_j}&=
\frac{(-b)^{m-1}}{\prod_{j=1}^m(a + bx_j)}\\
\sum_{j=1}^m
\frac{(a + bx_j)^{-2}}{\Delta_j}&=
(-b)^{m-1}
\frac{
		\sum_{i=1}^m
		\prod_{\substack{j=1\\j\neq i}}^m
		(a + bx_j)
	}
	{\prod_{j=1}^m(a + bx_j)^2},
\end{split}
\end{align}
for $r=1$ in \eqref{Vandermonde identity}
\begin{align}
\sum_{j=1}^m
\frac{(a + bx_j)^{m-s}}{\Delta_j}=
\delta_{1s} b^{m-1}
\hspace{1cm}
\text{for any}
\hspace{.5cm}
s=1,\ldots,m,
\end{align}
and for $p=1$ in \eqref{extended Vandermonde identity}
\begin{align}
\sum_{j=1}^m
\frac{(a + bx_j)^m}{\Delta_j}=b^m\sigma_1+b^{m-1}ma.
\end{align}
\end{rem}

\begin{lemma}\label[lemma]{linear sigma}
For $r,s=1,\ldots,m$ and $l=0,\ldots,m$, the sum
\begin{gather}\label{id}
\sum_{j=1}^m \sigma_{r-1}(\hat{x}_j) \sigma_{s-1}(\hat{x}_j) \frac{x_j^{m-l}}{\Delta_j} =
\begin{cases}
(-1)^l \sigma_{r+s-l-1} & \text{ if } r,s \geq l+1 \\
(-1)^{l-1} \sigma_{r+s-l-1} & \text{ if } r,s \leq l \\
0 & \text{ otherwise}
\end{cases}
\end{gather}
is linear in $\sigma_i$, $i=1,\ldots,m$.
\end{lemma}
\begin{proof}
We remark that both of the following identities \eqref{id1} and \eqref{id2} can be proven using $\sigma_r = \sigma_r(\hat{x}_j) + x_j \sigma_{r-1}(\hat{x}_j)$.
We fix $r$ and $s$, and assume without a loss of generality that $r \geq s$.
We consider two cases.
First, let $l$ be such that $0 \leq l \leq r-1$.
Then, using the Vandermonde identity \eqref{Vandermonde identity} and
\begin{gather}\label{id1}
\sigma_{r-1} =
\sum_{a=0}^{m-r}
\frac{(-1)^a \sigma_{r+a}}{x_j^{a+1}},
\end{gather}
we find that \eqref{id} equals
\begin{gather}
\sum_{a=0}^{m-r} (-1)^a \sigma_{r+a} \sum_{j=1}^m \sigma_{s-1}(\hat{x}_j) \frac{x_j^{m-l-a-1}}{\Delta_j} =
(-1)^l \sigma_{r+s-l-1} \sum_{a=0}^{m-r} \delta_{s,l+a+1}.
\end{gather}
Secondly, for $l$ such that $m \geq l \geq r$, we use
\begin{gather}\label{id2}
\sigma_{r-1}(\hat{x}_j) =
\sum_{a=0}^{r-1} (-1)^a x_j^a \sigma_{r-1-a}
\end{gather}
to find that \eqref{id} equals
\begin{gather}
\sum_{a=0}^{r-1} (-1)^a \sigma_{r-1-a} \sum_{j=1}^m \sigma_{s-1}(\hat{x}_j) \frac{x_j^{m-l+a}}{\Delta_j} =
(-1)^{l-1}\sigma_{r+s-l-1} \sum_{a=0}^{r-1} \delta_{s,l-a}.
\end{gather}
\end{proof}

Observe that for $l=0$, the evaluation of \eqref{id} can be written as
\begin{gather}\label{sum for l=0}
\sum_{j=1}^m \sigma_{r-1}(\hat{x}_j) \sigma_{s-1}(\hat{x}_j) \frac{x_j^m}{\Delta_j} =
\sigma_{r+s-1},
\end{gather}
where $\sigma_j = 0$ for $j > m$.

\begin{rem}
We wish to offer an alternative non-constructive proof based on language of Bochner-flat manifolds used in \cite{apostolov_cr_2020}. It was shown in \cite{apostolov_hamiltonian_2006,bryant2001bochner} that an orthotoric geometry given by a single polynomial $A_j = \Upsilon$, $j=1,\ldots,m$, of degree at most $m+2$ is Bochner-flat, and hence any of its twists is extremal (see \cite{apostolov_cr_2020}). This means that its scalar curvature
\begin{gather}
\sum_{j=1}^m \langle \mu, \beta \rangle^2 \frac{\Upsilon''(x_j)}{\Delta_j} - 2(m+1) \langle \mu, \beta \rangle (\partial_{x_j} \langle \mu, \beta \rangle) \frac{\Upsilon'(x_j)}{\Delta_j} +
(m+2)(m+1) (\partial_{x_j} \langle \mu, \beta \rangle)^2 \frac{\Upsilon(x_j)}{\Delta_j}
\end{gather}
is a linear combination of $\sigma_r$, $r=0,\ldots,m$, for any $\langle \mu, \beta \rangle$. Using $\Upsilon(x) = x^{m-l}$, $l=0,\ldots,m$, and Vandermonde identities, this claim shows that
\begin{gather}
\sum_{j=1}^m (\partial_{x_j} \langle \mu, \beta \rangle)^2 \frac{x_j^{m-l}}{\Delta_j}
\end{gather}
is a linear combination of $\sigma_r$, $r=0,\ldots,m$. Choosing $\langle \mu, \beta \rangle$ to be $\sigma_r$ and $\sigma_r + \sigma_s$ proves that \eqref{id} is linear in $\sigma_i$, $i=0,\ldots,m$.
\end{rem}

\begin{lemma}\label[lemma]{sigma m+1}
We have
\begin{align}
\sum_{j=1}^m \sigma_{r-1}(\hat{x}_j) \sigma_{s-1}(\hat{x}_j) \frac{x_j^{m+1}}{\Delta_j} = &
\sigma_r \sigma_s - \sigma_{r+s}, \\
\sum_{j=1}^m \sigma_{r-1}(\hat{x}_j) \sigma_{s-1}(\hat{x}_j) \frac{x_j^{m+2}}{\Delta_j} = &
\sigma_1 \sigma_r \sigma_s - \sigma_{r+1} \sigma_s - \sigma_r \sigma_{s+1} + \sigma_{r+s+1}, \\
\sum_{j=1}^m \sigma_{r-1}(\hat{x}_j) \sigma_{s-1}(\hat{x}_j) \frac{x_j^{m+3}}{\Delta_j} = &
(\sigma_1 \sigma_1 - \sigma_2) \sigma_r \sigma_s - \sigma_1 ( \sigma_r \sigma_{s+1} + \sigma_{r+1} \sigma_s) + \nonumber \\
&\sigma_r \sigma_{s+2} + \sigma_{r+2} \sigma_s + \sigma_{r+1} \sigma_{s+1} - \sigma_{r+s+2},
\end{align}
where $\sigma_j = 0$ for $j > m$.
\end{lemma}
\begin{proof}
We use $x_j \sigma_{r-1}(\hat{x}_j) = \sigma_r - \sigma_r(\hat{x}_j)$ to find
\begin{gather}
\sum_{j=1}^m \sigma_{r-1}(\hat{x}_j) \sigma_{s-1}(\hat{x}_j) \frac{x_j^{m+1}}{\Delta_j} =
\sigma_r \sum_{j=1}^m \sigma_{s-1}(\hat{x}_j) \frac{x_j^m}{\Delta_j} -
\sum_{j=1}^m \sigma_{r}(\hat{x}_j) \sigma_{s-1}(\hat{x}_j) \frac{x_j^m}{\Delta_j},
\end{gather}
which together with \eqref{higher Vandermonde} and \eqref{sum for l=0} proves the first claim.
To prove the second claim, one expands
\begin{gather}
\sum_{j=1}^m \sigma_{r-1}(\hat{x}_j) \sigma_{s-1}(\hat{x}_j) \frac{x_j^{m+2}}{\Delta_j} =
\sigma_r \sum_{j=1}^m \sigma_{s-1}(\hat{x}_j) \frac{x_j^{m+1}}{\Delta_j} -
\sum_{j=1}^m \sigma_r(\hat{x}_j) \sigma_{s-1}(\hat{x}_j) \frac{x_j^{m+1}}{\Delta_j},
\end{gather}
and then apply the first claim.
The third claim follows by an analogous computation.
\end{proof}

\begin{lemma}\label[lemma]{lemma m=}
Denote
\begin{gather}
f_i = \sum_{r=1}^{d_i} \beta_{ir} \sigma_i(x_{i1},\ldots,x_{id_i})
\end{gather}
for $i=1,\ldots,k$, so in the product Segre-Veronese factorization structure we have $\langle \mu, \beta \rangle = \beta_0 + \sum_{i=1}^k f_i$.
Suppose that there is $p$ such that $f_p$ is non-constant.
Then, for a polynomial $\Upsilon_p$ of degree at most $d_p+1$ we have
\begin{gather}
\sum_{q=1}^{d_p} \frac{\langle \mu, \beta \rangle^{m+3}}{\Delta_{pq}} \partial_{x_{pq}}^2 \frac{\Upsilon_p(x_{pq})}{\langle \mu, \beta \rangle^{m+1}} = \label{oo} \\
d_p(d_p+1) \gamma_p^{d_p+1} \left( \sum_{j=1}^k f_j \right)^2 - 2(m+1) (d_p+1) \gamma_p^{d_p+1} \left( \sum_{j=1}^k f_j \right) f_p + (m+1)(m+2) \gamma_p^{d_p+1} f_p^2,
\end{gather}
modulo affine-linear terms.
\end{lemma}
\begin{proof}
Using formulae of this section we find that the equation \eqref{oo} equals
\begin{gather}
d_p(d_p+1) \gamma_p^{d_p+1} \langle \mu, \beta \rangle^2 - 2(m+1) \langle \mu, \beta \rangle I + (m+1)(m+2) J,
\end{gather}
where
\begin{gather}
I = \sum_{q=1}^{d_p} \left( \partial_{x_{pq}}\langle \mu, \beta \rangle \right) \frac{\Upsilon_p'(x_{pq})}{\Delta_{pq}} =
\sum_{r=1}^{d_p} \sum_{l=0}^{d_p} (l+1) \gamma_p^{l+1} \beta_{pr} \sum_{q=1}^{d_p} \frac{\sigma_{r-1}(\hat{x}_{pq}) x^l}{\Delta_{pq}} = \nonumber \\
\sum_{r=1}^{d_p}(-1)^{r-1} (d_p+1-r) \gamma_p^{d_p+1-r} \beta_{pr} + (d_p+1) \gamma_p^{d_p+1} f_p
\end{gather}
and
\begin{gather}
J = \sum_{q=1}^{d_p} \left( \partial_{x_{pq}}\langle \mu, \beta \rangle \right)^2 \frac{\Upsilon_p(x_{pq})}{\Delta_{pq}} =
\sum_{r,s=1}^{d_p} \sum_{l=0}^{d_p+1} \gamma_p^l \beta_{pr} \beta_{ps} \sigma_{r-1}(\hat{x}_{pq}) \sigma_{s-1}(\hat{x}_{pq}) \frac{x_{pq}^l}{\Delta_{pq}} = \nonumber \\
\sum_{l=0}^{d_p} \gamma_p^l \left[ \sum_{r,s=1}^l (-1)^{l-1} \beta_{pr} \beta_{ps} \sigma_{r+s-l-1} + \sum_{r,s=l+1}^{d_p} (-1)^l \beta_{pr} \beta_{ps} \sigma_{r+s-l-1} \right] + \nonumber \\
\gamma_p^{d_p+1} f_p^2 - \gamma_p^{d_p+1} \sum_{r,s=1}^{d_p} \beta_{pr} \beta_{ps} \sigma_{r+s},
\end{gather}
which completes the proof.
\end{proof}

\begin{lemma} \label[lemma]{lemma m=d+1}
Denote
\begin{gather}
f_i = \sum_{r=1}^{d_i} \beta_{ir} \sigma_r(x_{i1},\ldots,x_{id_i}), \\
f_i^+ = \sum_{r=1}^{d_i} \beta_{ir} \sigma_{r+1}(x_{i1},\ldots,x_{id_i}), \hspace{.2cm}
f_i^{++} = \sum_{r=1}^{d_i} \beta_{ir} \sigma_{r+2}(x_{i1},\ldots,x_{id_i}),
\end{gather}
for $i=1,\ldots,k$, so in the product Segre-Veronese factorization structure we have $\langle \mu, \beta \rangle = \beta_0 + \sum_{i=1}^k f_i$.
Suppose that there is exactly one $p$ such that $f_p$ is non-constant and that $m=d_p+1$.
Then, for a polynomial $\Upsilon_p$ of degree at most $d_p+3$ we have
\begin{gather}
\sum_{q=1}^{d_p} \frac{\langle \mu, \beta \rangle^{m+3}}{\Delta_{pq}} \partial_{x_{pq}}^2 \frac{\Upsilon_p(x_{pq})}{\langle \mu, \beta \rangle^{m+1}} = \nonumber \\ \nonumber
(m+2)(m+1) (\beta_0)^2 \gamma_p^{d_p+3} (\sigma_1)^2 - 2(m+1) \beta_0 \gamma_p^{d_p+2} \sigma_1 f_p + 2 \gamma_p^{d_p+1} f_p^2 \\
 - 2(m+1) \gamma_p^{d_p+2} f_p f_p^+ +
2(m+1)(m+2) \beta_0 \gamma_p^{d_p+3} \sigma_1 f_p^+ + (m+1)(m+2) \gamma_p^{d_p+3} (f_p^+)^2 \label{oo1}
\end{gather}
modulo affine-linear terms.
\end{lemma}
\begin{proof}
Using formulae of this section we find that the equation \eqref{oo1} equals
\begin{gather}
\left[ (d_p+3)(d_p+2) \gamma_p^{d_p+3} (\sigma_1 \sigma_1 - \sigma_2) +
(d_p+2)(d_p+1) \gamma_p^{d_p+2} \sigma_1 +
(d_p+1)d_p \gamma_p^{d_p+1} \right]
\langle \mu, \beta \rangle^2 \nonumber \\
- 2(m+1) \langle \mu, \beta \rangle I +
(m+1)(m+2) J,
\end{gather}
where
\begin{gather}
I = \sum_{q=1}^{d_p} \left( \partial_{x_{pq}}\langle \mu, \beta \rangle \right) \frac{\Upsilon_p'(x_{pq})}{\Delta_{pq}} =
\sum_{r=1}^{d_p} \sum_{l=0}^{d_p+2} (l+1) \gamma_p^{l+1} \beta_{pr} \sum_{q=1}^{d_p} \frac{\sigma_{r-1}(\hat{x}_{pq}) x^l}{\Delta_{pq}} = \nonumber \\
\sum_{r=1}^{d_p}(-1)^{r-1} (d_p+1-r) \gamma_p^{d_p+1-r} \beta_{pr} + (d_p+1) \gamma_p^{d_p+1} f_p  + \nonumber \\
(d_p+2) \gamma_p^{d_p+2} (\sigma_1 f_p - f_p^+) +  
(d_p+3) \gamma_p^{d_p+3}  \left[ (\sigma_1 \sigma_1 - \sigma_2) f_p - \sigma_1 f_p^+ + f_p^{++} \right]
\end{gather}
by $x_{pq} \sigma_{s-1}(\hat{x}_{pq}) = \sigma_s - \sigma_s(\hat{x}_{pq})$, and
\begin{gather}
J = \sum_{q=1}^{d_p} \left( \partial_{x_{pq}}\langle \mu, \beta \rangle \right)^2 \frac{\Upsilon_p(x_{pq})}{\Delta_{pq}} =
\sum_{q=1}^{d_p}
\sum_{r,s=1}^{d_p} \sum_{l=0}^{d_p+3} \gamma_p^l \beta_{pr} \beta_{ps} \sigma_{r-1}(\hat{x}_{pq}) \sigma_{s-1}(\hat{x}_{pq}) \frac{x_{pq}^l}{\Delta_{pq}} = \nonumber \\
\sum_{l=0}^{d_p} \gamma_p^l \left[ \sum_{r,s=1}^l (-1)^{l-1} \beta_{pr} \beta_{ps} \sigma_{r+s-l-1} + \sum_{r,s=l+1}^{d_p} (-1)^l \beta_{pr} \beta_{ps} \sigma_{r+s-l-1} \right] + \nonumber \\
\gamma_p^{d_p+1} \left[ f_p^2 - \sum_{r,s=1}^{d_p} \beta_{pr} \beta_{ps} \sigma_{r+s} \right] + 
\gamma_p^{d_p+2} \left[ \sigma_1 f_p^2 - 2 f_pf_p^+ + \sum_{r,s=1}^{d_p} \beta_{pr} \beta_{ps} \sigma_{r+s+1} \right] + \nonumber \\
\gamma_p^{d_p+3} \left[ f_p^2(\sigma_1 \sigma_1 - \sigma_2) - 2 f_p (\sigma_1 f_p^+ - f_p^{++}) + (f_p^+)^2 - \sum_{r,s=1}^{d_p} \beta_{pr} \beta_{ps} \sigma_{r+s+2} \right]
\end{gather}
by $x_{pq}^2 \sigma_{s-1}(\hat{x}_{pq}) \sigma_{r-1}(\hat{x}_{pq}) = [\sigma_s - \sigma_s(\hat{x}_{pq})] [\sigma_r - \sigma_r(\hat{x}_{pq})]$, which completes the proof.
\end{proof}

\begin{lemma}
For $d\geq1$ we have
\begin{align}\label{funny identity}
\sum_{\nu=1}^d
\prod_{\substack{q=1\\q\neq\nu}}^d(\frac{a}{b}+x_{pq})=
\sum_{j=0}^{d-1}
(d-j)
\left(\frac{a}{b}\right)^{d-j-1}\sigma_j
\end{align}
\end{lemma}
\begin{proof}
Taking $\partial_{x_{p\nu}}$-derivative of the identity
\begin{align}
\prod_{q=1}^d(\frac{a}{b}+x_{pq})=
\sum_{j=0}^d
\left(\frac{a}{b}\right)^{d-j}\sigma_j(x_{p1},\ldots,x_{pd_p})
\end{align}
we find
\begin{align}\label{funny identity to be summed}
\prod_{\substack{q=1\\q\neq\nu}}^d(\frac{a}{b}+x_{pq})=
\sum_{j=0}^{d-1}
\left(\frac{a}{b}\right)^{d-j-1}\sigma_j(\hat{x}_{p\nu}).
\end{align}
We note that
\begin{align}\label{funny sigmas identity}
\sum_{\nu=1}^d\sigma_j(\hat{x}_{p\nu})=
k\sigma_j
\hspace{.5cm}
\text{for some }
k\in\mathbb{R}
\end{align}
since the left hand side of \eqref{funny sigmas identity} is symmetric in $x_{p1},\ldots,x_{pd}$ and has degree $j$. To specify $k$ we restrict \eqref{funny sigmas identity} to the diagonal, i.e. $x_{p1}=\cdots=x_{pd}=x$, and obtain $k=d-j$. Thus summing \eqref{funny identity to be summed} over $\nu=1,\ldots,d$ we get \eqref{funny identity}.
\end{proof}

\bibliography{lib}
\bibliographystyle{abbrv}

\end{document}